\documentclass[aop]{imsart_arxiv}

\usepackage{amsmath, amsfonts, amsthm, verbatim,
	amssymb,dsfont,color}

\usepackage[utf8x]{inputenc}
\usepackage{mathtools}

\usepackage{tikz}
\usetikzlibrary{arrows.meta,arrows,decorations.pathreplacing}
\RequirePackage[colorlinks,citecolor=blue,urlcolor=blue]{hyperref}

\startlocaldefs

\newcommand{\dd}{\mathrm{d}}
\newcommand{\R}{\mathbb{R}}
\newcommand{\N}{\mathbb{N}}
\newcommand{\Z}{\mathbb{Z}}
\newcommand{\E}{\mathbf{E}}

\renewcommand{\P}{\mathbf{P}}

\newcommand{\bone}{\mathds 1}
\newcommand{\Dim}{\mathrm{Dim}}
\newcommand{\side}{\mathrm{side}}
\newcommand{\dist}{\mathrm{dist}}
\newcommand{\HH}{\mathrm{H}}
\newcommand{\MM}{\mathrm{M}}

\theoremstyle{plain}
\newtheorem{theorem}{Theorem}[section]
\newtheorem{lemma}[theorem]{Lemma}
\newtheorem{proposition}[theorem]{Proposition}
\newtheorem{corollary}[theorem]{Corollary}
\newtheorem{condition}{Condition}
\newtheorem{theoremA}{Theorem}

\makeatletter
\newcommand{\settheoremtag}[1]{
	\let\oldthecondition\thecondition
	\renewcommand{\thecondition}{#1}
	\g@addto@macro\endcondition{
		\addtocounter{condition}{-1}
		\global\let\thecondition\oldthecondition}
}
\makeatother
\theoremstyle{remark}
\newtheorem{remark}[theorem]{Remark}

\newcommand{\cals}{{\cal S}}

\newcommand{\calf}{{\cal F}}

\newcommand{\cale}{{\cal E}}

\newcommand{\calq}{{\cal Q}}

\newcommand{\al}{{\alpha}}
\newcommand{\la}{{\lambda}}
\newcommand{\La}{{\Lambda}}
\newcommand{\eps}{{\varepsilon}}
\newcommand{\ga}{{\gamma}}
\newcommand{\Ga}{{\Gamma}}
\newcommand{\vp}{{\varphi}}
\newcommand{\si}{{\sigma}}

\newcommand{\Om}{{\Omega}}

\newcommand{\ov}{\overline}

\newcommand{\wh}{\widehat}
\newcommand{\wt}{\widetilde}

\newcommand{\cc}{\mathrm{c}}

\newcommand{\Leb}{\mathrm{Leb}}

\newcommand{\bthm}{\begin{theorem}}
	\newcommand{\ethm}{\end{theorem}}

\newcommand{\bthmA}{\begin{theoremA}}
	\newcommand{\ethmA}{\end{theoremA}}

\newcommand{\bcor}{\begin{corollary}}
	\newcommand{\ecor}{\end{corollary}}

\newcommand{\blem}{\begin{lemma}}
	\newcommand{\elem}{\end{lemma}}

\newcommand{\bprop}{\begin{proposition}}
	\newcommand{\eprop}{\end{proposition}}

\newcommand{\bcond}{\begin{condition}}
	\newcommand{\econd}{\end{condition}}

\newcommand{\bdf}{\begin{definition}}
	\newcommand{\edf}{\end{definition}}

\newcommand{\bex}{\begin{example}}
	\newcommand{\eex}{\end{example}}

\newcommand{\brem}{\begin{remark}}
	\newcommand{\erem}{\end{remark}}

\newcommand{\bpr}{\begin{proof}}
	\newcommand{\epr}{\end{proof}}

\newcommand{\benu}{\begin{enumerate}}
	\newcommand{\eenu}{\end{enumerate}}

\newcommand{\beq}{\begin{equation}}
	\newcommand{\eeq}{\end{equation}}

\newcommand{\bit}{\begin{itemize}}
	\newcommand{\eit}{\end{itemize}}

\newcommand{\bff}{\textbf}

\numberwithin{equation}{section}

\allowdisplaybreaks[3]
\endlocaldefs

\begin{document}
	\begin{frontmatter}
		\title{A landscape of peaks: The intermittency islands of the  stochastic 
			heat equation with L\'evy noise}
		\runtitle{Intermittency islands of the SHE}
		
		\begin{aug}
			\author[A]{\fnms{Carsten} \snm{Chong}\ead[label=e1]{carsten.chong@columbia.edu}},
			\and
			\author[B]{\fnms{P\'eter} \snm{Kevei}\ead[label=e2]{kevei@math.u-szeged.hu}}
			
			\address[A]{Department of Statistics,
				Columbia University,
				\printead{e1}}
			
			\address[B]{Bolyai Institute,
				University of Szeged
				\printead{e2}}
		\end{aug}

\begin{abstract}
We show that the spatial profile of the solution to the stochastic heat equation features multiple layers of intermittency islands if the driving noise is non-Gaussian. On the one hand, as expected, if the noise is sufficiently heavy-tailed, the largest peaks of the solution will be taller under multiplicative than under additive noise. On the other hand, surprisingly, as soon as the noise has a finite moment of order $\frac2d$, where $d$ is the spatial dimension, the largest peaks will be of the same order for both additive and multiplicative noise, which is in sharp contrast to the behavior of the solution under Gaussian noise. However, in this case, a closer inspection reveals a second layer of peaks, beneath the largest peaks, that is exclusive to multiplicative noise and that can be observed by sampling the solution on the lattice. Finally, we compute the macroscopic Hausdorff and Minkowski dimensions of the intermittency islands of the solution. Under both additive and multiplicative noise, if it is not too heavy-tailed, the largest peaks will be self-similar in terms of their large-scale multifractal behavior. But under multiplicative noise, this type of self-similarity is not present in the peaks observed on the lattice.
\end{abstract}

	 
	 
 \end{frontmatter}

\tableofcontents

\section{Introduction}

Consider the stochastic heat equation (SHE)
\beq\label{eq:SHE}  \partial_t Y(t,x)=\frac{1}{2} \Delta Y (t,x)+ \si(Y(t,x))\dot \La(t,x),\qquad (t,x)\in(0,\infty)\times\R^d,  \eeq
driven by a  space-time white noise $\dot\La$, 
where  either  $\si(x)=1$ and $Y(0,x)=0$ (the case of \emph{additive noise}) or 
$\si(x)=x$ and $Y(0,x)=1$ (the case of \emph{multiplicative noise}). 
In this work, we fix $t>0$ and explore the macroscopic behavior of $Y$ as $\lvert 
x\rvert\to\infty$, where $\lvert\cdot\rvert$ denotes the Euclidean norm on $\R^d$. If 
$d=1$ and $\dot\La$ is Gaussian, it is well known that for  fixed $t>0$,
\beq\label{eq:Gauss} \begin{dcases} \limsup\limits_{\lvert x \rvert\to\infty} 
\dfrac{Y(t,x)}{(\log |x|)^{1/2}} = \biggl ( \frac{4t}{\pi} \biggr)^{\frac14} &\text{if 
} \si(x)=1,\\
\limsup\limits_{\lvert x \rvert\to\infty} \dfrac{\log Y(t,x)}{(\log |x|)^{2/3}} =  
\biggl( \frac{9t}{32}\biggr )^{\frac13} &\text{if } \si(x)=x\end{dcases}\eeq
almost surely; see \cite{Chen15,Conus12,Khoshnevisan17}. If $\dot\La$ follows a non-Gaussian distribution, in which case $\dot\La$  is called a \emph{Lévy noise}, 
the results we obtain are rather unexpected. To give a flavor of them, let us suppose in this introductory part that $\dot \La$ is a Lévy noise with Lévy measure 
\beq\label{eq:la} \la((-\infty,1])=0,\quad \la((z,\infty))= z^{-\al},\quad z>1, \eeq
for some $\al>0$. In this case, 
$\dot \La$ is a compound Poisson noise with Pareto-distributed weights.
If $\al$ is small, the noise is relatively heavy-tailed; if $\al$ is large, the noise is relatively light-tailed.
In the latter case, the analogous result to \eqref{eq:Gauss} reads as follows.
\bthmA \label{thm:A} Suppose that $\al>\frac 2d$. If $f: (0,\infty) \to (0,\infty)$ is 
nondecreasing, then for both additive and multiplicative noise almost surely,
\beq\label{eq:or}
\limsup_{x \to \infty} 
\frac{\sup_{\lvert y\rvert \leq x} Y(t,y)}{f(x)} = \infty
\qquad \text{or} \qquad 
\limsup_{x \to \infty} 
\frac{\sup_{\lvert y\rvert \leq x} Y(t,y)}{f(x)} = 0,
\eeq
according to whether the integral 
\beq\label{eq:int2}\int_1^\infty x^{d-1} f(x)^{-\frac 2d}\, \dd x\eeq diverges or converges.
\ethmA

An obvious difference to \eqref{eq:Gauss} is
the fact that the spatial asymptotics of the solution are governed by an integral test. But this is not the most surprising part about Theorem~\ref{thm:A}; a similar integral test has been found in \cite{CK2} for the behavior of $Y$ (with $\si(x)=1$) as $t\to\infty$. What is most striking in view of \eqref{eq:Gauss} is that the largest peaks of the solution  at a given time $t$ are of \emph{the same order} for both multiplicative and additive Lévy noise! It has been shown in \cite{Berger21b} that the solution to the SHE with multiplicative Lévy noise is always intermittent in all dimensions $d\geq1$, regardless of the details of $\dot\La$. While intermittency is an asymptotic concept that describes localization   on an exponential scale  as $t\to\infty$,  it is widely believed that the largest peaks of an intermittent process at finite times  should already exceed those of a non-intermittent process (e.g., the solution to \eqref{eq:SHE} with additive noise). Theorem~\ref{thm:A} shows that this belief is incorrect in general.

This being said, if $\dot \La$ is sufficiently heavy-tailed, multiplicative noise does produce higher peaks, even at finite times. 
 \bthmA\label{thm:B} Suppose that $\al<\frac 2d$ and let $\theta_\al=1-\frac d2(\al-1)$. 
\benu
\item[(i)] In the case of additive noise, we have the two possibilities in \eqref{eq:or} depending on whether the following integral diverges or converges:
\beq\label{eq:int3}\int_1^\infty x^{d-1} f(x)^{-\al}\, \dd x.\eeq 
\item[(ii)] In the case of multiplicative noise,    there are $0<L_\ast\leq L^\ast<\infty$ such that for all $L>L^\ast$,
\beq\label{eq:ub1} \limsup_{  x \to\infty} \frac{\sup_{\lvert y\rvert\leq x} Y(t,y)}{x^{d/\al} e^{L(\log x)^{1/(1+\theta_\al)}}} =0\qquad\text{a.s.}, \eeq
while for all $L<L_\ast$,
\beq\label{eq:lb1} \limsup_{  x \to\infty} \frac{\sup_{\lvert y\rvert\leq x} Y(t,y)}{x^{d/\al} e^{L(\log x)^{1/(1+\theta_\al)}}} =\infty\qquad\text{a.s.} \eeq
\eenu
\ethmA 
So the gain in the multiplicative case is a factor  roughly of order  $e^{L(\log x)^{1/(1+\theta_\al)}}$.
It follows from the previous two theorems  that the highest peaks in the solution to the SHE are taller for multiplicative than for additive noise if and only if $\dot\La$ has very heavy tails. If $\al>\frac2d$, Theorem~\ref{thm:A} seems to suggest that whether \eqref{eq:SHE} is subjected to multiplicative and additive noise cannot be distinguished based on the macroscopic behavior of the solution at a given time point. But this turns out to be false, too: in fact, there is a second layer of peaks, beneath the largest peaks studied in Theorem~\ref{thm:A}, that is exclusive to the solution under multiplicative noise. This second layer of peaks can be observed, for example, by sampling on the discrete lattice $\Z^d$ instead of $\R^d$.
\bthmA\label{thm:C} Consider $Y$ on the lattice $\Z^d$.
\benu
	\item[(i)] If $\al>\frac2d$ and $\si(x)=1$,  then almost surely, for any nondecreasing $f: (0,\infty) \to (0,\infty)$,
	\[
	\limsup_{x \to \infty} 
	\frac{\sup_{y\in\Z^d,\lvert y\rvert \leq x} Y(t,y)}{f(x)} = \infty
	\qquad \text{or} \qquad 
	\limsup_{x \to \infty} 
	\frac{\sup_{y\in \Z^d,\lvert y\rvert \leq x} Y(t,y)}{f(x)} = 0,
	\]
	according to whether the integral \eqref{eq:int2} diverges or converges.
	\item[(ii)] If $\al\in(\frac2d,1+\frac2d)$ and $\si(x)=x$, then  there are $0<M_\ast\leq M^\ast<\infty$ such that for  all $M>M^\ast$,
	\beq\label{eq:ub2}  \limsup_{  x \to\infty} \frac{\sup_{y\in\Z^d, \lvert y\rvert\leq x} Y(t,y)}{x^{d/\al} e^{M(\log x)^{1/(1+\theta_\al)}}} =0\qquad\text{a.s.}, \eeq
	while for   $M<M_\ast$,
	\beq\label{eq:lb2} \limsup_{  x \to\infty} \frac{\sup_{y\in\Z^d,\lvert y\rvert\leq x} Y(t,y)}{x^{d/\al} e^{M(\log x)^{1/(1+\theta_\al)}}} =\infty\qquad\text{a.s.} \eeq
\item[(iii)]
If $\al\geq1+\frac2d$ and $\si(x)=x$, then there are $0<M_\ast\leq M^\ast<\infty$ such that for   $M>M^\ast$,
	\beq\label{eq:ub3} \limsup_{  x \to\infty} \frac{\sup_{y\in\Z^d, \lvert y\rvert\leq x} Y(t,y)}{x^{d^2/(2+d)} e^{M(\log x)(\log\log\log x)/\log\log x}} =0\qquad\text{a.s.} \eeq
	while for   $M<M_\ast$,
	\beq\label{eq:lb3} \limsup_{  x \to\infty} \frac{\sup_{y\in\Z^d, \lvert y\rvert\leq x} Y(t,y)}{x^{d^2/(2+d)} e^{M(\log x)(\log\log\log x)/\log\log x}} =\infty\qquad\text{a.s.} \eeq
	\item[(iv)] If $\al<\frac2d$, then the statements of Theorem~\ref{thm:B} remain valid for $\sup_{y\in\Z^d,\lvert y\rvert\leq x} Y(t,y)$.
\eenu
\ethmA

The behavior described in part (iii) of Theorem~\ref{thm:C} is particularly interesting. We do not know of any other natural model with this type of growth asymptotics. 

\subsection{Review of literature}

Let us put Theorems~\ref{thm:A}--\ref{thm:C} in  the context of the existing literature. 
Until a few years ago, the majority of works on the SHE driven by non-Gaussian Lévy noise  
 focused on existence, uniqueness and regularity of  solutions, usually assuming  
strong moment assumptions on the noise or studying specific noises only (e.g., 
$\al$-stable noise); see   
\cite{Balan14,Chong,Chong1,CDH,Kosmala21,Mueller98,Mytnik02,Peszat07, SLB98}. More 
recently, building on \cite{Berger21a},  the paper \cite{Berger21b} derived the most 
general existence and uniqueness conditions known up to date for the SHE with 
multiplicative Lévy noise (which are necessary and sufficient for $d=1,2$ and almost 
optimal for $d\geq3$). Furthermore, extending \cite{CK1}, it was further shown in 
\cite{Berger21b}  that the solution to the SHE with multiplicative Lévy noise is 
\emph{strongly intermittent} in all dimensions for all non-trivial Lévy noises. (This is 
another unexpected feature of \eqref{eq:SHE}, because if $d\geq3$, intermittency does not 
always occur if one considers the SHE on a lattice \cite{Ahn92b,Ahn92,Carmona94} or the 
SHE with short-range correlated Gaussian noise \cite{Chen19,Lacoin11}.) 
For the SHE with additive Lévy noise, the authors showed in \cite{CK2} that as 
$t\to\infty$, $Y(t,x)$ for fixed $x$ satisfies a weak but violates a strong law of large 
numbers, a property  referred to as \emph{additive intermittency}. Finally, by showing 
that directed polymers in  heavy-tailed environments have the SHE with Lévy noise as a 
scaling limit in the intermediate disorder regime, \cite{Berger21} established a first 
discrete statistical mechanics model that rescales to a Lévy-driven SHE in continuous 
space and time (the analogous result for convergence to the SHE with Gaussian noise was 
shown in \cite{Alberts14}).  For results on the stochastic  wave equation with Lévy noise, we refer 
to \cite{Balan21, Balan16}.
                     
\subsection{Overview of the remaining paper}

After a rigorous introduction to the SHE with Lévy noise in Section~\ref{sec:prelim},  we state and prove in Section~\ref{sec:tail} tight upper and lower bounds on the probability tails of the solution $Y(t,x)$ and of its local spatial supremum $\sup_{x\in Q} Y(t,x)$, where $Q\in\calq$ and $\calq=\{x+(0,1)^d: x\in\R^d\}$ is the collection of all unit cubes in $\R^d$. Theorems~\ref{thm:sol-ht} and \ref{thm:sup-ht} cover the results when the tail of the noise is  heaviest, while Theorems~\ref{thm:sol-lt} and \ref{thm:sup-lt} contain the statement when the tail is lighter. These tail bounds are the main technical achievements of the paper; we will give more background on the approach we take to prove them in Section~\ref{sec:tail}. 
In Section~\ref{sec:peaks}, we will then use these tail bounds to prove Theorems~\ref{thm:peak-R-lt}--\ref{thm:peak-Z}, which extend Theorems~\ref{thm:A}--\ref{thm:C} to general Lévy noises. In Section~\ref{sec:dim}, we further show how the tail bounds of Section~\ref{sec:tail} can be used to determine/bound the macroscopic Hausdorff and Minkowski dimensions of the peaks of $Y$. For the SHE with Gaussian noise, this program has been carried out in \cite{Khoshnevisan17}. In the Lévy setting, multiple scales appear: in the case of additive noise, or in the case of multiplicative noise if the noise is not too heavy-tailed, we show in Theorems~\ref{thm:dim} and \ref{thm:dim2} that the largest peaks of the solution are not only multifractal in the sense of \cite{Khoshnevisan17} but actually \emph{self-similar} in terms of their multifractal behavior. At the same time, in the multiplicative case, the largest peaks  under a very heavy-tailed noise or the peaks observed on the lattice $\Z^d$ for any Lévy noise are multifractal but not self-similar. This is a consequence of Theorems~\ref{thm:dim} and~\ref{thm:dim3}. Finally, the \hyperref[appn]{Appendix} contains some technical results needed in the proofs. 

Except for Section~\ref{sec:dim} where we treat both additive and multiplicative noise, we only consider and prove results for the case of multiplicative Lévy noise in Sections~\ref{sec:tail} and \ref{sec:peaks}. In the case of additive noise, we can obtain exact tail asymptotics using the theory of regular variation, which is why we have deferred them to a companion paper \cite{islands_add}. In particular,  the parts of Theorem~\ref{thm:A}--\ref{thm:C} concerning additive noise also follow from \cite{islands_add}.

In what follows, we use $C$, $C_1$, $C_2$ etc.\ to denote constants which do not depend on any important parameters and whose values may change from line to line. Furthermore, if we plug in a real number $x$ for an integer-valued index (e.g., $\sum_{i=1}^x$ or $Y^{(x)}$ if $Y^{(n)}$ is a sequence indexed by $n\in\N$), we always mean plugging in $\lfloor x\rfloor$, the integer-part of $x$.

\section{Preliminaries}\label{sec:prelim}

Throughout this paper, we assume that $\dot\La$ is a space-time white noise on $\R^{1+d}$, that is, a stationary random generalized function that gives independent values when applied to test functions of disjoint support. It is well known (see \cite[Ch.\ 4.4]{Gelfand64} and \cite{Rajput89}) that $\dot \La$ is infinitely divisible in this case with Lévy--It\^o decomposition
\beq\label{eq:La}\begin{split} \La(\dd t,\dd x) &= b\,\dd t\,\dd x+ v \, W(\dd t,\dd x) + \int_{(-1,1)} z \,(\mu-\nu)(\dd t,\dd x,\dd z)\\
	&\quad + \int_{(-1,1)^\cc} z \, \mu(\dd t,\dd x,\dd z),\end{split}\eeq
where $b\in\R$, $v\in[0,\infty)$, $\dot W$ is a Gaussian space-time white noise, and $\mu$ is a Poisson random measure on $\R^{1+d}$ with intensity measure $\nu=\dd t\otimes\dd x \otimes \la(\dd z)$, where $\la$, the \emph{Lévy measure} of $\dot \La$, satisfies $\int_\R (1\wedge z^2)\,\la(\dd z)<\infty$. The last two terms in \eqref{eq:La} will be denoted by $\La_{<}(\dd t,\dd x)$ and $\La_\geq(\dd t,\dd x)$, respectively. In this paper, we assume 
\beq\label{eq:bv0} b=0\qquad\text{and}\qquad v=0.\eeq
The first assumption is no restriction, because $b\neq0$ would only change the solution 
$Y$ to \eqref{eq:SHE} by an additive or multiplicative 
constant, depending on whether $\si(x)=1$ or $\si(x)=x$ (cf.\ \cite[Sect.~3.3]{CK1}).  Regarding the second 
assumption, note  that \eqref{eq:SHE} has a mild solution for $v\neq0$ only if $d=1$, in 
which case \eqref{eq:Gauss} suggests---and one can modify the proofs in this paper to show 
this rigorously---that the macroscopic behavior of $x\mapsto Y(t,x)$ for fixed $t$ is 
dominated by the jump part. In addition to \eqref{eq:bv0}, we   further assume that 
$\dot\La$ is   spectrally positive, that is, 
\beq\label{eq:pos} \la((-\infty,0))=0. \eeq
The  condition \eqref{eq:pos} is needed to guarantee positivity of the solution $Y$ in the case of multiplicative noise \cite[Thm.~2.1]{Berger21b}, which is crucial for the lower bound proofs in this paper. In principle, all upper bound results remain valid if we consider signed noise, but we refrain from adding this extra bit of generality to keep the exposition simple.

From now on until the end of Section~\ref{sec:peaks}, we only consider the case of multiplicative noise, that is, we will assume
\beq\label{eq:mult} \si(x)=x.  \eeq
In this case, a predictable process $Y(t,x)$ is called a \emph{mild solution} to \eqref{eq:SHE} if for all $(t,x)\in(0,\infty)\times\R^d$,
\beq\label{eq:PAM} Y(t,x)=1+\int_0^t \int_{\R^d} g(t-s,x-y)Y(s,y)\,\La(\dd s,\dd y)\qquad\text{a.s.}, \eeq
where
$
g(t,x) =  (2 \pi  t)^{-d/2} e^{- {|x|^2}/(2  t)}\bone_{\{t>0\}}$
is the heat kernel in dimension $d$. In Section~\ref{sec:dim}, where we  consider both additive and multiplicative noise again, we will use the notation
\beq\label{eq:add} Y_+(t,x)=\int_0^t\int_{\R^d} g(t-s,x-y)\,\La(\dd s,\dd y)\eeq
for the solution to \eqref{eq:SHE} with additive Lévy noise.

Let us introduce the following truncated moments of the Lévy measure $\la$:
\beq\label{eq:mu} \mu_p(\la)=\int_{(0,\infty)} z^p\,\la(\dd z), \quad m_p(\la)=\int_{(0,1)} z^p\,\la(\dd z),\quad M_p(\la)=\int_{[1,\infty)} z^p\,\la(\dd z)  \eeq
and 
\beq\label{eq:mu-log} m^{\log}_p(\la)=\int_{(0,1)} z^p\lvert \log z\rvert\,\la(\dd z),\qquad m^{(\log)}_p(\la) = \int_{(0,1)} z^p\lvert \log z\rvert^{\bone_{\{p=1\}}}\,\la(\dd z). \eeq
Under the assumption
\beq\label{eq:ass}  \int_{[1,\infty)} (\log   z )^{\frac d2} \,\la(\dd z) + m^{\log}_{1+2/d}(\la)\bone_{\{d\geq2\}}<\infty, \eeq
it was shown in \cite{Berger21b} (see Thm.\ 2.5, Rem.\ 2.6 and the discussion in Sect.\ 3.3) that 
\beq\label{eq:Yst} u(s,y;t,x) = g(t-s,x-y)+ \sum_{N=1}^\infty \int_{((s,t)\times\R^d)^N}  \prod_{i=1}^{N+1} g(\Delta t_i,\Delta x_i)\prod_{j=1}^N\La(\dd t_j,\dd x_j)\eeq
is well defined and finite,
where $\Delta t_i=t_i-t_{i-1}$, $\Delta x_i=x_i-x_{i-1}$, $(t_{N+1},x_{N+1})=(t,x)$ and $(t_0,x_0)=(s,y)$. Furthermore, $u(s,y;\cdot,\cdot)$ is a mild solution to \eqref{eq:SHE}
on $(s,\infty)\times\R^d$ with $\si(x)=x$ and initial condition $u(s,y;s,\cdot)=\delta_y$. In what follows, we write  $Y_<(t,x)$ and $u_<(s,y;t,x)$ for the process obtained by substituting $\La_<$  for $\La$ in  \eqref{eq:PAM} and  \eqref{eq:Yst}, respectively. We let  $u_<(s,y;t,x)=0$ whenever $s\geq t$ and $Y_<(t,x)=0$ whenever $t< 0$. Then, similarly to \cite[Eq.\ (8.4)]{Berger21b},  a mild solution to \eqref{eq:SHE} is given by
\beq\label{eq:separate} Y(t,x)=\sum_{N=0}^\infty \int_{((0,t)\times\R^d)^N} Y_<(t_1,x_1) \prod_{i=2}^{N+1} u_<(t_{i-1},x_{i-1};t_i,x_i)\prod_{j=1}^N\La_\geq(\dd t_j,\dd x_j), \eeq
where the term for $N=0$ is $Y_<(t,x)$.
By the independence properties of $\La$, for any fixed $t_1<\dots<t_N$, we have that  $Y_<(t_1,x_1)$, $u_<(t_1,\cdot;t_2,\cdot),\dots,u_<(t_N,\cdot;t,\cdot)$ are independent of each other and also independent of $\La_\geq$. Note that \eqref{eq:ass} is necessary and sufficient for the existence of solutions to \eqref{eq:SHE} in dimensions $d=1,2$ and close to optimal in dimensions $d\geq3$  \cite{Berger21b}.

\section{Tail bounds on the solution and its local supremum}\label{sec:tail}

The main device to obtain Theorems~\ref{thm:A}--\ref{thm:C} (and their generalizations) are sharp probability tail bounds on $Y(t,x)$ and $\sup_{x\in Q} Y(t,x)$, where $Q$ is a unit cube in $\R^d$. In all results, we need to distinguish between a heavy-tailed and a light-tailed scenario, which motivates the following definitions depending on a parameter $\al$:

 \settheoremtag{(H-$\al$)}
\bcond \label{cond:ht}We have \eqref{eq:bv0} and \eqref{eq:pos}. Moreover, we have $m_{1+2/d}^{\log}(\la)<\infty$ and  $\la([R,\infty))\sim C  R^{-\al}$  for some   $C\in(0,\infty)$ as $R \to\infty$.
\econd

 \settheoremtag{(L-$\al$)}
\bcond \label{cond:lt} We have \eqref{eq:bv0} and \eqref{eq:pos}. Moreover, we have $0<m_{1+2/d}^{\log}(\la)+M_{\al}(\la)<\infty$.
\econd

Note that the notion of heavy- versus light-tailed is relative (and $\al$-dependent). In particular,   Condition~\ref{cond:lt} really only means that $\La$  has a finite moment of order $\al$ (in which case $\La$ may still be  heavy-tailed in the classical sense). In the following, we are going to prove tail bounds for two different processes (the solution and its local supremum), for each of which there will be a heavy-tailed case (Theorems~\ref{thm:sol-ht} and \ref{thm:sup-ht}) and a light-tailed case (Theorems~\ref{thm:sol-lt} and \ref{thm:sup-lt}). Each result in turn will involve an upper and a lower bound. Let us provide a short overview of the proof techniques:
\bit
\item All upper bounds, except for the tail of the local supremum of $Y$ in the light-tailed case (Theorem~\ref{thm:sup-lt}), are obtained by combining Markov's inequality with sharp moment estimates and then optimizing the exponent.
\item The upper bound in Theorem~\ref{thm:sup-lt} \emph{cannot} be obtained in this way. Instead, we first show that only ``large close'' jumps (in a certain sense) contribute to the tail and then use the explicit Poisson structure of the atoms to bound their tail behavior. For this part, we also use a decoupling inequality for tail probabilities (Lemma~\ref{lem:dec}) that is of independent interest.
\item For the lower bounds, the level of difficulty is reversed: for the supremum in the light-tailed case (Theorem~\ref{thm:sup-lt}), a  single (well-chosen) jump suffices to produce the tail. 
\item In all other cases, the main strategy is to find chains of close atoms of beneficial length $N$. An optimal number $N$ has to be sufficiently large (to be able to produce a tall peak) but at the same time not too large (such that the probability of having a chain of that length is not too small). It turns out that in the heavy-tailed case, for both the solution (Theorem~\ref{thm:sol-ht}) and the supremum (Theorem~\ref{thm:sup-ht}), one needs to consider a whole range of lengths $N$, while for the solution in the light-tailed case  (Theorem~\ref{thm:sol-lt}), considering a single length $N$ (depending on the size of the desirable peak, of course) is enough. An important observation is that for these lower bound proofs, it is crucial that we consider \eqref{eq:SHE} on an unbounded domain. The chains of atoms that lead to a tail event have to stretch arbitrarily far into space; on a bounded domain, the tail asymptotics of the solution would be different; see Remark~\ref{rem:unbounded}.
\eit

\subsection{Tail bounds for the solution} 

Let us  begin with   heavy-tailed noise.

\bthm\label{thm:sol-ht} Assume Condition~\ref{cond:ht} for some $\al\in(0,1+\frac 2d)$. For every $t>0$, there are constants $C_1,C_2\in(0,\infty)$ such that for all $x\in \R^d$ and $R>1$,
\beq\label{eq:sol-ht} C_1R^{-\al}e^{C_1 (\log R)^{1/(1+\theta_\al)}}\leq \P(  Y(t,x)   >R)\leq  C_2R^{-\al}e^{C_2 (\log R)^{1/(1+\theta_\al)}}.\eeq
\ethm

\bpr  \bff{Step 1: Upper bound} 

\smallskip
\noindent Let $\E_<$ and $\E_\geq$ denote  conditional expectation given $\La_\geq$ and 
$\La_<$, respectively. First suppose that $\al\in(0,1]$ and let $p\in(0,\al)$. Because 
$\La_\geq$ is a discrete measure, we can use the elementary inequality $\lvert \sum 
a_i\rvert^p\leq \sum \lvert a_i\rvert^p$, \eqref{eq:separate}, and the fact that 
$\E[X]=\E[\E_\geq[X]]$ to obtain 
\begin{align*} \E[  Y(t,x) ^p]&\leq \mathtoolsset{multlined-width=0.8\displaywidth}\begin{multlined}[t] \E[Y_<(t,x)^p]+\E\Biggl[\sum_{N=1}^\infty\int_{((0,t)\times\R^d\times [1,\infty))^N} Y_<(t_1,x_1)^p \\
	\times \prod_{i=2}^{N+1}   u_<(t_{i-1},x_{i-1};t_i,x_i) ^p\prod_{j=1}^N   z_j ^p \,\dd t_j\,\dd x_j\,\la(\dd z_j)\Biggr]\end{multlined}\\
&=\mathtoolsset{multlined-width=0.8\displaywidth}\begin{multlined}[t] \E[Y_<(t,x)^p]+ \sum_{N=1}^\infty M_p(\la)^N \int_{((0,t)\times\R^d)^N} \E [Y_<(t_1,x_1)^p ]\\
\times
\prod_{i=2}^{N+1} \E[  u_<(t_{i-1},x_{i-1};t_i,x_i) ^p]\prod_{j=1}^N  \,\dd t_j\,\dd x_j,\end{multlined}
\end{align*}
where $M_p(\la)=\int_{[1,\infty)}   z ^p\,\la(\dd z)$.
By Jensen's inequality, 
\begin{equation} \label{eq:p<1aux}
\begin{split}
\E[Y_<(t,x)^p] & \leq \E[Y_<(t,x)]^p = 1,\\
\E[  u_<(t_{i-1},x_{i-1};t_i,x_i) ^p]& \leq \E[  u_<(t_{i-1},x_{i-1};t_i,x_i) 
]^p = g(\Delta t_i,\Delta x_i)^p.
\end{split}
\end{equation}
Thus, recalling that $\theta_p=1-(p-1)\frac d2$, we have 
\begin{equation}\label{eq:Yp}\begin{split} \E[  Y(t,x) ^p]&\leq  
	1+ \sum_{N=1}^\infty M_p(\la)^N \int_{((0,t)\times\R^d)^N}  \prod_{i=2}^{N+1} 
g(\Delta t_i,\Delta x_i)^p\prod_{j=1}^N  \,\dd t_j\,\dd x_j\\
	&= 	1+ \sum_{N=1}^\infty (CM_p(\la))^N \int_{(0,t)^N} \bone_{\{t_1<\dots<t_N\}} \prod_{i=2}^{N+1} (\Delta t_i)^{-(p-1)\frac d2}\prod_{j=1}^N  \,\dd t_j\\
	&= \sum_{N=0}^\infty \frac{(CM_p(\la)\Ga(\theta_p)t^{\theta_p})^N}{\Ga(1+\theta_pN)}. 
\end{split}
\end{equation}
The first equality follows by noting that $g$ is a Gaussian density, while 
the second equality follows from \cite[Lemma~3.5]{Chong1}. By Lemma~\ref{lem:Stirling} and 
the fact that $0<\theta_p\leq 1+\frac d2$, we conclude that 
\beq\label{eq:mom-ub} \E[  Y(t,x) ^p] \leq C e^{(CM_p(\la)\Ga(\theta_p))^{1/\theta_p}t}\eeq
for some constant that does not depend on $p$.

The following tail bound is now an immediate consequence of Markov's inequality:
$$ \P(  Y(t,x)  >R)\leq CR^{-p} e^{(CM_p(\la)\Ga(\theta_p))^{1/\theta_p}t}\leq CR^{-p} e^{(CM_p(\la)\Ga(\theta_p))^{1/\theta_\al}t}$$
for all $p\in(0,\al)$. Under the tail assumption on $\la$, we have that $M_p(\la)\sim C(\al-p)^{-1}$. Inserting this expression into the previous line and choosing $p=\al-(\theta_\al\log R)^{-1/(1+1/\theta_\al)}$, we obtain that
$$ \P(  Y(t,x)  >R)\leq CR^{-\al}e^{C_\al(\log R)^{1-1/(1+1/\theta_\al)}} e^{C_{\al,t}(\log R)^{1/(\theta_\al+1)}}=CR^{-\al}e^{C_{\al,t}(\log R)^{1/(1+\theta_\al)}},$$
which completes the proof if $\al\in(0,1]$. 

If $\al\in(1,1+\frac 2d)$, note that $Y(t,x)=e^{m t}\ov Y(t,x)$ where $m=\int_{(-1,1)^c} z\,\la(\dd z)$ is the mean of $\La$ and $\ov Y(t,x)$ is the solution to \eqref{eq:PAM} when $\La$ is replaced by $\ov \La=\La - m\,\Leb$. Similarly to \eqref{eq:separate}, and with obvious notation, we have that 
\beq\label{eq:separate2} \ov Y(t,x)= \sum_{N=0}^\infty 
\int_{((0,t)\times\R^d)^N}  \ov Y_<(t_1,x_1) \prod_{i=2}^{N+1} \ov 
u_<(t_{i-1},x_{i-1};t_i,x_i)\prod_{j=1}^N\ov\La_\geq(\dd t_j,\dd x_j), \eeq
where the zeroth-order term in $\ov Y_<(t,x)$.
Thus, in dimensions $d\geq2$, using   the Burkholder--Davis--Gundy (BDG) inequality, we have for all $p\in(1,\al)$ that
\beq\label{eq:Ybarp}\begin{split}
\E[  \ov Y(t,x) ^p]^{\frac1p}&\leq  \sum_{N=0}^\infty (CM_p(\la)^{\frac 1p})^N \\
& \times \Biggl(\int_{((0,t)\times\R^d)^N} \E[\ov Y_<(t_1,x_1)^p] \prod_{i=2}^{N+1} \E[ \ov u_<(t_{i-1},x_{i-1};t_i,x_i) ^p]\prod_{j=1}^N \dd t_j\,\dd x_j\Biggr)^{\frac 1p},\raisetag{-3\baselineskip}
\end{split}
\eeq
where $C\in(0,\infty)$ is a constant that can be chosen uniformly for all $p$   close 
enough to $\al$ ($C$ may depend on $\al$). 
By \cite[Cor.\ 6.5]{Berger21b} (combined with Minkowski's integral inequality
together with (1.17) in \cite{Berger21b})
and its proof as well as Lemma~\ref{lem:Stirling}, we have that 
\begin{equation}\label{eq:ubar}\begin{split}  \E[ \ov u_<(t_{i-1},x_{i-1};t_i,x_i) ^p]^{\frac 1p} &\leq  C\,\Ga(\theta_p)^{\frac 1p}\Biggl(\sum_{k=0}^\infty\frac{(C\Ga(\frac{\theta_p}3))^{\frac kp}}{\Ga(\frac{\theta_p}3 k+\theta_p)^{\frac 1p}}\Biggr)g(\Delta t_i,\Delta x_i)\\
&\leq C\frac{\Ga(\theta_p)^{\frac 1p}}{\theta_p}e^{C\Ga(\frac{\theta_p}3)^{3/\theta_p}}g(\Delta t_i,\Delta x_i),\\
\E[ \ov Y_<(t,x) ^p]^{\frac 1p} &\leq  C\,\Ga(\theta_p)^{\frac 1p}\Biggl(\sum_{k=0}^\infty\frac{(C\Ga(\frac{\theta_p}3))^{\frac kp}}{\Ga(\frac{\theta_p}3 k+\theta_p)^{\frac 1p}}\Biggr) \leq C\frac{\Ga(\theta_p)^{\frac 1p}}{\theta_p}e^{C\Ga(\frac{\theta_p}3)^{3/\theta_p}}.\end{split}\raisetag{-3.5\baselineskip}\end{equation}
As $p<\al<1+\frac2d$, $\theta_p$ is bounded away from $0$ and
by \cite[Cor.\ 6.5]{Berger21b} we simply obtain
$$  \E[ \ov u_<(t_{i-1},x_{i-1};t_i,x_i) ^p] \leq Cg(\Delta t_i,\Delta x_i)^p\qquad \text{and}\qquad\E[\ov Y_<(t,x)^p]\leq C.$$
Inserting this into \eqref{eq:Ybarp}, we deduce the bound
\beq\label{eq:Yp2}
\E[  \ov Y(t,x) ^p]^{\frac1p}\leq  \sum_{N=0}^\infty (CM_p(\la)^{\frac 1p})^N  \Biggl(\int_{((0,t)\times\R^d)^N}  \prod_{i=2}^{N+1} g(\Delta t_i,\Delta x_i)^p\prod_{j=1}^N \dd t_j\,\dd x_j\Biggr)^{\frac 1p}
\eeq
Comparing with the estimate in \eqref{eq:Yp},  we can conclude by a   similar   argument.

Finally, if $d=1$ and $\al\in(1,2]$, one only needs to replace the bound in \eqref{eq:ubar} by
\beq\label{eq:ubar2}\begin{split} \E[ \ov u_<(t_{i-1},x_{i-1};t_i,x_i) ^p]^{\frac 1p}&\leq C\,\Ga(\theta_p)^{\frac 1p}\Biggl(\sum_{k=0}^\infty\frac{(C\Ga(\frac{\theta_p}3))^{\frac kp}}{\Ga( \theta_p k+\theta_p)^{\frac 1p}}\Biggr)g(\Delta t_i,\Delta x_i) \leq Cg(\Delta t_i,\Delta x_i),\\
	\E[ \ov Y_<(t,x) ^p]^{\frac 1p}&\leq C\,\Ga(\theta_p)^{\frac 1p}\Biggl(\sum_{k=0}^\infty\frac{(C\Ga(\frac{\theta_p}3))^{\frac kp}}{\Ga( \theta_p k+\theta_p)^{\frac 1p}}\Biggr)  \leq C,
\end{split}\raisetag{-2.5\baselineskip}\eeq
which also follow from the proof of \cite[Cor.\ 6.5]{Berger21b}.
If $d=1$, $\al\in (2,3)$ and $p\in(2,\al)$, the bounds in \cite[Prop.~6.1]{Berger21} do 
not yield optimal tail estimates, which is why we need to use a different approach. Since 
$\E[\ov Y(t,x)^p]=\E[\ov Y(t,0)^p]$ for all $x\in\R$ by stationarity, we can    use  
\cite[Thm.~1]{MR} (with $\alpha = 2$) and Minkowski's integral inequality
to show that
\begin{align*}
\E[\ov Y(t,0)^p]^{\frac1p}&\leq \mathtoolsset{multlined-width=0.8\displaywidth} 
\begin{multlined}[t] C\Biggl( 1 + \Biggl(\mu_2(\la)\int_0^t\int_\R 
g(t-s,x-y)^2\E[\ov Y(s,y)^p]^{\frac2p}\,\dd s\,\dd y\Biggr)^{\frac12}\\
+\Biggl(\mu_p(\la)\int_0^t\int_\R g(t-s,x-y)^p\E[\ov Y(s,y)^p]\,\dd s\,\dd 
y\Biggr)^{\frac 1p}\Biggr)\end{multlined}\\
&\leq  \mathtoolsset{multlined-width=0.8\displaywidth} \begin{multlined}[t] C\Biggl( 1 + 
\Biggl(\mu_2(\la)\int_0^t (t-s)^{-\frac12}\E[\ov Y(s,0)^p]^{\frac2p}\,\dd s 
\Biggr)^{\frac12}\\ +\Biggl(\mu_p(\la)\int_0^t  (t-s)^{-\frac{p-1}2} \E[\ov 
Y(s,0)^p]\,\dd s \Biggr)^{\frac 1p}\Biggr). \end{multlined}
\end{align*}
We can absorb $\mu_2(\la)$ into the constant $C$. Moreover, by H\"older's inequality (with respect to the measure $(t-s)^{-1/2}\,\dd s$), 
$$ \Biggl( \int_0^t (t-s)^{-\frac12} \E[\ov Y(s,0)^p]^{\frac 2p} \,\dd s\Biggr)^{\frac12} 
\leq\Biggl(\int_0^t (t-s)^{-\frac12} \E[\ov Y(s,0)^p]\,\dd 
s\Biggr)^{\frac1p}\Biggl(\int_0^t s^{-\frac 12}\,\dd s\Biggr)^{\frac12(1-\frac2p)}.$$
Since $s^{-\frac12}\leq C_t s^{-(p-1)/2}$ for $s\in(0,t]$, it follows that
\begin{align*}
\E[\ov Y(t,0)^p]\leq C\Biggl(1+\mu_p(\la)\int_0^t (t-s)^{-\frac 
{p-1}2}\E[\ov Y(s,0)^p]\,\dd s\Biggr)
\end{align*}
for some constant $C$ that may depend on $t$. Iterating this inequality and arguing as in \eqref{eq:Yp} and \eqref{eq:mom-ub}, we arrive at
\beq\label{eq:mom-p}\begin{split}\E[\ov Y(t,0)^p]&\leq C \sum_{N=0}^\infty (C\mu_p(\la))^N 
\int_{(0,t)^N} \bone_{\{t_1<\dots<t_N\}} \prod_{i=2}^{N+1} (\Delta t_i)^{-\frac 
{p-1}2}\prod_{j=1}^N  \,\dd t_j \\ &\leq C e^{(C\mu_p(\la)\Ga(\theta_p))^{1/\theta_p}t}.\end{split} \eeq
The proof can now be completed as in   the paragraph following \eqref{eq:mom-ub}.

\bigskip
\noindent\bff{Step 2: Lower bound} 
\smallskip

\noindent  
At this part it is convenient to treat the time on $(-\infty, \infty)$.
Let $(\tau_0,\eta_0)=(t,x)$ and 
$$\tau_i=\sup\{ u \in(0,\tau_{i-1})\ : \ \La_\geq(\{(s,y)\in[u,\tau_{i-1})\times\R^d :  
\lvert \eta_{i-1}-y\rvert\leq \sqrt{\tau_{i-1}-s}\})=1\}$$
for $i\in\N$, 
where 
$\eta_i$ is the spatial coordinate  of the atom associated to $\tau_i$. 
We denote the associated jump size by $\zeta_i$. Note that the numbering is reversed here, since we trace atoms backwards in time, starting at $(t,x)$. Clearly, 
if we write $\Delta \tau_i=\tau_{i-1}-\tau_{i}$ and $\Delta \eta_i=\eta_{i-1}-\eta_{i}$, the events
$$ A_N=\bigcap_{i=1}^N \{ \Delta \tau_i\leq \tfrac tN \}  \cap  \{   \Delta\tau_{N+1}>t  \},\qquad N\in\N,$$
are pairwise disjoint. Moreover, since $u_<$ is nonnegative in \eqref{eq:separate} (see \cite[Thm.\ 2.1]{Berger21b}), we have
\begin{align*}  Y(t,x)&\geq\int_{((0,t)\times\R^d)^N}  Y_<(t_1,x_1)\prod_{i=2}^{N+1} u_<(t_{i-1},x_{i-1};t_i,x_i)\prod_{j=1}^N\La_\geq(\dd t_j,\dd x_j)\\ 
	&\geq Y_<(\tau_N,\eta_N)\prod_{i=1}^{N} 
u_<(\tau_{i},\eta_{i};\tau_{i-1},\eta_{i-1})  \zeta_i\end{align*}
on the event $A_N$. Therefore,
\beq\label{eq:PY}\begin{split}
\P(Y(t,x)>R)&\geq \sum_{N=1}^\infty \P\Biggl( A_N \cap \Biggl\{ Y_<(\tau_N,\eta_N)\prod_{i=1}^{N} u_<(\tau_{i},\eta_{i};\tau_{i-1},\eta_{i-1})  \zeta_i >R \Biggr\}\Biggr)\\&=\sum_{N=1}^\infty \P(A_N) \P\Biggl( Y_<(\tau_N,\eta_N)\prod_{i=1}^{N} u_<(\tau_{i},\eta_{i};\tau_{i-1},\eta_{i-1})  \zeta_i>R \mathrel{\bigg|} A_N \Biggr).
\end{split}\raisetag{-3\baselineskip}\eeq
As $\La_\geq$ is a Poisson random measure,  $(\Delta\tau_i)_{i\in\N}$ is a sequence of independent and identically distributed variables with distribution function $1-e^{-C x^{1+d/2}}$, where $C=\pi^{d/2}/\Ga(\frac d2 +2)$. 
Thus,
\beq\label{eq:PAN} \P(A_N)=  e^{-C t^{1+d/2}}  (1-e^{-C (\frac tN)^{1+d/2}})^N\geq \frac{C^N}{N^{(1+\frac d2)N}}. \eeq

Next, we estimate the conditional probability in \eqref{eq:PY}. For simplicity, we write $\P_N=\P(\cdot\mid A_N)$, $\P^{\tau,\eta}$ for the conditional probability given the sequences $(\tau_i)_{i\in\N}$ and $(\eta_i)_{i\in\N}$ and $\P^{\tau,\eta}_<$ if we further condition on $\La_<$. Because the variables $\zeta_i$ are independent of $\La_<$, $\tau_i$ and $\eta_i$, Lemma~\ref{lem:JM} implies that
  \begin{align*}&\P_N\Biggl(Y_<(\tau_N,\eta_N)\prod_{i=1}^{N} u_<(\tau_{i},\eta_{i};\tau_{i-1},\eta_{i-1})  \zeta_i>R\Biggr)\\
&\qquad=\E_N\Biggl[ \P^{\tau,\eta}_<\Biggl(\prod_{i=1}^N \zeta_i>\frac{R}{Y_<(\tau_N,\eta_N)\prod_{i=1}^{N} u_<(\tau_{i},\eta_{i};\tau_{i-1},\eta_{i-1}) }\Biggr)\Biggr]\\
&\qquad \geq \mathtoolsset{multlined-width=0.8\displaywidth} \begin{multlined}[t]\frac{C^N}{(N-1)!} R^{-\al}\E_N\Biggl[Y_<^\al(\tau_N,\eta_N)\prod_{i=1}^{N} u^\al_<(\tau_{i},\eta_{i};\tau_{i-1},\eta_{i-1})\\
	\times\log^{N-1}\frac{R}{Y_<(\tau_N,\eta_N)\prod_{i=1}^{N} u_<(\tau_{i},\eta_{i};\tau_{i-1},\eta_{i-1}) }\Biggr].\end{multlined} \end{align*} 
Further restricting to the set $\{Y_<(\tau_N,\eta_N)\prod_{i=1}^{N} u_<(\tau_{i},\eta_{i};\tau_{i-1},\eta_{i-1}) \leq \sqrt{R}\}$, we obtain that
\beq\label{eq:plower}\begin{split} &\P_N\Biggl(Y_<(\tau_N,\eta_N)\prod_{i=1}^{N} u_<(\tau_{i},\eta_{i};\tau_{i-1},\eta_{i-1}) \zeta_i>R\Biggr)\\
&\qquad\geq \mathtoolsset{multlined-width=0.8\displaywidth} \begin{multlined}[t] \frac{(C\log R)^{N-1}}{(N-1)!} R^{-\al}\E_N\Biggl[Y_<^\al(\tau_N,\eta_N)\prod_{i=1}^{N} u^\al_<(\tau_{i},\eta_{i};\tau_{i-1},\eta_{i-1})\\
	\times\bone_{\{Y_<(\tau_N,\eta_N)\prod_{i=1}^{N} u_<(\tau_{i},\eta_{i};\tau_{i-1},\eta_{i-1})\leq \sqrt{R}\}}\Biggr].\end{multlined}\end{split}\eeq

The last line has the form $\E[X^\al\bone_{\{X\leq \sqrt{R}\}}]$, which can be bounded from below by
\beq\label{eq:exp}\begin{split}\E[X^\al\bone_{\{X\leq \sqrt{R}\}}]&=\E[X^\al]-\E[X^\al\bone_{\{X> \sqrt{R}\}}]\geq \E[X^\al]-\E[X^{\al p}]^{\frac 1p} \P(X>\sqrt{R})^{1-\frac1p}\\
&\geq \E[X^\al]-\E[X^{\al p}]R^{-\frac12\al (p-1)}\end{split}\eeq
thanks to H\"older's inequality and Markov's inequality.
Moreover, 
\beq\label{eq:help}\mathtoolsset{multlined-width=0.9\displaywidth} 
\begin{multlined}
\E_N\Biggl[Y_<^\al(\tau_N,\eta_N)\prod_{i=1}^{N} u^\al_<(\tau_{i},\eta_{i};\tau_{i-1},\eta_{i-1})\Biggr] \\= \E_N\Biggl[\E^{\tau,\eta}[Y_<^\al(\tau_N,\eta_N)]\prod_{i=1}^{N}\E^{\tau,\eta}\Biggl[ u^\al_<(\tau_{i},\eta_{i};\tau_{i-1},\eta_{i-1})\Biggr]\Biggr],\end{multlined}
\eeq
where we used the  independence of $Y_<(\tau_N,\eta_N)$ and the variables 
$u_<(\tau_{i},\eta_{i};\tau_{i-1},\eta_{i-1})$ as $i$ varies under $\P^{\tau,\eta}$.
Indeed, the sequence $\tau_i$ is determined by $\Lambda_{\geq}$, while $Y_<$ and $u_<$
are defined via $\Lambda_<$.

If $\al\in(0,1)$, we use Lemma~\ref{lem:PZ} (with some fixed $p\in(1,1+\frac 2d)$) and obtain
\begin{align*} \E^{\tau,\eta}\Bigl[ u^\al_<(\tau_{i},\eta_{i};\tau_{i-1},\eta_{i-1})\Bigr]&\geq C\frac{g(\Delta \tau_i,\Delta \eta_i)^{\al+\frac p{p-1}}}{\E^{\tau,\eta}[u^p_<(\tau_{i},\eta_{i};\tau_{i-1},\eta_{i-1})]^{\frac 1{p-1}}}\geq C g(\Delta \tau_i,\Delta \eta_i)^\al,\\
 \E^{\tau,\eta}[Y^\al_<(\tau_N,\eta_N)]&\geq C\frac{\E^{\tau,\eta}[Y_<(\tau_N,\eta_N)]^{\al+\frac{p}{p-1}}}{\E^{\tau,\eta}[Y^p_<(\tau_N,\eta_N)]^{\frac1{p-1}}}\geq C, \end{align*}
where the last step in both lines follows from \cite[Cor.\ 6.5]{Berger21b}. Thus, there is $C_1\in(0,\infty)$ such that
\beq\label{eq:lb}
\E_N\Biggl[Y_<^\al(\tau_N,\eta_N)\prod_{i=1}^{N} u^\al_<(\tau_{i},\eta_{i};\tau_{i-1},\eta_{i-1})\Biggr]\geq
C_1^N\E_N\Biggl[\prod_{i=1}^{N}g(\Delta \tau_i,\Delta \eta_i)^\al\Biggr].
\eeq
If $\al\geq1$, we can use Jensen's inequality in \eqref{eq:help} to take $\al$ outside of $\E^{\tau,\eta}$ and obtain \eqref{eq:lb} as well (with $C_1=1$),
since $\E^{\tau,\eta}[Y_<(\tau_N,\eta_N)]=1$ and $\E^{\tau,\eta}[u_<(\tau_{i},\eta_{i};\tau_{i-1},\eta_{i-1})]=g(\Delta \tau_i,\Delta \eta_i)$. 
At the same time, for any $\al\in(0,1+\frac2d)$, upon using Jensen's inequality if $\al\in(0,1]$ and \cite[Cor.~6.5]{Berger21b} if $\al\in(1,1+\frac 2d)$, we have the upper bound
\beq\label{eq:ub}\E_N\Biggl[Y_<^\al(\tau_N,\eta_N)\prod_{i=1}^{N} u^\al_<(\tau_{i},\eta_{i};\tau_{i-1},\eta_{i-1})\Biggr]\leq C^N_2
\E_N\Biggl[\prod_{i=1}^{N}g(\Delta \tau_i,\Delta \eta_i)^\al\Biggr].\eeq

Next, we evaluate $\E_N [\prod_{i=1}^{N}g(\Delta \tau_i,\Delta \eta_i)^\al]$. To this end, note that conditionally on $A_N$, the $\Delta\tau_i$'s are independent with density \beq\label{eq:density}f_N(x)=\frac{C(1+\tfrac d2)x^{\frac d2}e^{-Cx^{1+d/2}}}{1-e^{-C(\frac tN)^{1+d/2}}},\qquad x\in(0,\tfrac tN),\eeq
for all $i=1,\dots,N$,
while the $\Delta \eta_i$'s are independent and, conditioned on $\Delta \tau_i$'s are 
uniformly distributed on a centered ball with radius $\sqrt{\Delta \tau_i}$. Therefore,
\begin{align*}
	\E_N\Biggl[\prod_{i=1}^{N}g(\Delta \tau_i,\Delta \eta_i)^\al\Biggr] &= \E_N[g(\Delta \tau_1,\Delta \eta_1)^\al]^N\\
	&= \Biggl(\int_0^{t/N} \frac{f_N(s)}{\pi^{\frac d2} / \Gamma(\frac d2 +1)s^{\frac d2}}\int_{\R^d} g(s,y)^\al \bone_{\{\lvert y\rvert \leq \sqrt{s}\}}\,\dd y\,\dd s\Biggr)^N\\
	&\leq C^NN^{(1+\frac d2)N}\Biggl( \int_0^{t/N}\int_{\R^d} g(s,y)^\al\bone_{\{\lvert y\rvert \leq \sqrt{s}\}}\,\dd y\,\dd s \Biggr)^N\\
	&= C^NN^{(1+\frac d2)N}\Biggl(\int_0^{t/N} s^{-\frac d2(\al-1)}\,\dd s\Biggr)^N = \frac{C^NN^{(1+\frac d2)N}}{N^{\theta_\al N}}.
\end{align*}
In this calculation, we can replace $\leq$ by $\geq$ in the third line upon changing the value of $C$. As a consequence, if we combine this result with \eqref{eq:plower}, \eqref{eq:exp}, \eqref{eq:lb} and \eqref{eq:ub} (with  $\al p$ in the role of $\al$ and $p>1$ such that $\al p< 1+\frac 2d$), we obtain that 
\begin{align*}
	&\P_N\Biggl(Y_<(\tau_N,\eta_N)\prod_{i=1}^{N} u_<(\tau_{i},\eta_{i};\tau_{i-1},\eta_{i-1}) \zeta_i>R\Biggr)\\
	&\qquad\geq \frac{(C\log R)^{N-1} N^{(1+\frac d 2)N}}{N!}R^{-\al}\Biggl( 
\frac{C_1^N}{N^{\theta_\al N}}-\frac{C_2^N}{N^{\theta_{\al p} N}}R^{-\frac 12\al(p-1)} 
\Biggr)\\
	&\qquad\geq\frac{(C\log R)^{N-1} N^{(1+\frac d 
2)N}}{N^{(1+\theta_\al)N}}R^{-\al}\Biggl(1-\frac{(C_2/C_1)^NN^{\frac d2 \al 
(p-1)}}{R^{\frac 12\al(p-1)}} \Biggr)\\
	&\qquad\geq \frac{(C\log R)^{N-1} 
	N^{(1+\frac d 2)N}}{2N^{(1+\theta_\al)N}}R^{-\al},
\end{align*}
where the last step holds if $N\leq C_0\log R$ for some   small but fixed $C_0>0$. Together with \eqref{eq:PY} and \eqref{eq:PAN}, we have shown that 
\beq\label{eq:C0}
\P(Y(t,x)>R)\geq R^{-\al} (\log R)^{-1} \sum_{N=1}^{  C_0 \log R } \frac{(C\log 
R)^N}{N^{(1+\theta_\al)N}}.
\eeq

In order to bound this sum, we use integral approximations. Because the function $x\mapsto 
\vp(x)=  (C\log R)^x/x^{(1+\theta_\al)x}$ has a unique maximum at $x_0=(C\log 
R)^{1/(1+\theta_\al)}e^{-1}$ and 
$$ \frac{\vp(\lceil x_0\rceil)}{\vp(x_0)}\geq\frac{\vp(x_0+1)}{\vp(x_0)} = \frac{C\log R(1+(C\log R)^{1/(1+\theta_\al)}e^{-1})^{-1-\theta_\al}}{(1+1/[(C\log R)^{1/(1+\theta_\al)}e^{-1}])^{(1+\theta_\al)(C\log R)^{1/(1+\theta_\al)}e^{-1}}}\to 1$$
as $R\to\infty$, we have, for sufficiently large $R$,
\begin{align*}
	&\sum_{N=1}^{  C_0\log R } \frac{(C\log R)^N}{N^{(1+\theta_\al)N}} \geq \frac12 \int_0^{\lfloor C_0\log R\rfloor+1} \frac{(C\log R)^x}{x^{(1+\theta_\al)x}} \,\dd x\\ &\qquad  = \frac1{2(1+\theta_\al)} \int_0^{(1+\theta_\al)(\lfloor C_0 \log R\rfloor+1)} \frac{(C\log R)^{ y/(1+\theta_\al)}}{( y/(1+\theta_\al))^y}\,\dd y \\
	&\qquad\geq  \frac1{4(1+\theta_\al)} \sum_{N=0}^{  (1+\theta_\al)\lfloor C_0\log R\rfloor } \frac{[(C \log R)^{1/(1+\theta_\al)}(1+\theta_\al)]^N}{N^N} \\
	&\qquad\geq \frac1{4(1+\theta_\al)} \sum_{N=0}^{  (1+\theta_\al)\lfloor C_0\log R\rfloor } \frac{[(C \log R)^{1/(1+\theta_\al)}(1+\theta_\al)e^{-1}]^N}{N!}.
\end{align*}
By Taylor's theorem, this is further bounded from below by
\begin{align*}
	& \frac1{4(1+\theta_\al)}e^{(C \log R)^{1/(1+\theta_\al)}(1+\theta_\al)e^{-1}}\Biggl(1-\frac{[(C \log R)^{1/(1+\theta_\al)}(1+\theta_\al)e^{-1}]^{\lfloor (1+\theta_\al)\lfloor C_0\log R\rfloor\rfloor}}{\lfloor (1+\theta_\al)\lfloor C_0\log R\rfloor\rfloor!}\Biggr)\\
	&\qquad\geq \frac1{4(1+\theta_\al)}e^{(C \log R)^{1/(1+\theta_\al)}(1+\theta_\al)e^{-1}}\Biggl(1-\biggl(\frac{C^{C_0}C_0^{-1-\theta_\al}}{(\log R)^{\theta_\al}}\biggr)^{C_0\log R}\Biggr)\\
	&\qquad\geq\frac1{8(1+\theta_\al)}e^{(C \log R)^{1/(1+\theta_\al)}(1+\theta_\al)e^{-1}}.
\end{align*}
This completes the proof of the lower bound in  \eqref{eq:sol-ht}.
\epr

If the noise has lighter tails, a different slowly varying function appears in the tail.

\bthm\label{thm:sol-lt} Assume Condition~\ref{cond:ht} or \ref{cond:lt} with $\al=1+\frac 2d$. For every $t>0$, there is $C>0$ such that for all $x\in\R^d$ and $R>1$,
\beq\label{eq:sol-lt} 
 C_1R^{-1-\frac 2d}e^{C_1\frac{(\log R)(\log\log\log R)}{\log\log R}}
 	\leq \P(Y(t,x)>R)
	\leq C_2R^{-1-\frac 2d}e^{C_2\frac{(\log R)(\log\log\log R)}{\log\log R}}. 
	\eeq
\ethm
\bpr \bff{Step 1: Upper bound} 
\smallskip

\noindent  
First consider $d\geq2$, in which case $1+\frac2d\in(1,2]$. As in the upper bound proof of Theorem~\ref{thm:sol-ht}, it suffices to show the tail bound for $\ov Y(t,x)$.  By \cite[Prop.~6.3]{Berger21b} and our assumptions on $\la$, there are $\eta>0$ and $C>0$ such that for any $1<p<1+\frac 2d$, we have
\begin{align*}
\E[\ov Y(t,x)^p]^{\frac1p}\leq \sum_{N=0}^\infty (CL_p)^{\frac Np} 
t^{\frac{1-\theta_p}p}
\left( 
\int_{(0,t)^N}\bone_{\{t_1<\dots<t_N\}} (\Delta 
t_{N+1})^{\theta_p-1}\prod_{i=1}^N G_p(\Delta t_i)\,\dd t_i \right)^{\frac1p},
\end{align*}
where $L_p=\int_{(0,\eta)} z^{1+2/d}(3\lvert \log z\rvert+1)\,\la(\dd 
z)+\int_{[\eta,\infty)} z^p\,\la(\dd z)$ and $G_p(s)=s^{\theta_p/3-1}$ if $s\leq 1$ and 
$G_p(s)=s^{\theta_p-1}$ if $s\geq1$. On $[0,t]$, we have 
$G_p(s) \leq C s^{\theta_p/3-1}$ for 
some $C$ that only depends on $t$. Furthermore, if 
$m^{\log}_{1+2/d}(\la)+M_{1+2/d}(\la)<\infty$, then $L_p\leq C$ for some constant $C$ that 
is independent of $p$; if $\la([R,\infty))\sim C R^{-1-2/d}$ as $R\to\infty$, then 
$L_p\sim C\theta_p^{-1}$ as $p \uparrow 1 + \frac2d$, for some (other) constant $C$ that is 
also independent of $p$. Therefore, in both cases, 
\begin{align*}
\E[\ov Y(t,x)^p]^{\frac1p}
&\leq \sum_{N=0}^\infty (C\theta_p^{-1})^{\frac Np} 
\left( 
t^{1-\theta_p}\int_{(0,t)^N}\bone_{\{t_1<\dots<t_N\}} (\Delta 
t_{N+1})^{\theta_p-1}\prod_{i=1}^N (\Delta t_i)^{\frac{\theta_p}{3}-1}\,\dd t_i
\right)^{\frac1p}\\
&=\sum_{N=0}^\infty \frac{(C\theta_p^{-1}t^{\frac{\theta_p}3} 
\Ga(\frac{\theta_p}3))^{\frac N p}}{\Ga(\theta_p+\frac {\theta_p}3 N)^{\frac 1p}} 
\Ga(\theta_p)^{\frac1p} \leq 
C\frac{\Ga(\theta_p)^{\frac1p}}{\theta_p} 
\exp \{ t (C \theta_p^{-1} 
\Ga(\tfrac{\theta_p}3))^{3/\theta_p} \}
\end{align*}
by \cite[Lemma~A.3]{Berger21} and Lemma~\ref{lem:Stirling}. Thus, by Markov's inequality and possibly after enlarging $C$ in the second step,
\beq\label{eq:tail2} 
\P(\ov Y(t,x)>R)\leq CR^{-p} 
\frac{\Ga(\theta_p)}{\theta_p^p}
\exp \{ C (\theta_p^{-1} \Ga(\tfrac{\theta_p}3))^{3/\theta_p} \}\leq 
CR^{-p} e^{(\frac C{\theta_p})^{C/\theta_p}} \eeq
as $p$ is close enough to $1+\frac 2d$ (because   $\theta_p\downarrow 0$ 
as $p \uparrow 1 +2/d$ thus $\Ga(\theta_p)\sim \theta_p^{-1}$).

Let $W:(0,\infty)\to(0,\infty)$ be the (principal branch of the) Lambert $W$ function, that is, $W(x)$ is the unique solution on $(0,\infty)$ to the equation $We^W=x$. 
We choose $p=p(R)<1+\frac2d$ such that 
$$\frac C{\theta_p} =  \exp(W(\log\log R - \log \log\log R+\log \log\log\log R)).$$
Let us denote the expression on the right-hand side by $z=z(R)$. Then $p(R)=1+\frac 2d -\frac{2C}{dz(R)}$ and \eqref{eq:tail2} becomes
$$ \P(\ov Y(t,x)>R)\leq CR^{-1-\frac 2d}R^{\frac {2C}{dz(R)}} e^{z(R)^{z(R)}}= CR^{-1-\frac 2d}R^{\frac {2C}{dz(R)}}e^{e^{z(R)\log z(R)}}.$$
Note that $z=e^{W(x)}$ satisfies $z\log z= x$ by the definition of $W$. Therefore,
\beq\label{eq:tail3} 
\P(\ov Y(t,x)>R)\leq CR^{-1-\frac 2d}R^{\frac {2C}{dz(R)}}
e^{\frac{(\log R)(\log \log \log R)}{ \log \log R}}. \eeq
By \cite[Eq.~(4.13.10)]{NIST}, there exists $x_0\in(0,\infty)$ such that
\beq\label{eq:W} \log x - \log\log x\leq W(x)\leq \log x - \frac12\log\log x\eeq
for all  $x\geq x_0$.
Consequently, for sufficiently large $R$,
$$ z(R)\geq e^{W(\frac 12\log \log R)} \geq e^{\log(\frac12 \log\log R)-\log\log(\frac12\log\log R)} = \frac{\frac12\log\log R}{\log(\frac12\log\log R)},$$
which implies
$$ R^{\frac {2C}{dz(R)}}=e^{\frac{2C\log R}{dz(R)}} \leq e^{\frac{4C(\log R)\log(\frac12\log\log R)}{d \log \log R}}.$$
Combining this with \eqref{eq:tail3}, we obtain \eqref{eq:sol-lt} if $d\geq2$. The proof essentially remains the same if $d=1$: by \eqref{eq:mom-p}, because $\mu_p(\la)\leq m_2(\la)+M_p(\la)\leq C \theta_p^{-1}$ uniformly in $p\in[2,3)$,
$$ \E[Y(t,x)^p]\leq 
C
\exp \{ C (\theta_p^{-1}\Ga(\theta_p))^{1/\theta_p} \}
\leq Ce^{(\frac 
C{\theta_p})^{C/\theta_p}}. $$ 
With this bound, we can go back to \eqref{eq:tail2} and complete the proof as before.

\bigskip
\noindent\bff{Step 2: Lower bound} 
\smallskip

\noindent   Without loss of generality, we may assume that $M_0(\la)=\la([1,\infty))>0$. In this case, with the same notation as in the lower bound proof of Theorem~\ref{thm:sol-ht}, we have $\zeta_i\geq1$ and
\begin{align*}
	&\P_N\Biggl(Y_<(\tau_N,\eta_N)\prod_{i=1}^{N} u_<(\tau_{i},\eta_{i};\tau_{i-1},\eta_{i-1}) \zeta_i>R\Biggr)\\
	&\qquad\geq \mathtoolsset{multlined-width=0.8\displaywidth} \begin{multlined}[t]  
\P_N\Biggl(\Biggl\{ 2^{-N-1} \prod_{i=1}^{N} g(\Delta \tau_i,\Delta \eta_i) >R\Biggr\} 
\cap \{Y_<(\tau_N,\eta_N)>\tfrac12\}\\
	 \cap \bigcap_{i=1}^N \{u_<(\tau_{i},\eta_{i};\tau_{i-1},\eta_{i-1})>\tfrac12g(\Delta \tau_i,\Delta \eta_i)\} \Biggr)
	\end{multlined}\\
	&\qquad\geq \mathtoolsset{multlined-width=0.8\displaywidth} \begin{multlined}[t]  
\E_N\Biggl[\bone\Biggl\{2^{-N-1}\prod_{i=1}^{N} g(\Delta \tau_i,\Delta \eta_i) >R  
\Biggr\} \P^{\tau,\eta}(Y_<(\tau_N,\eta_N)>\tfrac12) \\
		\times\prod_{i=1}^N \P^{\tau,\eta}( u_<(\tau_{i},\eta_{i};\tau_{i-1},\eta_{i-1})>\tfrac12g(\Delta \tau_i,\Delta \eta_i )\Biggr],\end{multlined}
\end{align*}
where the second step follows by using the independence under $\P^{\tau,\eta}$ of the variables $Y_<(\tau_N,\eta_N)$ and $(u_<(\tau_{i},\eta_{i};\tau_{i-1},\eta_{i-1}))_{i=1,\dots,N}$. Observe that $\E^{\tau,\eta}[Y_<(\tau_N,\eta_N)]=1$ and $\E^{\tau,\eta}[u_<(\tau_{i},\eta_{i};\tau_{i-1},\eta_{i-1})]=g(\Delta \tau_i,\Delta \eta_i )]$. Thus, by Lemma~\ref{lem:PZ} and \cite[Cor.~6.5]{Berger21b}, there is a deterministic $C>0$ such that 
\beq\label{eq:prob} \P^{\tau,\eta}(Y_<(\tau_N,\eta_N)>\tfrac12)>C,\qquad \P^{\tau,\eta}( u_<(\tau_{i},\eta_{i};\tau_{i-1},\eta_{i-1})>\tfrac12g(\Delta \tau_i,\Delta \eta_i )> C \eeq
for all $i=1,\dots,N$. Moreover, $g(\Delta\tau_i,\Delta\eta_i)\geq (2\pi \Delta\tau_i)^{-d/2}e^{-1/2} = C(\Delta \tau_i)^{-d/2}$ by the definition of $\tau_i$. Hence,
\begin{multline*} \P_N\Biggl(Y_<(\tau_N,\eta_N)\prod_{i=1}^{N} 
u_<(\tau_{i},\eta_{i};\tau_{i-1},\eta_{i-1}) \zeta_i>R\Biggr)\\ 
\geq 
C^N\P_N\Biggl(2^{-N-1}\prod_{i=1}^{N} g(\Delta \tau_i,\Delta \eta_i) >R  \Biggr)\geq C^N\P_N\Biggl(\prod_{i=1}^{N} (\Delta \tau_i)^{-\frac d2}>C^{-N} R  \Biggr).
\end{multline*}
To evaluate this probability, recall the density of $\tau_i$ from \eqref{eq:density}. We have 
\begin{align*}
&\P_N\Biggl(\prod_{i=1}^{N} (\Delta \tau_i)^{-\frac d2}>C^{-N} R  
\Biggr)\\
&\qquad=C^N\int_{(0,\frac tN)^N} \frac{e^{-C\sum_{i=1}^N s_i^{1+d/2}}}{(1-e^{-C(\frac 
tN)^{1+d/2}})^N} \bone\Biggl\{ \prod_{i=1}^N s_i^{-\frac d2} > 
C^{-N}R\Biggr\}\prod_{i=1}^N s_i^{\frac d2}\,\dd s_i\\
&\qquad\geq C^NN^{(1+\frac d2)N}\int_{(0,\frac tN)^N}  
\bone\Biggl\{ \prod_{i=1}^N s_i < C^{N}R^{-\frac 2d}\Biggr\}\prod_{i=1}^N 
s_i^{\frac d2}\,\dd s_i\\
	&\qquad= C^NN^{(1+\frac d2)N}R^{-1-\frac 2d} \int_{(0,\frac {tR^{2/(dN)}}{CN})^N}  
\bone\Biggl\{ \prod_{i=1}^N u_i < 1 \Biggr\}\prod_{i=1}^N u_i^{\frac 
d2}\,\dd u_i.
\end{align*}
Provided that $\frac {tR^{2/(dN)}}{CN}>1$ (i.e., $N^N< (\frac tC)^N R^{2/d}$), we can use Lemma~\ref{lemma:iter-int} (keeping only the term corresponding to $i=N-1$) to obtain
\begin{align*}
	\P_N\Biggl(\prod_{i=1}^{N} (\Delta \tau_i)^{-\frac d2}>C^{-N} R  \Biggr) \geq C^NN^{(1+\frac d2)N}R^{-1-\frac 2d} \biggl(\log \frac {tR^{\frac2{dN}}}{CN}\biggr)^{N-1}.
\end{align*}
We will further restrict ourselves to $N$ such that 
\beq\label{eq:cond2} N^N \leq \biggl(\frac tC\biggr)^N R^{1/d},\eeq 
in which case  
\begin{align*}
	\P_N\Biggl(\prod_{i=1}^{N} (\Delta \tau_i)^{-\frac d2}>C^{-N} R  \Biggr) \geq C^NN^{\frac d2 N}R^{-1-\frac 2d}  \log^{N-1} R.
\end{align*}
In fact, we will consider $N$ such that equality is attained in \eqref{eq:cond2}, that is, we choose
\beq\label{eq:NR} N=Ke^{W(\frac 1{dK}\log R)}, \eeq
where $K$ is actually $\frac tC$ from above. Recalling \eqref{eq:PAN}, we obtain
\begin{align*}
	\P(Y(t,x)>R)&\geq   \P(A_N)\P_N\Biggl(Y_<(\tau_N,\eta_N)\prod_{i=1}^{N} u_<(\tau_{i},\eta_{i};\tau_{i-1},\eta_{i-1}) \zeta_i>R\Biggr)\\
	&\geq C^NN^{-N}R^{-1-\frac 2d}\log^{N-1}(R)\\
	&=  \frac{R^{-1-\frac 2d}(C\log R)^{K\exp(W(\frac 1{dK}\log R))}}{\log R(K\exp(W(\frac 1{dK}\log R)))^{K\exp(W(\frac 1{dK}\log R))}}.
\end{align*}
By \eqref{eq:W}, for sufficiently large $R$
\begin{align*}
\P(Y(t,x)>R)&\geq  
\frac{R^{-1-\frac 2d}(C\log R)^{K\exp(W(\frac 1{dK}\log R))}}
{(\log R)(\frac1d\log R)^{K\exp(W(\frac 1{dK}\log R))}}(\log (\tfrac1{dK} \log 
R))^{\frac12K\exp(W(\frac 1{dK}\log R))}\\
	&\geq \frac{R^{-1-\frac 2d}}{\log R} C^{K\exp(W(\frac 1{dK}\log R))} e^{\frac12K\exp(W(\frac 1{dK}\log R))\log\log\log R}\\
	&\geq R^{-1-\frac 2d}e^{\frac14K\exp(W(\frac 1{dK}\log R))\log\log\log R} \\
	&\geq  R^{-1-\frac 2d } e^{C\frac{(\log R)(\log\log\log R)}{\log \log R}},
\end{align*}
where the last step holds for some sufficiently small $C>0$.
\epr

In the previous proof, we used the following lemma, which is a uniform-in-$N$ version of  \cite[Lemma~4.1~(4)]{JessenMikosch},

\blem\label{lem:JM} Let $N\in\N$ and $X_1,\dots,X_N$ be independent and identically distributed  such that there are $C_0,\al\in(0,\infty)$ with  $\P(X_1>R)\geq C_0R^{-\al}$ for all $R\geq1$. Then there is $C\in(0,\infty)$ such that for all $N\in\N$ and $R>1$,
$$ \P\Biggl(\prod_{i=1}^N X_i >R\Biggr)\geq \frac{C^N}{(N-1)!} R^{-\al}\log^{N-1} R. $$
\elem
\bpr Conditionally on the event $A=\bigcap_{i=1}^N \{X_i>1\}$, the $X_i$'s are still independent and identically distributed and satisfy
\begin{align*} \P(X_1> R \mid A) &= \frac{\P(X_1>R\vee 1)}{\P(X_1>1)} \geq 
\bone_{(0,1)}(R)+\frac{C_0}{\P(X_1>1)} R^{-\al}\bone_{[1,\infty)}(R)\\ &\geq 
C_1R^{-\al}\bone_{[C_1^{1/\al},\infty)}(R)\end{align*}
with $C_1=C_0/\P(X_1>1)$, which belongs to $(0,1]$ by assumption. Let $Y_1,\ldots,Y_N$ be 
independent Pareto random variables with scale parameter $C_1^{1/\al}$ and shape parameter 
$\al$ (in particular, their tail function is given by the right-hand side of the previous 
display). Using quantile representation, one can construct these variables in such a way 
that conditionally on $A$, we have $X_i\geq Y_i$ almost surely. Thus,
$$ \P\Biggl(\prod_{i=1}^N X_i >R\Biggr) \geq \P(X_1>1)^N \P\Biggl(\prod_{i=1}^N X_i 
>R\mathrel{\bigg|} A\Biggr)\geq \P(X_1>1)^N\P\Biggl(\prod_{i=1}^N Y_i >R\Biggr).$$
It is an elementary result that $\sum_{i=1}^N \log ({Y_i}/{C_1^{1/\al}})$ is $\Ga(N,\al)$-distributed. Therefore,
\begin{align*}
	 \P\Biggl(\prod_{i=1}^N Y_i >R\Biggr)& = \P\Biggl(\sum_{i=1}^N \log 
\frac{Y_i}{C_1^{1/\al}} >\log\frac{R}{C_1^{N/\al}}\Biggr)=\frac{\al^N}{(N-1)!} \int_{\log 
(R/C_1^{N/\al})}^\infty u^{N-1}e^{-\al u}\,\dd u\\
	 &= \frac{\al^N}{(N-1)!}\int_{R/C_1^{N/\al}}^\infty u^{-\al-1} \log^{N-1}(u) \,\dd u\\
	 &\geq\frac{\al^N}{(N-1)!}\log^{N-1} \frac{R}{C_1^{N/\al}} \int_{R/C_1^{N/\al}}^\infty u^{-\al-1} \,\dd u\\ &=\frac{\al^{N-1}C_1^N}{(N-1)!}R^{-\al}\log^{N-1} \frac{R}{C_1^{N/\al}}.
\end{align*}
By decreasing $C_0$ if necessary, there is no loss of generality if we assume that $C_1<1$. This implies
\[ \P\Biggl(\prod_{i=1}^N X_i >R\Biggr)\geq\frac{\al^{N-1}(C_1\P(X_1>1))^N}{(N-1)!}R^{-\al}\log^{N-1} R, \]
proving the lemma.
\epr

\brem\label{rem:unbounded} 
In the lower bound proofs of both Theorem~\ref{thm:sol-ht} and \ref{thm:sol-lt}, it was 
crucial that   \eqref{eq:SHE} is considered on the whole space $\R^d$. To illustrate this 
point, let us take a standard Poisson noise (i.e., $\la=\delta_1$) and restrict the noise to a spatial domain $D$ with finite and positive Lebesgue measure $\lvert D\rvert$. Because $\lvert D\rvert<\infty$, there is only a finite Poisson-distributed number 
$L$ of points up to time $t$. Therefore,
$$ Y(t,x)=1+\sum_{N=1}^L \int_{((0,t)\times D)^N} \prod_{i=2}^{N+1} g(\Delta t_i,\Delta x_i)\prod_{j=1}^N \La(\dd t_j,\dd x_j). $$
Call the $N$-fold integral $I_N(t,x)$.
Either by bounding the tail probability explicitly or by estimating the $p$th moment and then optimizing, one can show that 
$$ \P( I_N(t,x) > R)\leq \frac{C^N}{N^N}R^{-1-\frac 2d}\log^N R$$
for some $C>0$ that is independent of $N$ and $R$.
Therefore, by conditioning on $L$,  
\begin{align*}
\P(Y(t,x)>R)&\leq \sum_{L=1}^\infty \frac{e^{-t\lvert D\rvert}(t\lvert D\rvert)^L}{L!} \P\Biggl(\sum_{N=1}^L I_N(t,x)> R-1\Biggr)\\&\leq \sum_{L=1}^\infty \frac{e^{-t\lvert D\rvert}(t\lvert D\rvert)^L}{L!}\sum_{N=1}^L \P(I_N(t,x)>\tfrac{R-1}{L})\\
&\leq \sum_{L=1}^\infty \frac{e^{-t\lvert D\rvert}(t\lvert D\rvert)^L}{L!} \sum_{N=1}^L \frac{C^N}{N^N} R^{-1-\frac2d} L^{1+\frac2 d} \log^N  R.
\end{align*}
Since $L!\geq (L-N)!N!$ and $(t\lvert D\rvert)^L L^{1+2/d} \leq C^L= C^{L-N} C^N$, we can use Lemma~\ref{lem:Stirling} to get
\begin{align*}
\P(Y(t,x)>R)
&\leq  e^{-t\lvert D\rvert}R^{-1-\frac2d} \sum_{N=1}^\infty \frac{(C\log R)^N}{N!N^N} \sum_{L=N}^\infty \frac{C^{L-N}}{(L-N)!}  \leq CR^{-1-\frac 2d} e^{C(\log R)^{1/2}},
\end{align*}
which is  much smaller than the tails we obtained in Theorem~\ref{thm:sol-lt}. 

Similarly, if $\dot \La$ has Lévy measure \eqref{eq:la} with $\al<1+\frac2d$, then one can show that 
$$ \P( I_N(t,x) > R)\leq \frac{C^N}{N^{(1+\theta_\al)N}}R^{-\al}\log^N R.$$
Again, if $\dot \La$ only acts on $D$,   we have
\beq\label{eq:tail}
\P(Y(t,x)>R)
\leq  CR^{-\al} \sum_{N=1}^\infty \frac{(C\log R)^N}{N!N^{(1+\theta_\al)N}}    \leq CR^{-\al} e^{C(\log R)^{1/(2+\theta_\al)}},
\eeq
which is  much lighter than the tails  derived in  Theorem~\ref{thm:sol-ht}.

Finally, let us mention \cite{Cohen08}, where the exact tail behavior of   
solutions to stochastic differential equations (SDEs) with multiplicative stable noise
was determined. 
Their proof heavily relies on an exact representation of the solution as a random product 
of heavy-tailed terms, which is not available for the SHE. In addition, the SDE situation 
differs from the SHE in two aspects: first, space only consists of one point and is 
therefore bounded; second, the fundamental solution, unlike the heat kernel, has no 
singularity. This is why the tail behavior of the solution to a stable SDE is of the form 
given by the right-hand side of \eqref{eq:tail} but without $\theta_\al$ in the exponent. 
The reader may verify that $\theta_\al$ enters \eqref{eq:tail} only because  the heat 
kernel has a singularity.
\erem

\subsection{Tail bounds for the local supremum}

We need the following   assumption.

\settheoremtag{(Sup)} 
\bcond\label{cond:sup} If $d=1$, then  $m_q(\la)<\infty$ for some $q\in(0,2)$. If $d\geq2$,   we have   
$m^{(\log)}_{2/d}(\la)<\infty$.
\econd

Note that $m_2(\la)<\infty$ for all Lévy measures, so Condition~\ref{cond:sup} is rather 
mild in dimension $d=1$. Also, if $m_{2/d+\eps}(\la)=\infty$ for some small 
$\eps>0$, then the solution to \eqref{eq:SHE} under additive noise is    unbounded on any 
non-empty open subset of $\R^d$ at a fixed time, see \cite[Theorem 3.3]{CDH}. Thus, 
Condition~\ref{cond:sup} is also rather mild in dimensions $d\geq2$. Recall that 
$\calq$ is the set of unit cubes .

\bthm\label{thm:sup-ht}   Assume Condition~\ref{cond:ht} for some $\al\in(0,\frac 2d]$ and Condition~\ref{cond:sup}.
For every $t>0$, there are constants $C_1,C_2\in(0,\infty)$ such that for all $Q\in\calq$ and $R>1$,
\beq\label{eq:sup-ht} C_1R^{-\al}e^{C_1 (\log R)^{1/(1+\theta_\al)}}\leq \P\biggl(  \sup_{x\in Q} Y(t,x)   >R\biggr)\leq  C_2R^{-\al}e^{C_2 (\log R)^{1/(1+\theta_\al)}}.\eeq
\ethm

\bpr We only need to prove the upper bound. The lower bound immediately follows from Theorem~\ref{thm:sol-lt}. Without loss of generality, assume that $Q=(0,1)^d$. We first consider $d\geq2$, in which case $m_1(\la)=\int_{(0,1)} z\,\la(\dd z)<\infty$. Therefore, $Y(t,x)=e^{-m_1(\la)t} \wh Y(t,x)$, where $\wh Y(t,x)$ is the mild solution to 
\beq\label{eq:Yhat} \wh Y(t,x)=1+\int_0^t\int_{\R^d}\int_{(0,\infty)} g(t-s,x-y)\wh Y(s,y)z\,\mu(\dd s,\dd y,\dd z), \eeq
and it suffices to prove the second inequality in \eqref{eq:sup-ht} for $\wh Y$ instead of $Y$. Similarly to \eqref{eq:separate} and \eqref{eq:separate2}, we have, with obvious notation, that
\beq\label{eq:separate3} \wh Y(t,x)=\sum_{N=0}^\infty \int_{((0,t)\times\R^d)^N}\wh Y_<(t_1,x_1)\prod_{i=2}^{N+1} \wh u_<(t_{i-1},x_{i-1};t_i,x_i) \prod_{j=1}^N  \,\La_\geq(\dd t_j,\dd x_j). \eeq
Therefore, using the estimate $(\sum a_i)^p\leq \sum a_i^p$ and independence, we obtain for any $0<p<\al<1$ that
\beq\label{eq:Y<}  \begin{split}
	\E\biggl[\sup_{x\in Q} \wh Y(t,x)^p\biggr]&\leq   \sum_{N=0}^\infty M_p(\la)^N \int_{((0,t)\times\R^d)^N} \E[\wh Y_<(t_1,x_1)^p]\\ 
&\quad\times\E\biggl[\sup_{x\in Q} \wh u_<(t_N,x_N;t,x)^p\biggr]\prod_{i=2}^N\E[\wh 
u_<(t_{i-1},x_{i-1};t_i,x_i)^p]\prod_{j=1}^N  \dd t_j\,\dd x_j. 
\end{split}\raisetag{-3\baselineskip}\eeq
Combining Lemma \ref{lem:sup}, \eqref{eq:p<1aux}, Lemma 3.5 in \cite{Chong1}, and
Lemma~\ref{lem:Stirling}, we obtain
\begin{align*}
&\E\biggl[\sup_{x\in Q} \wh Y(t,x)^p\biggr]\\
&\quad\leq \frac{C}{1-\frac d2 
p}\sum_{N=0}^\infty   (CM_p(\la) )^N   \hspace*{-3pt}
\int_{((0,t)\times\R^d)^N} \hspace*{-3pt}
(t-t_N)^{-\frac 
d2p}e^{-C\lvert x_N\rvert^2}\prod_{i=2}^N g(\Delta t_i,\Delta x_i)^p\prod_{j=1}^N \dd 
t_j\,\dd x_j \\
	&\quad\leq \frac{C}{1-\frac d2 p}\sum_{N=0}^\infty (CM_p(\la))^N  \int_{(0,t)^N}\bone_{\{t_1<\dots<t_N\}}(t-t_N)^{-\frac d2p} \prod_{i=2}^N (\Delta t_i)^{-\frac d2(p-1)}\prod_{j=1}^N \dd t_j  \\
	&\quad\leq\frac{C\Ga(1-\frac d2 p)t^{1-\frac d2 p}}{(1-\frac d2 
p)\Ga(\theta_p)}\sum_{N=0}^\infty \frac{(CM_p(\la) 
t^{\theta_p} \Ga(\theta_p))^N}{\Gamma(2\vee(N\theta_p-\frac d2))} \leq \frac{C\Ga(1-\frac d2 p)}{1-\frac d2 p} e^{CM_p(\la)^{1/\theta_p}}.
\end{align*}
By our assumptions on $\la$, we have $M_p(\la)\sim C/(\al-p)$ as $p\uparrow \al$. Therefore, if $\al\in(0,\frac 2d)$,
\begin{align*} \P\biggl(  \sup_{x\in Q} \wh Y(t,x)   >R\biggr)\leq  R^{-p}\,	\E\biggl[\sup_{x\in Q} \wh Y(t,x)^p\biggr]\leq C R^{-p}e^{C(\al-p)^{-1/\theta_p}}
\end{align*}
for some constant $C$ that does not depend on $p$. If $\al = \frac2d$, then we obtain
an extra factor $(1-\frac d2 p)^{-2}$ (since $\Gamma(x) \sim x^{-1}$ as $x \downarrow 
0$) in the previous line. But this is bounded by $Ce^{C( 2/d-p)^{-1/\theta_p}}$, so the 
last display remains valid upon enlarging the value of $C$ in the second step. So in all 
cases, as in the upper bound proof of Theorem~\ref{thm:sol-ht}, the current proof   can 
be completed by choosing
 $p=\al-(\theta_\al\log R)^{-1/(1+1/\theta_\al)}$. 
%

If $d=1$ and $\al \leq 1$, we can re-use \eqref{eq:Y<} and the subsequent argument except 
that we have to replace   $\wh Y_<$ and $\wh u_<$ by $Y_<$ and $u_<$ and use Jensen's 
inequality to raise the $p$-moments to $\theta$-moments for some $\frac 32\vee q<\theta<2$ 
before applying Lemma~\ref{lem:sup}, where $q$ is given in \emph{Condition (Sup)}. 
If $1 <\al \leq \tfrac{3}{2} \vee q$, 
then choose $\theta \in (\frac 32\vee q, 2)$ fix, while if 
$\alpha > \tfrac{3}{2} \vee q$ then let $\theta = p$. In both cases we let 
$p \uparrow \alpha$. We first observe that applying 
$\E_<[(\cdot)^{\theta}]^{1/\theta}$ instead of $\E[(\cdot)^p]$ in \eqref{eq:Y<} leads to
\begin{multline*}
\E_<\biggl[\sup_{x\in Q}   Y(t,x)^{\theta}\biggr]^{\frac{1}{\theta}}\leq 
  \sum_{N=0}^\infty   
\int_{((0,t)\times\R)^N} \E[  Y_<(t_1,x_1)^\theta]^{\frac{1}{\theta}}\\ 
 \times\E\biggl[\sup_{x\in Q}   
u_<(t_N,x_N;t,x)^\theta \biggr]^{\frac{1}{\theta}}\prod_{i=2}^N\E[  u_<(t_{i-1},x_{i-1};t_i,x_i)^\theta]^{\frac{1}{\theta}}\prod_{j=1}^N \La_{\geq}( \dd t_j,\dd x_j). 
\end{multline*}
By Lemma~\ref{lem:sup} and \cite[Cor.\ 6.5]{Berger21b}, the left-hand side is further bounded by
\begin{align*}
&\biggl(\frac{C}{1-\frac \theta2}\biggr)^{\frac 1\theta} \Biggl( \sum_{N=0}^\infty  C^N \int_{((0,t)\times\R)^N}   (t-t_N)^{-\frac 12}e^{-C^{-1}\lvert x_N\rvert^2} \prod_{i=2}^N g(\Delta t_i,\Delta x_i)\prod_{j=1}^N \La_{\geq}( \dd t_j,\dd x_j)\\
	&\qquad=\biggl(\frac{C}{1-\frac \theta2}\biggr)^{\frac 1\theta} \Biggl(1+   C  \int_0^t\int_\R   (t-t_N)^{-\frac 12}e^{-C^{-1}\lvert x_N\rvert^2} Y'_\geq(t_N,x_N)\, \La_{\geq}( \dd t_N,\dd x_N)\Biggr),
\end{align*}
where $Y'_\geq(t,x)$ is the mild solution to the stochastic heat equation with initial 
condition $1$ and noise $C\dot\La_\geq$. Hence, writing $\La_\geq(\dd s,\dd 
y)=(\La_\geq(\dd s,\dd y)-M_1(\la)\,\dd s\,\dd y) + M_1(\la)\,\dd s\,\dd y$, we obtain 
from the conditional Jensen's inequality, Minkowski's integral inequality, and the 
BDG inequality that
\begin{multline*}
\E\biggl[\sup_{x\in Q}   Y(t,x)^{p}\biggr]^{\frac{1}{p}}\leq \E\Biggl[	\E_<\biggl[\sup_{x\in Q}   Y(t,x)^{\theta}\biggr]^{\frac{p}{\theta}}\Biggr]^{\frac1p}\\
\leq \biggl(\frac{C}{1-\frac \theta2}\biggr)^{\frac 1\theta} \Biggl(1+   C M_1(\la) \int_0^t\int_\R   (t-s)^{-\frac 12}e^{-C^{-1}\lvert y\rvert^2} \E[ Y'_\geq(s,y)^p]^{\frac 1p}\,   \dd s\,\dd y\\
+\Biggl( CM_p(\la) \int_0^t\int_\R (t-s)^{-\frac p2} e^{-C^{-1}p\lvert y\rvert^2} \E[Y'_\geq(s,y)^p]\,\dd s\,\dd y\Biggr)^{\frac1p}\Biggr).
\end{multline*}
Since $\E[Y'_\geq(s,y)^p]\leq Ce^{CM_p(\la)^{1/\theta_p}}$ (cf.\ \eqref{eq:mom-ub}), it follows that
\begin{align*}
	\E\biggl[\sup_{x\in Q}   Y(t,x)^{p}\biggr]^{\frac{1}{p}}&\leq \biggl(\frac{C}{1-\frac \theta2}\biggr)^{\frac 1\theta} \Biggl(1+Ce^{M_p(\la)^{1/\theta_p}}+\biggl(\frac{CM_p(\la)}{1-\frac p2}\biggr)^{\frac 1p}e^{CM_p(\la)^{1/\theta_p}}\Biggr).
\end{align*}
Since $\theta$ is fixed, we obtain (both when $\al\in(0,\frac2d)$ and when $\al=\frac 2d$) that
$$ \E\biggl[\sup_{x\in Q}   Y(t,x)^{p}\biggr]^{\frac{1}{p}}\leq e^{CM_p(\la)^{1/\theta_p}}, $$
from which the second inequality in \eqref{eq:sup-ht} follows as before.
\epr

In the proof of the previous theorem, we used some technical moment estimates on the local 
supremum of $Y_<$, $u_<$ and $\widehat Y_{<}$, $\widehat u_{<}$.

\blem\label{lem:sup} Suppose that $d\geq2$. If $m^{(\log)}_{2/d}(\la)<\infty$,
then, for every $T>0$, there exists $C\in(0,\infty)$ such that for all 
$0<s<t\leq T$, $Q \in \calq$, $y\in \R^d$ and $0<p<\frac 2d$,
\beq\label{eq:ub-u<}\begin{split}  \E\biggl[\sup_{x\in Q} \wh Y_<(t,x)^p\biggr] &\leq C(1-\tfrac d2 
p)^{-1},\\ \E\biggl[\sup_{x\in Q} \wh u_<(s,y;t,x)^p\biggr] 
 &\leq C(1-\tfrac d2 p)^{-1}(t-s)^{-\frac d2p}e^{-C^{-1}\lvert y\rvert^2}. \end{split}
\eeq
If $d=1$ and there is $q\in(0,2)$ such that $m_q(\la)=\int_{(0,1)} z^q\,\la(\dd z)<\infty$, then there exists $C>0$ such that for all $p\in(q\vee \frac 32,2)$, 
\beq\label{eq:ub-u<2}  \begin{split}\E\biggl[\sup_{x\in Q}   Y_<(t,x)^p\biggr]& \leq C(1-\tfrac p2 
)^{-1},\\ \E\biggl[\sup_{x\in Q}   u_<(s,y;t,x)^p\biggr]  
&\leq C(1-\tfrac p2 )^{-1}(t-s)^{-\frac p2}e^{-C^{-1}\lvert y\rvert^2}. \end{split}\eeq
\elem

Note the lower bound $q \vee \tfrac{3}{2}$ for $d = 1$ in the moment inequality.
This is a minor technicality, one can extend the inequality for smaller $p$  
applying Lyapunov's inequality for moments.

\bpr
We only prove the uniform moment bound on $\wh u_<(s,y;t,x)$ or $ u_<(s,y;t,x)$; the 
bounds on $\wh Y_<(t,x)$ and $Y_<(t,x)$ can be shown in a similar fashion.
We may and do assume that $Q = (0,1)^d$.
By definition,
$$ \wh u_<(s,y;t,x)=g(t-s,x-y)+\int_s^t \int_{\R^d}\int_{(0,1)} g(t-r,x-v)\wh u_<(s,y;r,v)z\,\mu(\dd r,\dd v,\dd z). $$
If $y\in (-d-1,d+1)^d$, we use the bound
\beq\label{eq:bound}\begin{split}
&\sup_{x\in Q} \wh u_<(s,y;t,x) \\
&\quad\leq C\Biggl[ (t-s)^{-\frac d2} + \int_s^t 
\int_{(-d-1,d+1)^d} \int_{(0,1)} (t-r)^{-\frac d2} \wh u_<(s,y;r,v)z\,\mu(\dd r,\dd v,\dd 
z)\\
&\qquad+ \int_s^t \int_{\R^d\setminus (-d-1,d+1)^d} \int_{(0,1)} (t-r)^{-\frac 
d2}e^{-\frac{(\lvert v\rvert - \sqrt{d})^2}{2 (t-r)}}\wh u_<(s,y;r,v)z\,\mu(\dd r,\dd 
v,\dd z)\Biggr].
\end{split}\!\eeq
In the second integral, we have $(t-r)^{-  d/2}e^{- {(\lvert v\rvert - \sqrt{d})^2}/({2 
(t-r)})}\leq Ce^{- {(\lvert v\rvert - \sqrt{2})^2}/({4 T})}\leq C$. Furthermore, note that 
there are two ways of estimating the $p$th moment (for $p\in(0,1)$) of a Poisson integral 
of an adapted process $f$, namely
\beq\label{eq:twoways} \E\Biggl[\Biggl(\int f \,\dd \mu\Biggr)^p\Biggr] \leq \int \E[f^p] \,\dd \nu\qquad\text{or}\qquad   \E\Biggl[\Biggl(\int f \,\dd \mu\Biggr)^p\Biggr]\leq \Biggl(\int \E[f]\,\dd\nu\Biggr)^p, \eeq
depending on whether we use $(\sum a_i)^p\leq \sum a_i^p$ or Jensen's inequality. Applying the first method to the first integral in \eqref{eq:bound} if $(t-r)^{-d/2}z>1$ and the second method to the first integral in \eqref{eq:bound} if $(t-r)^{-d/2}z\leq 1$ as well as to the second integral in \eqref{eq:bound}, we derive the bound
\begin{align*}
&\E\biggl[\sup_{x\in Q} \wh u_<(s,y;t,x)^p\biggr]\leq C\Biggl[ (t-s)^{-\frac d2 p} \\
&\qquad + \int_s^t\int_{(-d-1,d+1)^d} \int_{(0,1)} 
(t-r)^{-\frac d2p}z^p\bone_{\{(t-r)^{-  d/2}z>1\}}\E[ \wh u_<(s,y;r,v)^p]\,\dd r\,\dd 
v\,\la(\dd z)\\
&\qquad +\Biggl(\int_s^t\int_{(-d-1,d+1)^d} \int_{(0,1)} (t-r)^{-\frac 
d2}z\bone_{\{(t-r)^{-  d/2}z\leq1\}}\E[\wh u_<(s,y;r,v)]\,\dd r\,\dd v\,\la(\dd 
z)\Biggr)^p\\
&\qquad +\Biggl(\int_s^t \int_{\R^d\setminus (-d-1,d+1)^d} \int_{(0,1)}  
\E[\wh u_<(s,y;r,v)]z\, \dd r\,\dd v\,\la(\dd z)\Biggr)^p\Biggr].
\end{align*}
As $\int_{\R^d} \E[\wh u_<(s,y;r,v)^p]\,\dd v\leq \int_{\R^d}  \E[\wh u_<(s,y;r,v)]^p\,\dd v\leq e^{m_1(\la) pt} \int_{\R^d}  g(r-s,v-y)^p\,\dd v\leq C(r-s)^{(1-p)d/2}\leq CT^{(1-p)d/2}\leq C$ (which remains true for $p=1$), we  obtain
\begin{multline*}
\E\biggl[\sup_{x\in Q} \wh u_<(s,y;t,x)^p\biggr] \leq C\Biggl[(t-s)^{-\frac d2 p} + \int_s^t\int_{(0,1)} (t-r)^{-\frac d2 p} z^p\bone_{\{(t-r)^{-  d/2}z>1\}}\,\dd r\,\la(\dd v)\\
+ \Biggl(\int_s^t \int_{(0,1)} (t-r)^{-\frac d2}z\bone_{\{(t-r)^{-  d/2}z\leq1\}} \,\dd 
r\,\la(\dd v)\Biggr)^p+m_1(\la)^p \Biggr].
\end{multline*}
For $0<p<\frac 2d$, the remaining   integrals can be bounded by
\begin{align*}
\int_0^t \int_{(0,1)} r^{-\frac d2 p} z^p\bone_{\{r^{-  d/2} z>1\}}\,\dd r\,\la(\dd v)&\leq \frac{m_{2/d}(\la)}{1-\frac d2 p},\\
\int_0^t \int_{(0,1)} r^{-\frac d2} z\bone_{\{r^{-  d/2} z\leq 1\}}\,\dd r\,\la(\dd v)&\leq 
Cm_{2/d}^{(\log)}(\la),
\end{align*}
respectively, which yields the claim for $y\in (-d-1,d+1)^d$.
If $y\in \R^d\setminus (-d-1,d+1)^d$, we only need to replace the uniform bound on 
$g(t-s,x-y)^p$ by $Ce^{-(\lvert y\rvert-\sqrt{d})^2/(4 T)}$.

If $d=1$, to ease notation, write $v$ for the stochastic part of $u_<$, that is,
\begin{equation} \label{eq:def-v}
v(s,y;t,x) = 
\int_s^t \int_{\R^d}\int_{(0,1)} g(t-r,x-v)\wh u_<(s,y;r,v)z\,\mu(\dd r,\dd v,\dd z).
\end{equation}
We first prove that for $p\in(\frac 32\vee q,2)$ and all $s,t\in[0,T]$ and 
$x,x'\in Q$,
\begin{equation}\label{eq:incr}
\E[\lvert v(s,y;t,x)- v(s,y;t,x')\rvert^p] 
\leq C (t-s)^{-\frac p2}e^{-C^{-1}\lvert 
y\rvert^2}\lvert x-x'\rvert^{3-p}
\end{equation}
for some constant that does not depend on $p$, $(s,y)$ or $(t,x,x')$ (but may  depend on $q$, $\la$ and $T$).
To this end, we 
use the BDG inequality and \cite[Cor.\ 6.5]{Berger21b} to get
\beq\label{eq:incr2}\mathtoolsset{multlined-width=0.9\displaywidth} 
\begin{multlined} 
\E[\lvert v(s,y;t,x)- v(s,y;t,x')\rvert^p] \leq 
\\
Cm_p(\la)\int_s^t\int_\R \lvert g(t-r,x-v)-g(t-r,x'-v)\rvert^pg(r-s,v-y)^p\,\dd r\,\dd 
v.
\end{multlined}\eeq
Note that $m_p(\la)\leq m_q(\la)$ and bound the integral by
\beq\label{eq:incr3}\begin{split}
&C\biggl(\frac{\lvert x-x'\rvert}{t-s}\biggr)^{3-p}	\int_s^{\frac{s+t}2} \int_\R ( g(t-r,x-v)+g(t-r,x'-v))^{2p-3} (r-s)^{-\frac p2}e^{-\frac{p\lvert v-y\rvert^2}{2 (r-s)}}\,\dd r\,\dd v\\
&\qquad\qquad+C(t-s)^{-\frac p2}\int_{\frac{s+t}2}^t \int_\R\lvert g(t-r,x-v)-g(t-r,x'-v)\rvert^p e^{-\frac{p\lvert v-y\rvert^2}{2 (r-s)}}\,\dd r\,\dd v,
\end{split}\raisetag{-2.5\baselineskip}\eeq
where we used the fact that $\lvert g(t-s,x-y)-g(t-s,x'-y)\rvert$ can either be simply bounded using the triangle inequality or using the mean-value theorem (with $\lvert\partial_x g(t,x)\rvert\leq \frac Ct$).

Let $I_1$ and $I_2$ be the two  expressions in \eqref{eq:incr3}. Then
\begin{align*}
I_1&\leq C\biggl(\frac{\lvert x-x'\rvert}{t-s}\biggr)^{3-p}	\int_s^{\frac{s+t}2}  
(t-r)^{\frac32-p} (r-s)^{-\frac p2}\int_\R ( e^{-\frac{\lvert x-v\rvert^2}{C(t-r)}}+ 
e^{-\frac{\lvert x'-v\rvert^2}{C(t-r)}})e^{-\frac{\lvert v-y\rvert^2}{C (r-s)}}\,\dd 
v\,\dd r\\
&\leq C\biggl(\frac{\lvert x-x'\rvert}{t-s}\biggr)^{3-p}	\int_s^{\frac{s+t}2} 
(t-r)^{2-p} (r-s)^{-\frac {p-1}2}(t-s)^{-\frac12}( e^{-\frac{\lvert x-y\rvert^2}{C(t-s)}}+ 
e^{-\frac{\lvert x'-y\rvert^2}{C(t-s)}})\,\dd r\\
&\leq C(t-s)^{-\frac 32}e^{-C^{-1}\lvert y\rvert^2}\lvert x-x'\rvert^{3-p}\int_s^t 
(r-s)^{-\frac{p-1}{2}}\,\dd r,
\end{align*}
while, by distinguishing whether $v\in(-2,2)$ or not and by using \cite[Lemme~A2]{SLB98} and the bound $\lvert\partial_x g(t,x)\rvert\leq  Ce^{-\lvert x\rvert^2/(Ct)}$ for $\lvert x\rvert >1$, we obtain
\begin{align*}
	I_2&\leq C(t-s)^{-\frac p2}\Biggl(e^{-C^{-1}\lvert y\rvert^2}\int_s^t\int_\R \lvert g(t-r,x-v)-g(t-r,x'-v)\rvert^p\,\dd r\,\dd v \\
	&\quad + \lvert x-x'\rvert^p\int_s^t\int_\R e^{-\frac{\lvert v\rvert^2}{C(t-r)}}e^{-\frac{\lvert v-y\rvert^2}{C(r-s)}}\,\dd v\,\dd r\Biggr)\\
	&\leq C(t-s)^{-\frac p2}e^{-C^{-1}\lvert y\rvert^2}\lvert x-x'\rvert^{3-p}.
\end{align*}
Therefore, both $I_1$ and $I_2$ are bounded by the right-hand side of \eqref{eq:incr}, as claimed.

From here, we get a moment bound on the local supremum of $v$ by using    a 
quantitative version of Kolmogorov's continuity theorem (see \cite[Eq.\ (6.7)]{Conus13} 
with $m=0$):
\begin{align*} 
\E\biggl[\sup_{x\in Q} v(s,y;t,x)^p\biggr]  &\leq C\Biggl(\E[v(s,y;t,0)^p ]+ 
\E\Biggl[\sup_{x,x'\in Q} \lvert v(s,y;t,x) - v(s,y;t,x')\rvert^p\Biggr]  \Biggr)\\
	&\leq C(t-s)^{-\frac p2} e^{-\frac{\lvert y\rvert^2}{2 T}}+C (t-s)^{-\frac p2}e^{-C^{-1}\lvert y\rvert^2}\frac{2^{(2p-1)/p}}{1-2^{-(2-p)/p}}\\
	&\leq C (1-\tfrac p2)^{-1}(t-s)^{-\frac p2}e^{-C^{-1}\lvert y\rvert^2}.
\end{align*}
Thus
\[
\E  \left[ \sup_{x \in Q} u_{<}(s,y;t,x)^p \right]
\leq  C (t-s)^{-\frac p2} + \E\biggl[\sup_{x\in Q} v(s,y;t,x)^p\biggr],
\]
and the statement follows.
\epr

For the tail bounds of the local supremum when the noise is relatively light-tailed, we need a preparatory result, which is a decoupling inequality for tail probabilities.
\blem\label{lem:dec} Let $(\calf_t)_{t\geq0}$ be a filtration on a probability space $(\Om,\calf,\P)$ and $\mu$ be an $(\calf_t)_{t\geq0}$-Poisson random measure on $[0,\infty)\times E$, where $E$ is a Polish space, with intensity measure $\nu$. Consider a nonnegative $(\calf_t)_{t\geq0}$-adapted process $H:\Om\times[0,\infty)\times E\to\R$ and a copy $H'$, which is defined on an additional probability space $(\Om',\calf',\P')$ (and therefore independent of $(\calf_t)_{t\geq0}$ on the product space). Let $\ov \P=\P\otimes\P'$ and define the random variables
$$ X= \int_{(0,\infty)\times E} H(t,x) \, \mu(\dd t,\dd x),\qquad X'=\int_{(0,\infty)\times E} H'(t,x) \, \mu(\dd t,\dd x).$$
Then, for any $\theta,R>0$,
\beq\label{eq:dec} \P(X>R)\leq 7\theta\,\P(X>\tfrac 13 R)+2\,\ov\P(X'>\tfrac 16R) +\theta^{-1}\ov\P(X'>\tfrac16\theta R). \eeq
In particular, for any $p>0$,
$$ \exists C>0,\ \forall R>0 : \ov\P(X'>R)\leq CR^{-p} \implies \exists C'>0,\ \forall 
R>0 : \P(X>R)\leq C'R^{-p}.$$
\elem
\bpr
The inequality \eqref{eq:dec} was shown on page 38 of \cite{Kallenberg17}. Therefore, only the last statement needs a proof. If $\ov\P(X'>R)\leq CR^{-p}$ for all $R>0$, \eqref{eq:dec} implies
$$  \P(X>R)\leq 7\theta\,\P(X>\tfrac 13 R)+2( 6^pC) R^{-p}  +6^pC\theta^{-1-p}R^{-p}. $$
Choose $\theta< \frac17 3^{-p}$ and define $K=2 (6^pC)+6^pC\theta^{-1-p}$. Then iterating the previous equation yields, for any $n\in\N$,
\begin{align*}
	\P(X>R)&\leq 7\theta\,\P(X>\tfrac 13 R) + KR^{-p}\leq 7\theta(7\theta\,\P(X>\tfrac 1{3^2} R)+ 3^pKR^{-p})+ KR^{-p}\\
	&=KR^{-p} +  3^p(7\theta) KR^{-p}+(7\theta)^2\P(X>\tfrac1{3^2} R)\\
	&\leq (1+3^p(7\theta))KR^{-p} + (7\theta)^2(7\theta\,\P(X>\tfrac 1{3^3} R) + K3^{2p}R^{-p})\\
	&=(1+3^p(7\theta)+[3^p(7\theta)]^2)KR^{-p} + (7\theta)^3\P(X>\tfrac1{3^3} R)\\
	&\leq \cdots\leq \Biggl(\sum_{j=0}^n [3^p(7\theta)]^j \Biggr)KR^{-p} + (7\theta)^{n+1}\P(X>\tfrac1{3^{n+1}} R).
\end{align*}
Since $\theta< \frac17 3^{-p}$, bounding the last probability by $1$ and letting $n\to\infty$, we conclude that
\[\P(X>R)\leq \frac{K}{1-3^p(7\theta)} R^{-p}.\qedhere\]
\epr

\bthm\label{thm:sup-lt} Assume Condition~\ref{cond:lt} with $\al=\frac 2d$ and Condition~\ref{cond:sup}.  Then, for every $t>0$, there are $C_1,C_2\in(0,\infty)$ such that for all $R>1$ and $Q\in \calq$,
\beq\label{eq:sup-lt} C_1R^{-\frac 2d}\leq \P\biggl(\sup_{x\in Q} Y(t,x)\biggr) \leq C_2R^{-\frac 2d}. \eeq
\ethm

\bpr   \bff{Step 1: Upper bound} 
\smallskip

\noindent   Without loss of generality, we may assume that $Q=(0,1)^d$. By assumption, $m_1(\la)<\infty$ if $d\geq2$. Thus, if we define $\wt Y=Y$ when $d=1$ and $\wt Y=\wh Y=e^{m_1(\la) t} Y$ when $d\geq2$, then it suffices to prove the theorem for $\wt Y$ instead of $Y$. To this end, write $\wt Y$ as a sum $1+Y_1+Y_2+Y_3+Y_4\bone_{\{d=1\}}$, where
 \begin{align*}
 	Y_1(t,x)&=\int_0^t\int_{D}\int_{(0,\infty)} 
g(t-s,x-y)z\bone_{\{(t-s)^{-d/2}z\geq1\}}\wt Y(s,y) \,\mu(\dd s,\dd y,\dd z),\\
 	Y_2(t,x)&=\int_0^t\int_{\R^d\setminus D}\int_{(0,\infty)} 
g(t-s,x-y)z\bone_{\{(t-s)^{-d/2}z\geq1\}}\wt Y(s,y) \,\mu(\dd s,\dd y,\dd z),\\
 	Y_3(t,x)& =\begin{cases}\displaystyle \int_0^t\int_{\R}\int_{(0,\infty)} 
g(t-s,x-y)z\bone_{\{(t-s)^{-1/2}z<1\}}\wt Y(s,y) \,(\mu-\nu)(\dd s,\dd y,\dd z)\\
\hfill  \text{if } d=1,\\
 \displaystyle	\int_0^t\int_{\R^d}\int_{(0,\infty)} 
g(t-s,x-y)z\bone_{\{(t-s)^{-d/2}z<1\}}\wt Y(s,y) \,\mu(\dd s,\dd y,\dd z)\\
\hfill \text{if } d\geq2,\end{cases}	\\
 	Y_4(t,x)&=\int_0^t\int_{\R}\int_{(0,\infty)} g(t-s,x-y)z(\bone_{\{(t-s)^{-d/2}z< 1\}}-\bone_{\{z<1\}})\wt Y(s,y)\,\dd s\,\dd y\,\la(\dd z)
 \end{align*}
and $D=(-2,2)^d$.
We analyze each part separately. 

For $Y_1$,   observe that $\sup_{x\in Q} Y_1(t,x)\leq C X$, 
where
$$ X=\int_0^t\int_{D}\int_{(0,\infty)} (t-s)^{-\frac d2} z\bone_{\{(t-s)^{-d/2}z\geq1\}}\wt Y(s,y)z\,\mu(\dd s,\dd y,\dd z).$$
By Lemma~\ref{lem:dec}, it suffices to bound the tail of $X'$, which is the same integral except that $\wt Y$ in the last display is replaced by $\wt Y'$, an independent copy of $\wt Y$. Further observe that
the number $N_t$ of atoms $(\tau,\eta,\zeta)$ that satisfy $\eta\in D$ and $(t-\tau)^{-d/2}\zeta>1$ is Poisson distributed with parameter
$$ L=4^d\int_0^t\int_{(0,\infty)} \bone_{\{(t-s)^{-d/2}z\geq1\}}\,\dd s\,\la(\dd 
z)=4^d\int_{(0,\infty)} (z^{\frac 2d}\wedge t)\,\la(\dd z)<\infty. $$
So conditionally on $N_t$, the corresponding atoms  $(\tau_i,\eta_i,\zeta_i)_{i=1,\dots,N_t}$ are independent and identically distributed with density  $L^{-1}\bone_{D}(y)\bone_{\{(t-s)^{-d/2}z>1\}}\, \dd s\,\dd y\,\la(\dd z)$. Therefore,
\begin{align}
	\P(X'>R)&=e^{-L}\sum_{N=1}^\infty \frac{L^N}{N!}\E'\Biggl[\P\Biggl(\sum_{i=1}^N (t-\tau_i)^{-\frac d2}\wt Y'(\tau_i,\eta_i) \zeta_i>R\mathrel{\bigg|} N_t=N\Biggr)\Biggr]\nonumber\\
	&\leq e^{-L}\sum_{N=1}^\infty \frac{L^N}{N!}N\E'\biggl[\P\biggl(  (t-\tau_1)^{-\frac d2}\wt Y'(\tau_1,\eta_1) \zeta_1>\tfrac RN\mathrel{\Big|} N_t=N\biggr)\biggr]\label{eq:aux}\\
	&\leq e^{-L}\sum_{N=1}^\infty \frac{L^{N-1}}{(N-1)!}\int_0^t\int_{D}\int_{(0,\infty)} \P'((t-s)^{-\frac d2}\wt Y'(s,y)z>\tfrac RN)\,\dd s\,\dd y\,\la(\dd z).\nonumber
\end{align} 
Choosing $p\in(1 \vee \frac 2d, 1+\frac2d)$ and recalling the notation $\E_<$ and 
$\E_\geq$ from the proof of Theorem~\ref{thm:sol-ht}, we can use Markov's inequality and 
\cite[Lemma~8.1]{Berger21b} to obtain
 \begin{align*}
\P'((t-s)^{-\frac d2}\wt Y'(s,y)z>\tfrac RN)&=\P((t-s)^{-\frac d2}\wt Y(s,y)z>\tfrac RN)\\
&\leq \E_\geq\Big[R^{-p}\E_<[\wt Y(s,y)^p]N^pz^p(t-s)^{-\frac d2 p} \wedge 1\Bigr]\\
&\leq  \E_\geq\Big[CR^{-p}Y''(s,y)^pN^pz^p(t-s)^{-\frac d2 p} \wedge 1\Bigr],
\end{align*} 
where $Y''$ is the solution to \eqref{eq:PAM} driven by $\beta\La_\geq$ and $\beta=\beta(p,T)>0$ is a constant.

Since $\E[X\wedge1]= \int_0^1 \P(X>u) \,\dd u$,   
\[
\P'((t-s)^{-\frac d2}\wt Y'(s,y)z>\tfrac RN)
\leq  C \int_0^1 \P(Y''(s,y)   > \tfrac{Ru^{1/p}}{Nz}(t-s)^{\frac d2}) \,\dd u  .
\]
Now observe that $Y''$ is a series of multiple stochastic integrals with respect to the positive measure $\beta\La_\geq$. Together with the stationarity of $\beta\La_\geq$, it follows that $Y''(s,y)$ is stochastically dominated by $Y''(t,0)$. Thus, replacing $Y''(s,y)$ by $Y''(t,0)$ only increases the probabilities in the last display. Making this modification and 
inserting the resulting bound back into \eqref{eq:aux}, we can change variables  $s\mapsto r=(t-s)^{d/2}Ru^{1/p}/(Nz)$ to obtain
\beq\label{eq:Xprime}\begin{split}
	\P(X'>R)&\leq   C\mu_{2/d}(\la)R^{-\frac 2d}\Biggl(\sum_{N=1}^\infty \frac{L^{N-1}N^{2/d}}{(N-1)!}\Biggr) \int_0^1 u^{-\frac 2{dp}}\, \dd u \int_0^\infty \P(Y''(t,0)>r)   r^{\frac 2d-1}  \,\dd r  \\
	&\leq C\mu_{2/d}(\la)R^{-\frac 2d}. 
\end{split} \raisetag{-2.5\baselineskip}\eeq
Note that the $\dd u$-integral is finite  since $p>\frac2d$ and that the last integral is just a multiple of $\E[Y''(t,0)^{2/d}]$ and hence also   finite because $\mu_{2/d}(\la)$ is. 

Next, we consider $Y_2$. Because $x\in(0,1)^d$ and $y\notin (-2,2)^d$, the distance 
$\lvert x-y\rvert$ is bounded from below by $1$. Therefore, $g(t-s,x-y)\leq 
Ce^{-C^{-1}\lvert y\rvert^2}$ for some $C\in(0,\infty)$. Since this removes the 
singularity of $g$ around $0$ as well as the dependence on $x$, it is easy to show that 
$\sup_{x\in Q} Y_2(t,x)$ has a uniformly bounded moment of order $\frac 2d$  on $[0,t]$. Thus, the 
tail of $\sup_{x\in Q} Y_2(t,x)$ is lighter and hence negligible in \eqref{eq:sup-lt}.

The term $Y_4$ is only present if $d=1$. In this case, we use the bound
\begin{multline*} \E\biggl[\sup_{x\in Q} \lvert Y_4(t,x)\rvert^2\biggr]^{\frac 12} \leq 
C\int_0^t\int_{\R}\int_{(0,\infty)} (t-s)^{-\frac 12}e^{-C^{-1}\lvert y\rvert^2} z \\
\times
\lvert\bone_{\{s^{-1/2}z< 1\}}-\bone_{\{z<1\}}\rvert\E[\lvert \wt 
Y(s,y)\rvert^2]^{\frac12}\,\dd s\,\dd y\,\la(\dd z).
\end{multline*}
Evaluating the difference of indicator functions in the last line we  bound this by
\begin{align*}& C 
\int_0^t\int_{(0,\infty)} s^{-\frac 12}z\bone_{\{1\leq z<s^{1/2}\}}\,\dd s\,\la(\dd 
z)+C\int_0^t\int_{(0,\infty)} s^{-\frac 12}z\bone_{\{s^{1/2}\leq z<1\}}\,\dd s\,\la(\dd 
z)\\
	&\qquad\leq C
   \int_1^{t\vee1} s^{-\frac 12}\int_{[1,s^{1/2})}z\,\la(\dd z)\,\dd s +C \int_{(0,1)} z\int_0^t s^{-\frac12} \bone_{\{s<z^{2}\}}\,\dd s\,\la(\dd z)\\
&\qquad\leq C(M_1(\la)+m_2(\la)),
\end{align*}
which is finite. Thus, $\sup_{x\in Q} \lvert Y_4(t,x)\rvert$ does not contribute to the tail in \eqref{eq:sup-lt}, either.

For the last remaining term  $Y_3$, if $d=1$ use  \cite[Thm.~1]{MR} (with $\alpha = p 
= 2$) and Minkowski's integral inequality to obtain
 \begin{multline*}
\E[\lvert Y_3(t,x)-Y_3(t,x')\rvert^2]\leq 
C  \int_0^t\int_{\R}\int_{(0,\infty)} \lvert g(t-s,x-y)-g(t-s,x'-y)\rvert^2  \\
 \times\E[\wt Y(s,y)^2] z^2\bone_{\{(t-s)^{-1/2}z<1\}}\,\dd s\,\dd y\,\la(\dd z)
\end{multline*} 
for all $x,x'\in\R$. Since $\E[\wt Y(s,y)^2]$ is uniformly bounded on $[0,t]\times\R$, it 
can be absorbed into the constant $C$. Observe that 
$(t-s)^{-1/2}z<1$ implies 
\begin{align*}\lvert g(t-s,x-y)-g(t-s,x'-y)\rvert^2 z^2&
\leq C(t-s)^{-d}z^2\Bigl\lvert e^{-\frac{\lvert x-y\rvert^2}{2 (t-s)}} - 
e^{-\frac{\lvert x'-y\rvert^2}{2 (t-s)}} \Bigr\rvert^2\\
&\leq C(t-s)^{-\frac d2 q}z^q\Bigl\lvert e^{-\frac{\lvert x-y\rvert^2}{2 (t-s)}} - 
e^{-\frac{\lvert x'-y\rvert^2}{2 (t-s)}} \Bigr\rvert^q\\
&=C\lvert g(t-s,x-y)-g(t-s,x'-y)\rvert^qz^q,
 \end{align*}
where $q < 2$ is the exponent from Condition~\ref{cond:sup}. With 
this estimate and again  \cite[Lemme~A2]{SLB98}, we conclude that
$$ \E[\lvert Y_3(t,x)-Y_3(t,x')\rvert^2]\leq 
C\Bigl( \lvert x-x'\rvert^{3-q} \Bigr).$$
Since $3-q >1$,  it follows from \cite[Thm.~4.3]{Khos09} that
$$ \E\biggl[\sup_{x\in Q} Y_3(t,x)^2\biggr]\leq \E[Y_3(t,0)^2] + 
\E\biggl[\sup_{x,x'\in 
Q}\lvert Y_3(t,x)-Y_3(t,x')\rvert^2\biggr]<\infty,  $$
which shows that $\P(\sup_{x\in Q} Y_3(t,x)>R) =o(R^{-2})$.

If $d=2$, we simply bound
\begin{align*} 
	 \E\biggl[\sup_{x\in Q} Y_3(t,x)\biggr]&\leq C\int_0^t \int_{\R^d} \int_{(0,\infty)} (t-s)^{-1}e^{-C^{-1}\lvert y\rvert^2}z \bone_{\{(t-s)^{-1}z<1\}}\,\dd s\,\dd y\,\la(\dd z)\\
	 &\leq \int_{(0,t)} z \int_{z}^t s^{-1}\,\dd s\,\la(\dd z)\leq C\int_{(0,t^{d/2})} z(1+\lvert\log z\rvert) \,\la(\dd z),  
 \end{align*}
which shows $\P(\sup_{x\in Q} Y_3(t,x)>R)=o(R^{-1})$.

If $d\geq3$, we write $Y_3(t,x)=\sum_{i=0}^\infty Y_{3,i}(t,x)$ where 
\begin{align*}
Y_{3,0}(t,x)&=\int_0^t\int_{\R^d\setminus (-2,2)^d}\int_{(0,\infty)} 
g(t-s,x-y)z\bone_{\{(t-s)^{-d/2}z<1\}}\wt Y(s,y) \,\mu(\dd s,\dd y,\dd z),\\
Y_{3,i}(t,x)&=\int_0^t\int_{(-2,2)^d}\int_{(0,\infty)} 
g(t-s,x-y)z\bone_{\{(t-s)^{-d/2}z<1\}}\wt Y(s,y) \,\mu_i(\dd s,\dd y,\dd z),
\end{align*}
where the $\mu_i$'s are independent Poisson random measures (with intensities $\dd t\,\dd x\,\la_i(\dd z)$) such that $\mu=\sum_{i=1}^\infty \mu_i$ and such that  $4^d\int_0^t\int_{(0,\infty)} \bone_{\{(t-s)^{-d/2}z<1\}} \,\dd s\,\la_i(\dd z)\leq2$. Such a decomposition is indeed possible, see Lemma~\ref{lem:decom}. In the same way as we did for $Y_1$ and $Y_2$, one can now show that 
$$ \P\biggl( \sup_{x\in Q} Y_{3,0}(t,x) > R\biggr) =o(R^{-\frac 2d}),\qquad \P\biggl( \sup_{x\in Q} Y_{3,i}(t,x) > R\biggr) \leq C\mu_{2/d}(\la_i)R^{-\frac 2d}$$
for some $C>0$ that does not depend on $i$. Borrowing a truncation trick from the proof of \cite[Lemma~4.24]{Resnick08}, we now bound
\begin{align*}
&\P\biggl( \sup_{x\in Q} \sum_{i=1}^\infty Y_{3,i}(t,x) > R\biggr)\\&\qquad\leq   \sum_{i=1}^\infty \P\biggl( \sup_{x\in Q}  Y_{3,i}(t,x) > R\biggr) + \P\Biggl( \sum_{i=1}^\infty \sup_{x\in Q}  Y_{3,i}(t,x) \bone\biggl\{\sup_{x\in Q}  Y_{3,i}(t,x)\leq R\biggr\}> R\Biggr)\\
&\qquad\leq C\mu_{2/d}(\la)R^{-\frac 2d} + \frac1R\sum_{i=1}^\infty \E\Biggl[ \sup_{x\in Q}  Y_{3,i}(t,x) \bone\biggl\{\sup_{x\in Q}  Y_{3,i}(t,x)\leq R\biggr\}\Biggr].
\end{align*}
As $ \E[X\bone_{\{X\leq R\}}]\leq\E[ R\wedge X]= \int_0^R \P(X>u)\,\dd u$, it follows that
\begin{align*}
\P\biggl( \sup_{x\in Q} \sum_{i=1}^\infty Y_{3,i}(t,x) > R\biggr)&\leq C\mu_{2/d}(\la)R^{-\frac 2d} + \frac1R\sum_{i=1}^\infty \int_0^R \P\biggl( \sup_{x\in Q}  Y_{3,i}(t,x)>u \biggr)\,\dd u\\
&\leq C\mu_{2/d}(\la)\Biggl(R^{-\frac 2d} + \frac1R\int_0^R u^{-\frac 2d}\,\dd u\Biggr)\leq C\mu_{2/d}(\la)R^{-\frac 2d}.
\end{align*}

\noindent\bff{Step 2: Lower bound} 
\smallskip

\noindent  
Without loss of generality, we may assume that $M_0(\la)=\la([1,\infty))>0$ and that $Q=(0,1)^d$. By \eqref{eq:separate} and the positivity of $u_<(t_{i-1},x_{i-1};t_i,x_i)$, 
$$ \sup_{x\in Q} Y(t,x)\geq \sup_{x\in Q} \int_0^t\int_\R 
u_<(s,y;t,x)Y_<(s,y)\,\La_\geq(\dd s,\dd y) \geq  \sup_{x\in Q} u_<(\tau,\eta;t,x) 
Y_<(\tau,\eta)$$
on the event $A=\{\La_\geq([0,t]\times Q\times [1,\infty)) = 1\}$.
Since $\P(A)=tM_0(\la) e^{-tM_0(\la)}$, we have
\begin{align*} \P\biggl( \sup_{x\in Q} Y(t,x) > R\biggr)&\geq  
\mathtoolsset{multlined-width=0.7\displaywidth} \begin{multlined}[t]  
tM_0(\la) e^{-tM_0(\la)} \P\biggl(\sup_{x\in Q}  u_<(\tau,\eta;t,x)Y_<(\tau,\eta) > R,
 \\
Y_<(\tau,\eta)>\tfrac12,\	u_<(\tau,\eta;t,x)>\tfrac12 g(t-\tau,x-\eta)\mathrel{\Big|}A\biggr)\end{multlined}\\
&\geq  \mathtoolsset{multlined-width=0.7\displaywidth} \begin{multlined}[t]  
tM_0(\la) e^{-tM_0(\la)} \P\biggl(\sup_{x\in Q}  g_<(t-\tau,x-\eta) >4 R,\ 
Y_<(\tau,\eta)>\tfrac12, \\
	u_<(\tau,\eta;t,x)>\tfrac12 g(t-\tau,x-\eta)\mathrel{\Big|}A\biggr). \end{multlined}
\end{align*}
With a similar argument as in \eqref{eq:prob}, it follows that
\begin{align*} \P\biggl( \sup_{x\in Q} Y(t,x) > R\mathrel{\Big|}A\biggr)&\geq C\P\biggl(\sup_{x\in Q}  g(t-\tau,x-\eta) >4 R\mathrel{\Big|}A\biggr)\\&=C\P ((2\pi (t-\tau))^{-\frac d2}>4R\mid A).
\end{align*}
The lower bound in \eqref{eq:sup-lt} now follows from the fact that $\tau$ has a uniform distribution on $[0,t]$.
\epr

\section{The spatial peaks of the solution}\label{sec:peaks}

Armed with the  probability tail bounds from the previous section, we can now state and prove extensions of Theorems~\ref{thm:A}--\ref{thm:C} to general multiplicative Lévy noises. 
\bthm\label{thm:peak-R-lt} Fix $t>0$ and let $Y$ be the mild solution to \eqref{eq:PAM}. Assume Condition~\ref{cond:sup} and Condition~\ref{cond:lt} with $\al=\frac2d$. If $d=1$ and $m_1(\la)=\infty$, assume Condition~\ref{cond:lt} with some $\al>2$. Then the statement of Theorem~\ref{thm:A} remains true.
\ethm

If $d=1$ and $\La$ has infinite variation jumps, we need $\la$ to have a finite moment of order slightly bigger than $2$, in particular, in order to derive \eqref{eq:aux2} below. We strongly believe that it is not necessary for Theorem~\ref{thm:peak-R-lt} to hold.

\bpr[Proof of Theorem~\ref{thm:peak-R-lt}] Let us suppose that the integral converges. For $r>0$ and $0<r_1<r_2<\infty$, let $B_\infty(r)=\{x\in\R^d: \max_{i=1,\dots,d} \lvert x_i\rvert\leq r\}$ and $B_\infty(r_1,r_2)=B_\infty(r_2)\setminus B_\infty(r_1)$. Then, for any $K>0$, 
\begin{align}\nonumber
	\P\Biggl(\sup_{x\in B_\infty(n,n+1)} Y(t,x)>\frac{f(n)}{K}\Biggr)&\leq Cn^{d-1}\sup_{Q_0\in \calq, \, Q_0\subseteq B_\infty(n,n+1)} \P\Biggl(\sup_{x\in Q_0} Y(t,x)>\frac{f(n)}{K}\Biggr)\\
	&\leq CK^{-\frac 2d}n^{d-1}f(n)^{-\frac 2d},\label{eq:bound2} 
\end{align}
where the first step holds because the number of cubes from $\calq$  intersecting $B_\infty(n,n+1)$ is $O(n^{d-1})$ and the second step follows from Theorem~\ref{thm:sup-lt}. By the integral test for convergence, these probabilities are summable, so by the first Borel--Cantelli lemma, 
\[ \sup_{x\in B_\infty(n,n+1)} Y(t,x) \leq \frac{f(n)}{K}\]
for all but finitely many $n$, almost surely. This shows
\[ \limsup_{x\to\infty} \frac{\sup_{\lvert y\rvert\leq x} Y(t,y)}{f(x)}\leq K^{-1}\]
and hence the claim because $K>0$ was arbitrary.

For the converse, there is no loss of generality if we assume that $\la([1,\infty))>0$. If $d\geq2$, recall that $m_1(\la)<\infty$, so by \eqref{eq:Yhat}, $Y(t,x)=e^{-m_1(\la)t} \wh Y(t,x)\geq e^{-m_1(\la)t}   \int_0^t\int_{\R^d} g(t-s,x-y)  \,\La(\dd s,\dd y)$, which is a multiple of the solution to the heat equation with additive Lévy noise. Hence, the result follows from \cite{islands_add}. The same argument applies \ in $d=1$ if $m_1(\la)<\infty$.

If $d=1$ and $m_1(\la)=\infty$, the proof is more technical due to infinite variation jumps. We assume without loss of generality that $f$ is smooth.
With the same notation as in the upper bound proof of Theorem~\ref{thm:sup-lt}, we have 
\beq\label{eq:Ys} Y(t,x)\geq Y_0(t,x)+Y_3(t,x)+Y_4(t,x),\eeq
where
\beq\label{eq:Y0} Y_0(t,x)=\int_0^t\int_{x-1}^{x+1}\int_0^\infty g(t-s,x-y)z\bone_{\{(t-s)^{-1/2}z\geq1\}} Y(s,y)\,\mu(\dd s,\dd y,\dd z). \eeq
Next, let $\vp(x)=\sqrt{x+f(x)^2}$ and $h$ be an increasing function to be determined later and define 
$$ m=m(n)=h(n), \quad \beta=\beta(n)=h(n)^{1/4},\quad  R=R(n)=K {\vp(nh(n))}.$$
Since $x\vee f(x)^2\leq \vp(x)^2\leq 2( x\vee f(x)^2)$, the divergence of the integral in \eqref{eq:int2} implies
\beq\label{eq:phiint} \int_1^\infty \frac{1}{\vp(x)^2}\,\dd x=\infty\eeq
by \cite[Lemma~3.4 (2)]{CK2}.
We shall approximate $Y_0$ by
\beq\label{eq:Zmb} Z^{(m,\beta)}(t,x)= \int_0^t\int_{x-1}^{x+1}\int_0^\infty g(t-s,x-y)z\bone_{\{(t-s)^{-1/2}z\geq1\}} \lvert Y^{(m,\beta)}(s,y)\rvert \,\mu(\dd s,\dd y,\dd z),\eeq
where $Y^{(m,\beta)}$ is defined in Lemma~\ref{lem:tech}.
More precisely, writing
$I_n=(n h(n)-1,nh(n))$, we want to prove that
\beq\label{eq:BC}\begin{split}
\sum_{n=1}^\infty \P\Biggl(\sup_{x\in I_n} Z^{(m(n),\beta(n))}(t,x)>R(n)\Biggr) &=\infty,\\ 
\sum_{n=1}^\infty \P\Biggl(\sup_{x\in I_n} \lvert Y_0(t,x)-Z^{(m(n),\beta(n))}(t,x)\rvert>\tfrac14 R(n)\Biggr) &<\infty,\\
   \sum_{n=1}^\infty \P\Biggl(\sup_{x\in I_n} \lvert Y_3(t,x)+Y_4(t,x)\rvert>\tfrac14 R(n)\Biggr)&<\infty.  \end{split}\eeq
By the first Borel--Cantelli lemma, this implies that 
\beq\label{eq:BC1} \sup_{x\in I_n} \lvert Y_0(t,x)-Z^{(m(n),\beta(n))}(t,x)\rvert\leq \tfrac14 R(n),\qquad \sup_{x\in I_n} \lvert Y_3(t,x)+Y_4(t,x)\rvert\leq\tfrac14 R(n)\eeq
for all but finitely many $n$, almost surely. Moreover, by \eqref{eq:Zmb} and the definition of $Y^{(m,\beta)}$, the variables $\sup_{x\in I_n} Z^{(m(n),\beta(n))}(t,x)$ are measurable with respect to the $\si$-field generated by the restriction of $\La$ on $[0,t]\times (nh(n)-2-2\beta(n)(tm(n))^{1/2},nh(n)+1+2\beta(n)(tm(n))^{1/2})$, because only atoms of $\La$ that are within a distance of $\sum_{i=1}^m \lvert \Delta y_i\rvert\leq m^{1/2}(\sum_{i=1}^m \lvert \Delta y_i\rvert^2)^{1/2}\leq \beta(sm)^{1/2}$ from $x$ contribute to $Y^{(m,\beta)}(t,x)$ and because the same holds true for $Y^{(m,\beta)}_<(t,x)$ and $u_<^{(m,\beta)}(s,y;t,x)$. Since $\beta(n)(tm(n))^{1/2} =o(h(n))$, the considered variables are independent for different $n$. By the second Borel--Cantelli lemma, \eqref{eq:BC} also implies that
$$ \sup_{x\in I_n} Z^{(m(n),\beta(n))}(t,x)>R(n)$$
for infinitely many $n$, almost surely. Combining this with \eqref{eq:BC1}, it follows that
$$ \limsup_{n\to\infty} \frac{\sup_{\lvert x\rvert \leq nh(n)} Y(t,x)}{f(nh(n))}\geq \limsup_{n\to\infty} \frac{\sup_{\lvert x\rvert \leq nh(n)} Y(t,x)}{\vp(nh(n))} \geq \frac{K}{2}$$
almost surely. As $K$ is arbitrary, the claim follows.

To prove \eqref{eq:BC}, we start with the first statement and notice that
\begin{multline*} \P\Biggl(\sup_{x\in I_n} Z^{(m(n),\beta(n))}(t,x) >R(n)\Biggr)\\\geq \P\Biggl(\bigvee_{i=1}^N (2\pi(t-\tau_i))^{-\frac12}\lvert Y^{(m(n),\beta(n))}(\tau_i,\eta_i)\rvert\zeta_i>R(n)\Biggr),  \end{multline*}
where $(\tau_i,\eta_i,\zeta_i)_{i=1}^N$ are atoms of $\La_1$, where $\La_1$ is the 
restriction of $\La$ to the set $\{(s,y,z)\in(0, t)\times(x-1,x+1)\times(0,\infty): 
(t-s)^{-1/2}z>1\}$ (with the convention $\bigvee_{i=1}^0 a_i =0$). By 
\cite[Lemma~4.1]{Kallenberg17},
\begin{multline*} \P\Biggl(\sup_{x\in I_n} Z^{(m(n),\beta(n))}(t,x) >R(n)\Biggr)\\\geq \frac12\,\P\otimes\P'\Biggl(\bigvee_{i=1}^N (2\pi(t-\tau_i))^{-\frac12}\lvert   Y^{\prime (m(n),\beta(n))}(\tau_i,\eta_i)\rvert\zeta_i>R(n)\Biggr),  \end{multline*}
where $Y^{\prime(m,\beta)}$ is a copy of $Y^{(m,\beta)}$ that is independent of $\La$ (and defined on an auxiliary probability space $(\Om',\calf',\P')$). Hence, the right-hand side of the previous display is bounded from below by
$$ \frac12\,\P(N\geq1)\P\otimes\P'\Biggl ( (2\pi(t-\tau_1))^{-\frac12} \zeta_1>2R(n),\ \lvert   Y^{\prime (m(n),\beta(n))}(\tau_1,\eta_1)\rvert >\frac12\Biggr).  $$
As in \eqref{eq:prob}, we have $\P'(\lvert   Y^{\prime (m,\beta)}(t,x)\rvert >\frac12)>C$ locally uniformly in $t$ and $x$ (and uniformly in $m$ and in $\beta$ outside a neighborhood of   $0$). Therefore,
$$ \P\Biggl(\sup_{x\in I_n} Z^{(m(n),\beta(n))}(t,x) >R(n)\Biggr)\geq \frac12 C \P((2\pi(t-\tau_1))^{-\frac12} \zeta_1>2R(n))\geq C R(n)^{-  2},  $$
where the last step follows by a quick computation. Thus, the first relation in \eqref{eq:BC} is satisfied if 
\beq\label{eq:sum1} \sum_{n=1}^\infty \frac{1}{\vp(nh(n))^2}=\infty \iff  \int_1^\infty \frac{1}{\vp(xh(x))^2}\,\dd x=\infty. \eeq

Let us move on to the second relation in \eqref{eq:BC}. By Lemma~\ref{lem:Y0Z},
\beq\label{eq:aux2} \P\Biggl(\sup_{x\in I_n} \lvert Y_0(t,x)-Z^{(m(n),\beta(n))}(t,x)\rvert>\tfrac14 R(n)\Biggr)\leq 16\,C_{m(n),\beta(n)}R(n)^{-2},\eeq
where $C_{m,\beta}$ is the constant from the lemma. With our choices of $m$ and $\beta$, we have $C_{m(n),\beta(n)}=O(e^{-C^{-1}h(n)^{1/4}})$, so the second line in \eqref{eq:BC} is implied by
\beq\label{eq:sum2} \sum_{n=1}^\infty \frac{e^{-C^{-1}h(n)^{1/4}}}{\vp(nh(n))^2}<\infty \iff \int_1^\infty  \frac{e^{-C^{-1}h(x)^{1/4}}}{\vp(xh(x))^2}\,\dd x<\infty. \eeq
Lastly, because $\sup_{x\in I_n} \lvert   Y_3(t,x)+Y_4(t,x)\rvert$ has uniformly (in $n$) bounded moments of order $\al>2$,  we have   that the last relation in \eqref{eq:BC} is implied by
$$ \sum_{n=1}^\infty \frac{1}{\vp(nh(n))^\al}<\infty \iff \int_1^\infty  \frac{1}{\vp(xh(x))^\al}\,\dd x<\infty. $$
And this is true, because $\vp(xh(x))\geq \vp(x)\geq \sqrt{x}$. Consequently, in order to complete the proof, it remains to choose $h$ such that \eqref{eq:sum1} and \eqref{eq:sum2} are satisfied.

To this end, we will restrict our choice of $h$ to the class of increasing smoothly varying functions of index $0$ (see \cite[Ch.~1.8]{BGT}). In this case, if we change the variable $x$ to $y=g(x)=xh(x)$, then there exist $y_0>0$ and a smoothly varying function $h^\#$ with index $0$ (the de Bruijn conjugate of $h$) such that $g^{-1}(y)=yh^\#(y)$ for all $y>y_0$ (see \cite[Thm.~1.8.9]{BGT}). Moreover, for $y>y_0$, we have $(g^{-1})'(y)=h^\#(y)+y(h^\#)'(y)$ and $y(h^\#)'(y) = o(h^\#(y))$ (see \cite[p.\ 44]{BGT}). Therefore the conditions in \eqref{eq:sum1} and \eqref{eq:sum2} are equivalent to having both
$$ \int_{y_0}^\infty \frac{h^\#(y)}{\vp(y)^2}\,\dd y=\infty\qquad \text{and}\qquad \int_{y_0}^\infty \frac{e^{-C^{-1}h^\#(y)^{-1/4}}h^\#(y)}{\vp(y)^2}\,\dd y<\infty.$$
Note that $h^\#$ is decreasing (as $h$ is increasing). Thus, the previous line is implied by
\beq\label{eq:sum3} \int_{y_0}^\infty \frac{h^\#(y)}{\vp(y)^2}\,\dd y=\infty\qquad \text{and}\qquad \int_{y_0}^\infty \frac{h^\#(y)^2}{\vp(y)^2}\,\dd y<\infty.\eeq
To achieve this, we  choose  $h$ as the de Bruijn conjugate of 
$$ h^\#(y)=\Biggl(\int_{\frac12}^y  \vp(u)^{-2}\,\dd u\Biggr)^{-1},\qquad y>1. $$
Indeed, $h^\#$ is smoothly varying with index $0$ as it has the Karamata representation 
$ h^\#(y)= c \exp(-\int_{1}^y \eps(t)\,\frac{\dd t}{t})$,
where $$ \lim_{t\to\infty} \eps(t)=  \lim_{t\to\infty} \frac{t}{\vp(t)^2\int_{1}^t  
\vp(u)^{-2}\,\dd u}=  \lim_{t\to\infty} \frac{1}{[1+(f(t)^2)']\int_{1}^t  
\vp(u)^{-2}\,\dd u+1}=0$$
by \eqref{eq:phiint}. Finally, $h^\#$ satisfies  \eqref{eq:sum3} by \eqref{eq:phiint} and the Abel--Dini theorem \cite[p.\ 290]{Knopp90}.
\epr


\bthm\label{thm:peak-R-ht} Fix $t>0$ and let $Y$ be the mild solution to \eqref{eq:PAM}. If Condition~\ref{cond:ht} for some $\al\in(0,\frac2d]$ and Condition~\ref{cond:sup} are satisfied, then  Theorem~\ref{thm:B} (ii) remains true.
\ethm
\bpr Define (with the convention $\inf\emptyset = \infty$)
\beq\label{eq:L}  L_0 =  \inf\biggl\{L\in(0,\infty) \ : \ \limsup_{R\to\infty}R^{\al} e^{-L(\log R)^{1/(1+\theta_\al)}} \P\biggl(\sup_{x\in(0,1)^d} Y(t,x)>R\biggr)<\infty  \biggr\} \!\!\!\!
  \eeq
and $L^\ast=\al^{-1}(\frac{d}{\al})^{1/(1+\theta_\al)}L_0$. By Theorem~\ref{thm:sup-ht}, we have $L_0<\infty$. For any $L>L^\ast$ and any fixed $\eps>0$ (to be chosen later), combining the upper bound in Theorem~\ref{thm:sup-ht} with a similar argument to \eqref{eq:bound2} shows that 
\begin{align*}
	&\P\Biggl(\sup_{x\in B_\infty(n,n+1)} Y(t,x)>\frac{n^{d/\al} e^{L(\log n)^{1/(1+\theta_\al)}}}{K}\Biggr)\\
	&\qquad\leq Cn^{-1} \exp(-L\al(\log n)^{1/(1+\theta_\al)}+(L_0+\eps)[\log (K^{-1}n^{d/\al} e^{L(\log n)^{1/(1+\theta_\al)}})]^{1/(1+\theta_\al)})\\
	&\qquad\leq Cn^{-1}\exp(-[L\al-(L_0+\eps)(\tfrac d\al+\eps)^{1/(1+\theta_\al)}](\log n)^{1/(1+\theta_\al)})
\end{align*}
for sufficiently large $n$ (note that $C$ in the previous display may depend on $\eps$, $L$ and $K$ but not on $n$). By our assumption on $L$, if $\eps$ is small enough, we  have $L\al-(L_0+\eps)(\tfrac d\al+\eps)^{1/(1+\theta_\al)}>0$. Therefore, the probabilities in the previous display are summable, so by the Borel--Cantelli lemma,
$$ \limsup_{x\to\infty}\frac{\sup_{\lvert y\rvert\leq x} Y(t,y)}{x^{d/\al} e^{L(\log x)^{1/(1+\theta_\al)}}}\leq \frac{1}{K} $$
almost surely. Therefore, \eqref{eq:ub1}   follows by letting $K\to\infty$. Equation \eqref{eq:lb1} is a direct consequence of Theorem~\ref{thm:peak-Z} (i) (with $L_\ast=M_\ast$), which we state and prove next.
\epr

\bthm\label{thm:peak-Z} Fix $t>0$ and let $Y$ be the mild solution to \eqref{eq:PAM}.
\benu
\item[(i)] If Condition~\ref{cond:ht} holds for some $\al\in(0,1+\frac2d)$,  then the statement of Theorem~\ref{thm:C} (ii) remains valid.
\item[(ii)] If Condition~\ref{cond:ht} or  \ref{cond:lt} holds with   $\al=1+\frac2d$, then the statement of Theorem~\ref{thm:C} (iii) remains valid.
\eenu
\ethm

\bpr We assume without loss of generality that $\la([1,\infty))>0$. Let $C_0$ be the number from \eqref{eq:C0} in the case of (i) and $C_0=1$ in the case of (ii). Furthermore, recall the definition of $Y^{(N)}(t,x)$ from Lemma~\ref{lem:tech}, which satisfies $Y(t,x)\geq Y^{(N)}(t,x)$.
 If $\la$ satisfies the conditions of part (i),  define (with the convention $\sup \emptyset = 0$)
\beq\label{eq:M1}\begin{split} 
		M_1 &=  \sup\biggl\{M\in(0,\infty) \ : \   \liminf_{R\to\infty} R^{\al} e^{-M(\log R)^{1/(1+\theta_\al)}}\P(Y^{(\lfloor C_0\log R\rfloor) )}(t,0)>R)>0 \biggr\},\\
		M_2 &=  \inf\biggl\{M\in(0,\infty) \ : \ \limsup_{R\to\infty} R^{\al} e^{-M(\log R)^{1/(1+\theta_\al)}}\P(Y(t,0)>R)<\infty \biggr\}
\end{split} \eeq
and let $M_\ast = \al^{-1}(\frac{d}{\al})^{1/(1+\theta_\al)}M_1$ and $M^\ast=\al^{-1}(\frac{d}{\al})^{1/(1+\theta_\al)}M_2$; if  $\la$ satisfies the conditions of part (ii), define
\beq\label{eq:M2}\begin{split} 
	M_1 &=  \sup\biggl\{M\in(0,\infty) \ : \ \liminf_{R\to\infty} R^{1+\frac2d} e^{-M\frac{(\log R)(\log\log\log R)}{\log \log R}}\P(Y^{(  C_0\log R  )}(t,0)>R)>0 \biggr\},\\
	M_2 &=  \inf\biggl\{M\in(0,\infty) \ : \ \limsup_{R\to\infty} R^{1+\frac2d} e^{-M\frac{(\log R)(\log\log\log R)}{\log \log R}} \P(Y(t,0)>R)<\infty \biggr \}
\end{split} \eeq
and let $M_\ast=\frac{d^3}{(d+2)^2} M_1$ and $M^\ast= \frac{d^3}{(d+2)^2} M_2$  instead. Both \eqref{eq:ub2} and \eqref{eq:ub3} can be shown similarly to \eqref{eq:ub1}, so we leave the details to the reader. Also, the proofs of \eqref{eq:lb2} and \eqref{eq:lb3} are similar, so we only give the details for the former and assume   Condition~\ref{cond:ht} for some $\al\in(0,1+\frac2d)$.

In the lower bound proof of Theorem~\ref{thm:sol-ht}, we have seen that $M_1>0$. Let $N(R)=C_0\log R$, $m(R)=\log R$ and $\beta(R)=(\log R)^{2}$. Then, by Lemma~\ref{lem:tech}, we also have
\beq\label{eq:liminf}  \liminf_{R\to\infty} R^{\al} e^{-M_0(\log R)^{1/(1+\theta_\al)}}\P(Y^{(N(R),m(R),\beta(R))}(t,x)>R)>0 \eeq
for all $0<M_0<M_1$ and $x\in\R^d$. Similarly to what we observed in the proof of 
Theorem~\ref{thm:peak-R-lt}, the variable $Y^{(N(R),m(R),\beta(R))}(t,x)$ is measurable 
with respect to the $\sigma$-algebra generated by
$\La$ restricted to a ball of radius $(tN(R))^{1/2}+\beta(R)(tm(R))^{1/2}=O((\log 
R)^{5/2})$ around $x$. Therefore, if we let $f(n)=n^{d/\al}e^{M(\log 
n)^{1/(1+\theta_\al)}}$ (with $M<M_\ast$) and $R=R(n)=Kf(n\log^3 n)$ for $n,K\in\N$ and 
distribute $k(n)= c n^{d-1}$ many points, $c > 0$, from $\Z^d$ to the annulus 
$B_\infty(n\log^3 n-1,n\log^3 n)$ such that these points, say, 
$x^{(n)}_1,\dots,x^{(n)}_{k(n)}$ are at least $\log^3 n$ apart from each other, then all 
but finitely many of the variables
	$$ \Bigl\{Y^{(n)}_i= Y^{(N(R(n)),m(R(n)),\beta(R(n)))}(t,x^{(n)}_i)\ : \ i=1,\dots,k(n),\ n\in\N\Bigr\}$$
are independent of each other. Moreover, by \eqref{eq:liminf}, for any $\eps>0$  there is $C>0$ such that 
\begin{align*} &\sum_{n=1}^\infty\sum_{i=1}^{k(n)} \P( Y^{(n)}_i>R(n)) \geq \frac{C}{K^\al}\sum_{n=1}^\infty n^{d-1} R(n)^{-\al}e^{(M_1-\eps)(\log R(n))^{1/(1+\theta_\al)}}\\
	&\qquad\geq \sum_{n=1}^\infty n^{-1}(\log n)^{-3d}e^{-M\al(\log (n\log ^3 n))^{1/(1+\theta_\al)}}e^{(M_1-\eps)(\log R(n))^{1/(1+\theta_\al)}}.
\end{align*}
As $M<M_\ast$, if $\eps$ is small enough, the last series is infinite, so by the second Borel--Cantelli lemma, $Y^{(n)}_i>R(n)$ for infinitely many $n$ and $i$. At the same time, by Lemma~\ref{lem:tech}, our choice of $\beta=\beta(R(n))$ and $m=m(R(n))$ and the first Borel--Cantelli lemma, the events $\{\lvert Y^{(N(R(n)))}(t,x^{(n)}_i)- Y^{(N(R(n)),m(R(n)),\beta(R(n)))}(t,x^{(n)}_i)\rvert > \frac12 R(n)\}$ only occur finitely many times, which implies
$$ \limsup_{n\to\infty} \frac{\sup_{x\in\Z^d,\lvert y\rvert_\infty \leq n\log^3 n } Y(t,x)}{f(n\log^3 n)}\geq \frac K2  $$
almost surely. Because $K\in\N$ was arbitrary, this implies \eqref{eq:ub2}. 
\epr

\section{Macroscopic dimension of peaks}\label{sec:dim}

As another application of the tail estimates of Section~\ref{sec:tail}, we determine the \emph{macroscopic Hausdorff and Minkowski dimensions}  of the peaks of the solution to \eqref{eq:SHE}, both in the case of additive and multiplicative noise. In the case where $\dot\La$ is a Gaussian noise in dimension $1$, a similar program has been carried out by \cite{Khoshnevisan17}; see also \cite{Khoshnevisan18}. Let us briefly review the relevant definitions, first introduced by \cite{Barlow89, Barlow92} for subsets of $\Z^d$ and extended to subsets of $\R^d$ by \cite{Khoshnevisan17,Khoshnevisan18}  and \cite{Khoshnevisan17b}.  Writing $\calq_1=\{Q(x,r)=x+(0,r)^d : x\in \R^d, r\geq 1\}$ for the collection of cubes with side length $\side(Q(x,r))=r\geq1$, we define, for $E\subseteq \R^d$, $\rho>0$ and $n\in\N$, 
\[ \nu^n_\rho(E)=\inf\Biggl\{\sum_{i=1}^m \biggl(\frac{\side(Q_i)}{e^n}\biggr)^\rho: Q_1,\dots,Q_m\in \calq_1, E\cap\cals_n \subseteq \bigcup_{i=1}^m Q_i\Biggr\}, \]
where $\cals_n=B_\infty(e^{n-1},e^n)$. Letting $\log_+(x)=\log(x\vee e)$, we define  
\beq\label{eq:Dim} \begin{split} \Dim_\HH(E)&=\inf\Biggl\{ \rho>0: \sum_{n=1}^\infty \nu^n_\rho(E)<\infty\Biggr\}, \\
\Dim_\MM(E)&=\limsup_{n\to\infty}\frac1n \log_+\bigl\lvert\{q\in \Z^d\cap\cals_n: E\cap Q(q,1)\neq \emptyset\}\bigr\rvert. 
\end{split}\eeq 

\subsection{The multifractal nature of peaks}

We shall determine the macroscopic dimensions
of the largest  peaks observed on $\R^d$ and $\Z^d$.    Recall the convention that $Y_+$ denotes the solution to \eqref{eq:SHE} with $\si(x)=1$, while $Y$, as before, is the solution   with $\si(x)=x$. For  $\ga\in[0,\infty)$, we consider
\begin{equation}\label{eq:e-gamma}\begin{split}
		\cale^{+, \mathrm{c}}_\ga &= \{x\in\R^d: Y_+(t,x)\geq \lvert x\rvert^\ga\},\qquad 	\cale^{\times, \mathrm{c}}_\ga = \{x\in\R^d: Y(t,x)\geq \lvert x\rvert^\ga\},\\
		\cale^{+, \mathrm{d}}_\ga &= \{x\in\Z^d: Y_+(t,x)\geq \lvert 
x\rvert^\ga\}, \qquad 		\cale^{\times, \mathrm{d}}_\ga = \{x\in\Z^d: Y(t,x)\geq 
\lvert x\rvert^\ga\}.
	\end{split}
\end{equation}

\bthm\label{thm:dim}
Let 
$\ga\in[0,\infty)$. In the following,  $\Dim_{\HH|\MM}$ means one can take  $\Dim_{\HH}$ 
or $\Dim_{\MM}$ in the statement. Also, $\Dim_{\HH|\MM}(A)<0$  means that $A$ is a bounded 
set.
\benu
\item[(i)] Assume Condition~\ref{cond:sup} and Condition~\ref{cond:lt} with $\al=\frac2d$. If $d=1$ and $m_1(\la)=\infty$, assume Condition~\ref{cond:lt} with some $\al>2$. Then  almost surely,
\beq\label{eq:dim1}   \Dim_{\HH|\MM}(\cale^{\times,\mathrm{c}}_\ga) = \Dim_{\HH|\MM}(\cale^{+,\mathrm{c}}_\ga)= d-\tfrac 2d \ga.\eeq
\item[(ii)]
If Condition \ref{cond:sup} and Condition  \ref{cond:ht} hold with $\al \in(0,\frac 2d]$, then almost surely,
\beq\label{eq:dim1b} \Dim_{\HH|\MM}(\cale^{\times,\mathrm{c}}_\ga) = \Dim_{\HH|\MM}(\cale^{+,\mathrm{c}}_\ga) =d-\al\ga.
\eeq
\item[(iii)] If Condition \ref{cond:lt} holds with $\al =1+\frac 2d$, then almost surely,
\beq\label{eq:dim1c} \Dim_{\HH|\MM}(\cale^{\times,\mathrm{d}}_\ga) =  \Dim_{\HH|\MM}(\cale^{+,\mathrm{d}}_\ga) =d-(1+\tfrac 2d) \ga.\eeq
\item[(iv)]
If Condition  \ref{cond:ht} holds with $\al \in(0,1+\frac 2d]$, then almost surely,
\beq\label{eq:dim1d}\Dim_{\HH|\MM}(\cale^{\times,\mathrm{d}}_\ga) =  \Dim_{\HH|\MM}(\cale^{+,\mathrm{d}}_\ga) =d-\al\ga. \eeq
\eenu
\ethm
\bpr  The statements when the right-hand sides of \eqref{eq:dim1}--\eqref{eq:dim1d} are 
negative follow from Theorems~\ref{thm:peak-R-lt} and \ref{thm:peak-Z} in the case of 
multiplicative noise and from \cite[Theorems 6 and 7]{islands_add} 
in the case of additive noise. In the 
following, we only give the full details for the proof of 
$\Dim_\MM(\cale^{\times,\mathrm{c}}_\ga)\leq d-\frac2d \ga$ (Step 1) and the proof of 
$\Dim_\HH(\cale^{\times,\mathrm{c}}_\ga)\geq d-\frac2d \ga$ (Step 2), both under 
Condition~\ref{cond:sup} and Condition~\ref{cond:lt} with $\al=\frac 2d$ and the 
assumption $\ga \leq d^2/2$.  For both parts, the proofs are inspired by ideas from  
\cite{Khoshnevisan17}. By \cite[Lemma 3.1]{Barlow92}, Steps 1 and 2 imply 
$\Dim_\HH(\cale^{\times,\mathrm{c}}_\ga)=\Dim_\MM(\cale^{\times,\mathrm{c}}_\ga)=d-\frac 
2d \ga$. We explain towards the end of the proof (Step 3) why all other equalities in 
\eqref{eq:dim1}--\eqref{eq:dim1d} can be shown analogously.

\bigskip
\noindent\bff{Step 1: $\Dim_\MM(\cale^{\times,\mathrm{c}}_\ga)\leq d-\frac2d \ga$} 

\smallskip
\noindent 
Clearly,
\begin{align*}
\E\Bigl[\bigl\lvert\{q\in\Z^d\cap \cals_n: \cale^{\times,\mathrm{c}}_\ga\cap Q(q,1)\neq\emptyset\}\bigr\rvert\Bigr] &= \sum_{q\in\Z^d\cap \cals_n} \P(\cale^{\times,\mathrm{c}}_\ga\cap Q(q,1)\neq\emptyset)\\
&\leq \sum_{q\in\Z^d\cap \cals_n} \P\Biggl( \sup_{x\in Q(q,1)} Y(t,x) >(\lvert q\rvert-1)^\ga\Biggr).
\end{align*}
Since $\lvert\Z^d\cap\cals_n\rvert\leq Ce^{nd}$ for some $C>0$ and $\lvert q\rvert\geq e^{n-1}$ for all $q\in\Z^d\cap\cals_n$, Theorem~\ref{thm:sup-lt} implies
\beq\label{eq:upper}
\E\Bigl[\bigl\lvert\{q\in\Z^d\cap \cals_n: \cale^{\times,\mathrm{c}}_\ga\cap Q(q,1)\neq\emptyset\}\bigr\rvert\Bigr] 
\leq Ce^{nd}(e^{n-1}-1)^{-\frac 2d\ga} \leq Ce^{n(d-\frac 2d \ga)}
\eeq
for all $n\geq2$. By Markov's inequality, $  \P (\bigl\lvert\{q\in\Z^d\cap \cals_n: \cale^{\times,\mathrm{c}}_\ga\cap Q(q,1)\neq\emptyset\}\bigr\rvert > e^{\theta n} )$ is summable 
for all $\theta>d-\frac 2d \ga$. According to the first Borel--Cantelli lemma, for all $\theta$ in this range,
\[\limsup_{n\to\infty} \frac1n \log_+(|\{q\in\Z^d\cap \cals_n: \cale^{\times,\mathrm{c}}_\ga\cap Q(q,1)\neq\emptyset\}|)\leq \theta\]
almost surely. The upper bound on $\Dim_\MM(\cale^{\times,\mathrm{c}}_\ga)$ follows by letting $\theta\downarrow d-\frac2d \ga$.

\bigskip
\noindent\bff{Step 2: $\Dim_\HH(\cale^{\times,\mathrm{c}}_\ga)\geq d-\frac2d \ga$} 

\smallskip
\noindent We can assume that $d-\frac 2d \ga>0$. As in the proof of Theorem~\ref{thm:peak-R-lt}, the case $d\geq2$ is easier because the jumps are summable. So starting with $d\geq2$, we choose $\theta\in(2\ga/d^2,1)$ and consider the grid 
\beq\label{eq:xnk}\{x^n_k:k=1,\dots,K^n\}=\{e^{n-1}+ie^{\theta n}: i\in\N,\ 1\leq i\leq e^{n(1-\theta)}-e^{n(1-\theta)-1}\}^d\eeq
 in $\cals_n$ and, within each of the cubes $Q(x^n_k,e^{\theta n})$, the subgrid $$\{z^n_{k,\ell}:\ell=1,\dots,L^n_k\}=x^n_k+\{j\in\N:1\leq j\leq e^{\theta n}\}^d.$$
For every $k=1,\dots,K^n$ and $\ell=1,\dots, L^n_k$, we introduce the random fields
\beq\label{eq:Ynkl} Y_{k,\ell}^{n}(t,x)=\int_0^t \int_{Q(z^n_{k,\ell},1)}g(t-s,x-y)\,\La(\dd s,\dd y), \qquad (t,x)\in(0,\infty)\times\R^d,\eeq
which are independent for different values of $k$ and $\ell$  and satisfy $Y^n_{k,\ell} (t,x)\leq \wh Y(t,x)=e^{m_1(\la)t}Y(t,x)$. Therefore, for all $n\in N$ and $k=1,\dots,K^n$, 
\begin{align*}
\P\Biggl( \sup_{x\in Q(x^n_k,e^{\theta n})} \frac{Y(t,x)}{|x|^\ga} < 1\Biggr)&\leq \P\Biggl(\max_{\ell=1,\dots,L^n_k} \sup_{x\in Q(z^n_{k,\ell},1)} Y(t,x)<e^{n\ga}\Biggr)\\
 &\leq  \P\Biggl(\max_{\ell=1,\dots,L^n_k} \sup_{x\in Q(z^n_{k,\ell},1)} Y^n_{k,\ell}(t,x)<e^{m_1(\la)t+n\ga}\Biggr)\\
  &= \prod_{\ell=1}^{L^n_k} \P\Biggl(\sup_{x\in Q(z^n_{k,\ell},1)} Y^n_{k,\ell}(t,x)<e^{m_1(\la)t+n\ga}\Biggr).
\end{align*}

By Theorem~\ref{thm:sup-lt}, the last probability is less than or equal to $1-Ce^{-2n\ga/d}$. Applying the bound $1-x\leq e^{-x}$ and noticing that $\frac12 e^{\theta nd}\leq L^n_k\leq e^{\theta nd}$ by construction, we have
 \beq\label{eq:help5}
 \P\Biggl( \sup_{x\in Q(x^n_k,e^{\theta n})} \frac{Y(t,x)}{|x|^\ga} < 1\Biggr) \leq  \exp\bigl(-CL^n_k e^{-\frac 2d\ga n}\bigr)\leq  \exp\bigl(-\tfrac12 Ce^{(\theta d-\frac2d\ga )n}\bigr).
 \eeq
Since $K^n\leq e^{(1-\theta)nd}$ and $\theta d-\frac 2d \ga>0$ by our choice of $\theta$, we conclude that
 \[
 \sum_{n=1}^\infty \sum_{k=1}^{K^n}\P\Biggl( \sup_{x\in Q(x^n_k,e^{\theta n})} \frac{Y(t,x)}{|x|^\ga} < 1\Biggr) \leq  \sum_{n=1}^\infty \exp\bigl((1-\theta)nd-\tfrac12 Ce^{(\theta d-\frac2d\ga )n}\bigr) <\infty.
 \]
So the Borel--Cantelli lemma implies that the following holds with probability one: except for finitely many $n$, the intersection $Q(x^n_k,e^{\theta n})\cap \cale^{\times,\mathrm{c}}_\ga$ is nonempty for all $k=1,\dots,K^n$. In other words, the set $\cale^{\times,\mathrm{c}}_\ga$ is almost surely \emph{$\theta$-thick} in the sense of \cite[Def.~4.3]{Khoshnevisan17}. Thus, $\Dim_\HH(\cale^{\times,\mathrm{c}}_\ga)\geq d(1-\theta)$ almost surely by \cite[Prop.~4.4]{Khoshnevisan17} and the lower bound follows by letting $\theta\downarrow 2\ga/d^2$.

For $d=1$,  recall the processes $Y_0$, $Y_3$, $Y_4$ and $Z^{(m,\beta)}$ from \eqref{eq:Ys} and \eqref{eq:Zmb}. This time, we let $m(n)=n$, $\beta(n)=n^2$ and consider  the subgrid
$ \{\wh z^n_{k,\ell} = x^n_k + \ell n^3 : \ell\in\N,\ 1\leq \ell\leq \wh L^n_k=\lfloor n^{-3}e^{\theta n}\rfloor\} $.
Any two points in $\{\wh z^{n}_{k,\ell}: k=1,\dots,K^n, \ell=1,\dots, \wh L^n_k\}$ are at least $Cn^3$ apart from each other, where $C$ is a positive number. Therefore, for large $n$ and $k=1,\dots, K^n$, the   variables $$\Biggl\{\sup_{x\in(\wh z^n_{k,\ell}-1,\wh z^n_{k,\ell})}Z^{(m(n),\beta(n))}(t,x) \ : \  \ell=1,\dots, \wh L^n_k\Biggr\}$$ are independent of each other; cf.\ the paragraph after \eqref{eq:BC1}. Thus, for any $\theta>2 \ga$,
\begin{multline*} \sum_{n=1}^\infty \sum_{k=1}^{K^n}\P\Biggl( \max_{\ell=1,\dots, \wh L^n_k} \sup_{x\in(\wh z^n_{k,\ell}-1,\wh z^n_{k,\ell})} \frac{Z^{(m(n),\beta(n))}(t,x)}{|x|^\ga} < 3\Biggr)\\ \leq  \sum_{n=1}^\infty \exp\bigl((1-\theta)n-\tfrac12 Cn^{-3}e^{(\theta -2\ga )n}\bigr)<\infty.\end{multline*}
At the same time, by Lemma~\ref{lem:Y0Z},
\begin{multline*}
\sum_{n=1}^\infty \sum_{k=1}^{K^n}\sum_{\ell=1}^{\wh L^n_k} \P\Biggl(  \sup_{x\in(\wh z^n_{k,\ell}-1,\wh z^n_{k,\ell})} \frac{\lvert Y_0(t,x)- Z^{(m(n),\beta(n))}(t,x)\rvert}{|x|^\ga} >1\Biggr)\\
\leq \sum_{n=1}^\infty e^{(1-\theta)n}e^{\theta n}n^{-3}  (e^{-n^2} + C^n n^{-\frac{n}{12}}) e^{-2\ga n}  n\ga <\infty,
\end{multline*}
which shows that  $\cale^{\times,\mathrm{c},0}_\ga=\{x\in \R^d:Y_0(t,x)\geq 2\lvert x\rvert^\ga\}$ is $\theta$-thick, whence $\Dim_\HH(\cale^{\times,\mathrm{c},0}_\ga)\geq 1-2\ga$.

In addition, combining Step 1 with how we estimated $Y_3$ and $Y_4$ in the proof of Theorem~\ref{thm:peak-R-lt}, we have that   $\cale^{\times,\mathrm{c},34}_\ga =\{x\in \R^d:\lvert Y_3(t,x)+Y_4(t,x)\rvert\geq \lvert x\rvert^\ga\}$ satisfies $\Dim_\HH(\cale^{\times,\mathrm{c},34}_\ga)\leq \Dim_\MM(\cale^{\times,\mathrm{c},34}_\ga)\leq 1-\al\ga < 1-2\ga$. Since $\cale^{\times,\mathrm{c}}_\ga\supseteq \cale^{\times,\mathrm{c},0}_\ga \setminus \cale^{\times,\mathrm{c},34}_\ga$ and $\Dim_\HH(A)$  remains unchanged when a set of lower dimension is removed (see \cite[Property~(viii), p.\ 128]{Barlow92}), we conclude that $\Dim_\HH(\cale^{\times,\mathrm{c}}_\ga)\geq 1-2\ga$. 

\bigskip
\noindent\bff{Step 3: The remaining equalities} 

\smallskip
\noindent  Steps 1 and   2 show that 
$\Dim_\HH(\cale^{\times,\mathrm{c}}_\ga)=\Dim_\MM(\cale^{\times,\mathrm{c}}_\ga)=d-\frac 
2d \ga$ under Condition~\ref{cond:sup} and Condition~\ref{cond:lt} with $\al=\frac 2d$. 
With the same methods, all remaining equalities in \eqref{eq:dim1}--\eqref{eq:dim1d} can 
be deduced from the tail estimates in Theorems~\ref{thm:sol-ht}, \ref{thm:sol-lt} and 
\ref{thm:sup-ht} in the case of multiplicative noise and from the analogous results  
\cite[Theorems 2 and 5]{islands_add} in the case of additive noise. Note that the slowly 
varying functions in the tail estimates in Theorems~\ref{thm:sol-ht}, \ref{thm:sol-lt} 
and \ref{thm:sup-ht} are negligible on the scale of sets $\cale^{(\cdot)}_\ga$. This is 
why the macroscopic dimensions of $\cale^{(\cdot)}_\ga$ are the same for both additive 
and multiplicative noise.
\epr

 \subsection{Self-similarity of intermittency islands (or lack thereof)}

While $\cale_{d^2/2}$  is almost surely unbounded by Theorem~\ref{thm:peak-R-lt},  both 
its Minkowski and Hausdorff dimensions  are zero as the previous theorem asserts. Loosely 
speaking, the peaks that contribute to $\cale_{d^2/2}$ are too rare under the standard 
scale to have a positive macroscopic fractal dimension. However, under additive noise or 
under a multiplicative noise that is not too heavy-tailed, we can show that  after 
appropriate  changes of scale, these peaks will again exhibit  a multifractal structure 
that is reminiscent of the peaks studied so far. In fact, as we shall show, there exist 
infinitely many layers of peaks which, despite being defined on different scales, all 
share the same multifractal behavior. In these cases, we conclude  that the spatial peaks 
form \emph{large-scale self-similar multifractals}.
\bthm\label{thm:dim2} Let $\lvert \cdot\rvert$ be a norm on $\R^d$ and $x^q=\frac{x}{|x|}|x|^q$ for $x\neq0$ and $q>0$. Furthermore, for $n\in\N$, let $\log^{(n)}(r)=\log_+(\log^{(n-1)}(r))$ for $r\in\R$ (with $\log^{(0)}(r)=r$) and define $\log^{(n)}(x)=\frac{x}{|x|}\log^{(n)}(\lvert x\rvert)$ for $x\in\R^d\setminus\{0\}$.
\benu
\item[(i)] Assume Condition~\ref{cond:sup} and Condition~\ref{cond:lt} with $\al=\frac2d$. If $d=1$ and $m_1(\la)=\infty$, assume Condition~\ref{cond:lt} with some $\al>2$.
 For $N\in\N$ and $\ga>0$, consider  
\beq\label{eq:eNgamma} \begin{split} \cale^{(\times, \mathrm{c},N)}_{\ga} &= \Biggl\{ x\in\R^d:Y(t,x)\geq \lvert x \rvert^{\frac{d^2}2} \Biggl(\prod_{p=1}^{N-1} \lvert\log^{(p)}( x)\rvert^{\frac d2}\Biggr)\lvert\log^{(N)}(x)\rvert^{\frac \ga d} \Biggr\},\\
	\cale^{(+, \mathrm{c},N)}_{\ga} &= \Biggl\{ x\in\R^d:Y_+(t,x)\geq \lvert x\rvert^{\frac {d^2}2} \Biggl(\prod_{p=1}^{N-1} \lvert\log^{(p)}( x)\rvert^{\frac d2}\Biggr)\lvert\log^{(N)}(x)\rvert^{\frac\ga d} \Biggr\}.\end{split} \eeq
Then almost surely,
\beq\label{eq:dim2} \Dim_{\HH|\MM}\Bigl(\bigl(\log^{(N)}(\cale^{(\times,\mathrm{c},N)}_\ga)\bigr)^{\frac 1d}\Bigr)=\Dim_{\HH|\MM}\Bigl(\bigl(\log^{(N)}(\cale^{(+,\mathrm{c},N)}_\ga)\bigr)^{\frac 1d}\Bigr)= d-\tfrac 2d \ga. \eeq
\item[(ii)] Assume Condition~\ref{cond:sup} and Condition~\ref{cond:ht}  with $\al\in(0,\frac 2d)$.
For $N\in\N$ and $\ga>0$, let 
\beq\label{eq:eNgamma-b} 
	\cale^{(+, \mathrm{c},N)}_{\ga} = \Biggl\{ x\in\R^d:Y_+(t,x)\geq \lvert x\rvert^{\frac {d}\al} \Biggl(\prod_{p=1}^{N-1} \lvert\log^{(p)}( x)\rvert^{\frac 1\al}\Biggr)\lvert\log^{(N)}(x)\rvert^{\frac\ga d} \Biggr\}. \eeq
Then almost surely,
\beq\label{eq:dim2b} \Dim_{\HH|\MM}\Bigl(\bigl(\log^{(N)}(\cale^{(+, \mathrm{c},N)}_\ga)\bigr)^{\frac 1d}\Bigr) = d-\al \ga. \eeq
\item[(iii)] Assume   Condition~\ref{cond:ht}  with $\al\in(0,1+\frac 2d)$ or Condition~\ref{cond:lt} with $\al=1+\frac2d$. In both cases, consider 
for $N\in\N$ and $\ga>0$   the sets
\beq\label{eq:eNgamma-c} 
\cale^{(+, \mathrm{d},N)}_{\ga} = \Biggl\{ x\in\Z^d:Y_+(t,x)\geq \lvert x\rvert^{\frac {d}\al} \Biggl(\prod_{p=1}^{N-1} \lvert\log^{(p)}( x)\rvert^{\frac 1\al}\Biggr)\lvert\log^{(N)}(x)\rvert^{\frac\ga d} \Biggr\}. \eeq
Then almost surely,
\beq\label{eq:dim2c} \Dim_{\HH|\MM}\Bigl(\bigl(\log^{(N)}(\cale^{(+, \mathrm{d},N)}_\ga)\bigr)^{\frac 1d}\Bigr)= d-\al \ga. \eeq
\eenu
\ethm


\bpr[Proof of Theorem~\ref{thm:dim2}] The proofs of \eqref{eq:dim2}, \eqref{eq:dim2b} and \eqref{eq:dim2c} are completely analogous. We therefore only show the part in \eqref{eq:dim2} concerning $\cale^{(\times,\mathrm{c},N)}_\ga$.
We begin with a technicality: the result in \eqref{eq:dim2} does not depend on the choice 
of the norm $\lvert\cdot\rvert$. Indeed, let  $\lVert\cdot\rVert$ be another norm on 
$\R^d$ and write $\log^{(N)}_{\lvert\cdot\rvert}$ and $\log^{(N)}_{\lVert \cdot\rVert}$ 
and, similarly, $\cale^{(\times,\mathrm{c},N)}_{\ga,\lvert\cdot\rvert}$ and 
$\cale^{(\times,\mathrm{c},N)}_{\ga,\lVert \cdot \rVert}$ as well as 
$(\cdot)^r=(\cdot)^r_{\lvert\cdot\rvert}$ and $(\cdot)=(\cdot)_{\lVert\cdot\rVert}$ to 
emphasize the dependence on the chosen norm. For $K>0$, further let 
$\cale^{(\times,\mathrm{c},N)}_{\ga,\lvert\cdot\rvert}(K)$ be  the right-hand side of the 
first line of \eqref{eq:eNgamma} when $Y(t,x)$ is replaced by $Y(t,x)/K$, and define 
$\cale^{(\times,\mathrm{c},N)}_{\ga,\lVert\cdot\rVert}(K)$ analogously. By the equivalence 
of norms on $\R^d$, there is $C\geq1$ such that 
$\cale^{(\times,\mathrm{c},N)}_{\ga,\lvert\cdot\rvert}(CK)\subseteq 
\cale^{(\times,\mathrm{c},N)}_{\ga,\lVert\cdot\rVert}(K)\subseteq 
\cale^{(\times,\mathrm{c},N)}_{\ga,\lvert\cdot\rvert}(C^{-1}K)$ for all $K>0$. Thus,
\[ \bigl(\log^{(N)}_{\lvert\cdot\rvert}(\cale^{(\times,\mathrm{c},N)}_{\ga,\lvert\cdot\rvert}(CK))\bigr)_{\lvert\cdot\rvert}^{\frac 1d} \subseteq\bigl(\log^{(N)}_{\lvert\cdot\rvert}(\cale^{(\times,\mathrm{c},N)}_{\ga,\lVert\cdot\rVert}(K))\bigr)_{\lvert\cdot\rvert}^{\frac 1d} = f\Bigl(\bigl(\log^{(N)}_{\lVert\cdot\rVert}(\cale^{(\times,\mathrm{c},N)}_{\ga,\lVert\cdot\rVert}(K))\bigr)_{\lVert\cdot\rVert}^{\frac 1d}\Bigr)\]
with the function $f$ from \eqref{eq:f}. This function is bounded and Lipschitz continuous outside a ball containing the origin according to Lemma~\ref{lem:Lip}. Since the macroscopic Hausdorff and Minkowski dimensions are monotone and insensitive to adding or deleting bounded subsets,   $$\Dim_{\HH|\MM}((\log^{(N)}_{\lvert\cdot\rvert}(\cale^{(\times,\mathrm{c},N)}_{\ga,\lvert\cdot\rvert}(CK))_{\lvert\cdot\rvert}^{\frac 1d}) \leq \Dim_{\HH|\MM}((\log^{(N)}_{\lVert\cdot\rVert}(\cale^{(\times,\mathrm{c},N)}_{\ga,\lVert\cdot\rVert}(K)))_{\lVert\cdot\rVert}^{\frac 1d})$$ by Lemma~\ref{lem:Lipschitz}. By symmetry, this inequality also holds if we  switch the role of the two norms. So the part in \eqref{eq:dim2} concerning $\cale^{(\times,\mathrm{c},N)}_\ga$ is proved if  we show that  
\begin{align*} \Dim_{\HH|\MM}((\log^{(N)}_{\lVert\cdot\rVert}(\cale^{(\times,\mathrm{c},N)}_{\ga,\lVert\cdot\rVert}(K)))_{\lVert\cdot\rVert}^{1/d})&\leq d-\tfrac 2d \ga,\\ \Dim_{\HH|\MM}((\log^{(N)}_{\lvert\cdot\rvert}(\cale^{(\times,\mathrm{c},N)}_{\ga,\lvert\cdot\rVert}(K)))_{\lvert\cdot\rvert}^{1/d})&\geq d-\tfrac 2d \ga
\end{align*} 
for every $K>0$.  
In order to simplify notation, we omit all subscripts $|\cdot|$ and $\|\cdot\|$ in the following and write $|\cdot|$ for \emph{both} norms, with the agreement that $|\cdot|$ is the supremum norm in Step 1 and the Euclidean norm in Step 2 below.

Next, let $\exp^{(n)}(x)=\frac{x}{|x|}\exp^{(n)}(\lvert x\rvert)$ for $n\in\N$ and $x\in\R^d\setminus\{0\}$, where $\exp^{(n)}$ is the iterated exponential defined in Lemma~\ref{lem:Lip}. A moment's thought reveals that
\[ \Dim_{\HH|\MM}\bigl((\log^{(N)}(\cale^{(\times,\mathrm{c},N)}_{\ga}(K)))^{\frac 1d}\bigr) = \Dim_{\HH|\MM}\bigl(\ov\cale^{(N)}_{\ga}(K)\bigr),\]
where for $K>0$,
\[ \ov\cale^{(N)}_{\ga}(K)=\Biggl\{x\in\R^d: Y(t,\exp^{(N)}(x^d))>K \lvert\exp^{(N)}(x^d)\rvert^{\frac{d^2}2}\Biggl(\prod_{p=1}^{N-1} \lvert\exp^{(p)}(x^d)\rvert^{\frac d2}\Biggr)|x|^{\ga}\Biggr\}. \]

As the statement of the theorem for $\ga>d^2/2$ follows from Theorem~\ref{thm:peak-R-lt}, we may (and will) assume $d-\frac2d \ga \geq0$ in the following. By \cite[Lemma 3.1]{Barlow92}, it is enough to prove that $\Dim_\MM(\ov \cale^{(N)}_\ga(K))\leq d-\frac2d \ga$ and $\Dim_\HH(\ov\cale^{(N)}_\ga(K))\geq d-\frac2d \ga$ almost surely.

\bigskip
\noindent\bff{Step 1: Upper bound} 

\smallskip
\noindent 
Let $k(n)\in\N$ and $n-1=:r^{(n)}_0<\dots<r^{(n)}_{k(n)-1}<n+1\leq r^{(n)}_{k(n)}$ be defined via the relations
\[   \exp^{(N)}(\exp(dr^{(n)}_i))-\exp^{(N)}(\exp(dr^{(n)}_{i-1})) = 1,\qquad i =1,\dots,k(n),\]
and let $\cals_n(i)= B_\infty(\exp(r^{(n)}_{i-1}),\exp(r^{(n)}_{i}))$ for $i=1,\dots,k(n)$. We required $r^{(n)}_{k(n)}\geq n+1$ and not just $r^{(n)}_{k(n)}\geq n$ in order that $\bigcup_{q\in\Z^d\cap \cals_n} Q(q,1)\subseteq \bigcup_{i=1}^{k(n)}\cals_n(i)$. For every $n$ and $i$, the image of $\cals_n(i)$ under the mapping $\exp^{(N)}((\cdot)^d)$ can be covered with unit cubes $(Q^{n,i}_j: j=1,\dots,\ell_n(i))$, where 
\beq\label{eq:lni}\ell_n(i)\leq C[(\exp^{(N)}(\exp(dr^{(n)}_i)))^d-(\exp^{(N)}(\exp(dr^{(n)}_{i-1})))^d]\eeq
for some finite $C>0$   independent of $n$ and $i$. Denoting the pre-image of $Q^{n,i}_j$ 
under the same mapping by $\wt Q^{n,i}_j$ and assuming  $\wt Q^{n,i}_j\cap 
[-\exp(r^{(n)}_{i-1}),\exp(r^{(n)}_{i-1})]^d=\emptyset$ without loss of generality, we 
have 
\begin{align*}
&\bigl\lvert\{q\in\Z^d\cap \cals_n: \ov\cale^{(N)}_{\ga}\cap Q(q,1)\neq\emptyset\}\bigr\rvert\\
&\qquad \leq \mathtoolsset{multlined-width=0.9\displaywidth} \begin{multlined}[t]\sum_{i=1}^{k(n)} \sum_{j=1}^{\ell_n(i)} \bone\Biggl\{ \exists x\in \wt Q^{n,i}_j:  Y(t,\exp^{(N)}(x^d))> K\lvert\exp^{(N)}(x^d)\rvert^{\frac{d^2}2}\\
	\times\Biggl(\prod_{p=1}^{N-1} \lvert\exp^{(p)}(x^d)\rvert^{\frac d2}\Biggr)|x|^{\ga}\Biggr\}\end{multlined}\\
&\qquad \leq \mathtoolsset{multlined-width=0.9\displaywidth} \begin{multlined}[t]
\sum_{i=1}^{k(n)} \sum_{j=1}^{\ell_n(i)} \bone\Biggl\{ \exists x\in \wt Q^{n,i}_j:  Y(t,\exp^{(N)}(x^d))> K (\exp^{(N)}(\exp(dr^{(n)}_{i-1})))^{\frac {d^2}2}\\
\times\Biggl(\prod_{p=1}^{N-1} (\exp^{(p)}(\exp(dr^{(n)}_{i-1})))^{\frac d2}\Biggr)\exp(\ga r^{(n)}_{i-1})  \Biggr\}\end{multlined}\\
&\qquad=\mathtoolsset{multlined-width=0.9\displaywidth} \begin{multlined}[t]\sum_{i=1}^{k(n)} \sum_{j=1}^{\ell_n(i)} \bone\Biggl\{ \sup_{x\in Q^{n,i}_j} Y(t,x)  > K (\exp^{(N)}(\exp(dr^{(n)}_{i-1})))^{\frac{d^2}2}\\
\times\Biggl(\prod_{p=1}^{N-1} (\exp^{(p)}(\exp(dr^{(n)}_{i-1})))^{\frac d2}\Biggr)\exp(\ga r^{(n)}_{i-1})  \Biggr\}.\end{multlined}
\end{align*}

Taking expectation and using \eqref{eq:lni} and Theorem~\ref{thm:sup-lt}, we obtain
\begin{align*}
&\E\Bigl[\bigl\lvert\{q\in\Z^d\cap \cals_n: \ov\cale^{(N)}_{\ga}\cap Q(q,1)\neq\emptyset\}\bigr\rvert\Bigr] \\
&\qquad\leq C\sum_{i=1}^{k(n)} \bigl[\exp(d\exp^{(N)}(dr^{(n)}_i))-\exp(d\exp^{(N)}(dr^{(n)}_{i-1}))\bigr]  \exp(-d\exp^{(N)}(dr^{(n)}_{i-1}))\\
&\qquad\quad\times\Biggl(\prod_{p=1}^{N-1} \exp(-\exp^{(p)}(dr^{(n)}_{i-1}))\Biggr)\exp(-\tfrac2d\ga r^{(n)}_{i-1}).
\end{align*}
Applying the mean-value theorem to the difference in brackets and noticing that the derivative of $r\mapsto  \exp(d\exp^{(N)}(dr))$ increases in $r$, we further deduce that
\begin{align*}
&\E\Bigl[\bigl\lvert\{q\in\Z^d\cap \cals_n: \ov\cale^{(N)}_{\ga}\cap Q(q,1)\neq\emptyset\}\bigr\rvert\Bigr] \\
&\qquad\leq Cd^2\sum_{i=1}^{k(n)} (r^{(n)}_i-r^{(n)}_{i-1})  \exp(d\exp^{(N)}(dr^{(n)}_i)) \Biggl(\prod_{p=1}^{N}\exp^{(p)}(dr^{(n)}_i)\Biggr) \\
&\qquad\quad\times\exp(-d\exp^{(N)}(dr^{(n)}_{i-1}))\Biggl(\prod_{p=1}^{N-1} \exp(-\exp^{(p)}(dr^{(n)}_{i-1}))\Biggr)\exp(-\tfrac2d\ga r^{(n)}_{i-1}).
\end{align*} 

By construction, $\exp(\exp^{(N)}(dr^{(n)}_i) = 1+\exp(\exp^{(N)}(dr^{(n)}_{i-1}))\leq 
2\exp(\exp^{(N)}(dr^{(n)}_{i-1}))$. Taking logarithm consecutively on both sides, we also 
get $\exp^{(p)}(dr^{(n)}_i)\leq 2\exp^{(p)}(dr^{(n)}_{i-1})$ for $p=1,\dots,N$. Thus, we 
can simplify the estimate in the previous display to
\begin{align*}
\E\Bigl[\bigl\lvert\{q\in\Z^d\cap \cals_n: \ov\cale^{(N)}_{\ga}\cap Q(q,1)\neq\emptyset\}\bigr\rvert\Bigr] &\leq Cd^22^{N+1}\sum_{i=1}^{k(n)} (r^{(n)}_i-r^{(n)}_{i-1})      \exp((d-\tfrac2d\ga)r^{(n)}_{i-1})\\
&\leq Cd^22^{N+1}e^{(n+1)(d-\frac 2d \ga)}\sum_{i=1}^{k(n)} (r^{(n)}_i-r^{(n)}_{i-1})\\
&\leq Cd^22^{N+2}e^{(n+1)(d-\frac 2d \ga)}.
\end{align*} 
This estimate is analogous to the bound \eqref{eq:upper} in the proof of Theorem~\ref{thm:dim}, so the proof can be completed just like there.

\bigskip
\noindent\bff{Step 2: Lower bound} 

\smallskip
\noindent As in the proof of Theorem~\ref{thm:dim}, our strategy will be to show that $\ov\cale^{(N)}_\ga(K)$ is $\theta$-thick for all $\theta\in (2\ga/d^2,1)$, assuming $d-\frac 2d \ga>0$ without loss of generality. We   again consider the grid $\{x^n_k:k=1,\dots,K^n\}$ from \eqref{eq:xnk} and the associated cubes $Q(x^n_k,e^{\theta n})$. Unfortunately, if $d\geq2$, we do not have sufficient control over the shape of the images that we obtain from applying the mapping $\exp^{(N)}$ to these cubes. This is why in $d\geq2$, we will inscribe some auxiliary geometric solids that are easier to analyze in those cubes. For the remaining proof, we only consider the case $d\geq2$; the one-dimensional situation is geometrically much simpler and is therefore left to the reader (the potential existence of infinite variation jumps can be addressed as in the proof of Theorem~\ref{thm:dim}). 

For $d\geq2$, we consider  geometric shapes that we call \emph{(spherical) shell sectors}; see Figure~\ref{fig:1}. These are obtained by intersecting a shell (i.e., the set difference of two balls with the same center) with a cone that has this center as apex. Equivalently, a shell sector is the difference of two concentric sectors. (A sector results from cutting a ball into two parts by a hyperplane and taking the union of the smaller part with the cone formed by the intersection, an $(n-1)$-dimensional ball, as base and the center of the cut ball as apex; ``concentric'' here means that both sectors have the same apex and the same axis of revolution.)

A shell sector $S=S(A, O, \rho,s)$ (see Figure~\ref{fig:2} for illustration) is uniquely parametrized by four parameters: its \emph{apex} $A$ (i.e., the joint apex of the two sectors),  its \emph{suspension point} $O$ (i.e., the center of the base of the larger sector), its \emph{base radius} $\rho$ (i.e., the radius of the base of the larger sector), and its \emph{side length} $s$ (i.e., the difference of the radii of the two balls). Several other characteristics of $S$ will be important to us: its \emph{inner radius} $r$ and \emph{outer radius} $R$ (i.e., the radius of the smaller and the larger ball, respectively), its \emph{inner vertex} $v$ and \emph{outer vertex} $V$ (i.e., the point on the boundary of the smaller and larger ball, respectively, that is collinear with $A$ and $O$), its \emph{height} $h$ (i.e., the distance between the base center of the smaller sector and $V$),  its \emph{angle} $\phi$ (i.e., the largest possible angle between the half-lines $AO$ and $AP$, where $P$ is a boundary point of $S$), and its \emph{direction} $w=(O-A)/|O-A|$.

\begin{figure}[!tbh]
	\centering
\begin{minipage}[b]{0.45\textwidth}
	\begin{tikzpicture}[scale=0.8]
	\def\R{6};
	\def\r{3};
	\def\a{20}
	\def\eps{0.1}
	\draw [ultra thick, domain=90-\a:90+\a] plot ({\R*cos(\x)},{\R*sin(\x)});
	\draw [ultra thick, domain=90-\a:90+\a] plot ({\r*cos(\x)}, {\r*sin(\x)});
	\draw [dashed, domain=105+\a:90+\a] plot ({\R*cos(\x)},{\R*sin(\x)});
	\draw [dashed, domain=90-\a:75-\a] plot ({\R*cos(\x)},{\R*sin(\x)});
	\draw [dashed, domain=180:90+\a] plot ({\r*cos(\x)}, {\r*sin(\x)});
	\draw [dashed, domain=90-\a:0] plot ({\r*cos(\x)}, {\r*sin(\x)});
	\draw[-] (0,0) -- ({\r*cos(90+\a)},{\r*sin(90+\a)});
	\draw[-] (0,0) -- ({\r*cos(90-\a)},{\r*sin(90-\a)});
	\draw[ultra thick] ({\r*cos(90-\a)},{\r*sin(90-\a)}) -- ({\R*cos(90-\a)}, {\R*sin(90-\a)});
	\draw[ultra thick] ({\r*cos(90+\a)},{\r*sin(90+\a)}) -- ({\R*cos(90+\a)}, {\R*sin(90+\a)});
	\draw[dashed] ({\R*cos(90-\a)-5},{\R*sin(90-\a)}) -- ({\R*cos(90+\a)+5},{\R*sin(90+\a)});
	\draw[dashed] ({\r*cos(90-\a)-4},{\r*sin(90-\a)}) -- ({\r*cos(90+\a)+4},{\r*sin(90+\a)});
	
	\node  at (1.9,0.25){\footnotesize smaller ball};
	\node  at (3,4.5){\footnotesize larger ball};
	
	\node[label={[align=center]\footnotesize smaller \\ \footnotesize sector}]  at (0,1.5){};
	\node[label={[align=center]\footnotesize larger \\ \footnotesize sector}]  at (0,3.5){};
	\node[label={[align=center]\footnotesize hyperplane}]  at (-3,2.75){};
	\node[label={[align=center]\footnotesize hyperplane}]  at (-3,5.55){};
	
	\draw[-{Latex[length=7]}] ({-\R*sin(\a)-2},-1) -- ({-\R*sin(\a)-2},{7*cos(25)+1.5}) node[right] {$\mathbb{R}$};
	\draw[-{Latex[length=7]}] ({-\R*sin(\a)-2},-1) -- ({\R*sin(\a)+0.5},-1) node[above] {$\mathbb{R}^{d-1}$};
	\end{tikzpicture}
	\caption{The construction of a $d$-dimensional shell sector.}\label{fig:1} 
\end{minipage}
\hfill
\begin{minipage}[b]{0.45\textwidth}
	\begin{tikzpicture}[scale=0.8]
\def\R{7};
\def\r{3};
\def\a{25}
\def\eps{0.1}
\draw [ultra thick, domain=90-\a:90+\a] plot ({\R*cos(\x)},{\R*sin(\x)});
\draw [ultra thick, domain=90-\a:90+\a] plot ({\r*cos(\x)}, {\r*sin(\x)});
\draw [domain=90-\a:90] plot ({1.25*cos(\x)}, {1.25*sin(\x)});
\draw[dashed] (0,0) -- ({\r*cos(90+\a)},{\r*sin(90+\a)});
\draw[dashed] (0,0) -- ({\r*cos(90-\a)},{\r*sin(90-\a)});
\draw[dashed] ({\r*cos(90-\a)},{\r*sin(90-\a)}) -- ({\r*cos(90+\a)},{\r*sin(90+\a)});
\draw[ultra thick] ({\r*cos(90-\a)},{\r*sin(90-\a)}) -- ({\R*cos(90-\a)}, {\R*sin(90-\a)});
\draw[ultra thick] ({\r*cos(90+\a)},{\r*sin(90+\a)}) -- ({\R*cos(90+\a)}, {\R*sin(90+\a)});
\draw[-] ({\R*cos(90-\a)},{\R*sin(90-\a)}) -- ({\R*cos(90+\a)},{\R*sin(90+\a)});
\draw[dashed] (0,0) -- (0,\r);
\draw[-] (0,\r) -- (0,\R);
\draw[-{Latex[length=7]}] (0,0) -- (0,1.75) node[left] {$w$};

\node[label=below:$A$]  at (0,0)[circle,fill,inner sep=1.5pt, color=black]{};
\node[label=above left:$v$]  at (0,\r)[circle,fill,inner sep=1.5pt, color=black]{};
\node[label=above left:$V$]  at (0,\R)[circle,fill,inner sep=1.5pt, color=black]{};
\node[label=above left:$O$]  at (0,{\R*cos(\a)})[circle,fill,inner sep=1.5pt, color=black]{};
\node[label=above:$\phi$] at (0.2,0.4){};

\draw [decorate,decoration={brace,amplitude=5},xshift=-5,yshift=0]
(0,0) -- ({(\r-\eps)*cos(90+\a)},{(\r-\eps)*sin(90+\a)}) node [black,midway,xshift=-10,yshift=-5] 
{$r$};
\draw [decorate,decoration={brace,amplitude=5},xshift=-5,yshift=0]
({(\r)*cos(90+\a)},{(\r)*sin(90+\a)}) -- ({(\R-\eps)*cos(90+\a)},{(\R-\eps)*sin(90+\a)}) node [black,midway,xshift=-10,yshift=-5] 
{$s$};
\draw [decorate,decoration={brace,amplitude=5,mirror},xshift=5,yshift=0]
(0,0) -- ({(\R-\eps)*cos(90-\a)},{(\R-\eps)*sin(90-\a)}) node [black,midway,xshift=12,yshift=-2] 
{$R$};
\draw [decorate,decoration={brace,amplitude=5},xshift=0,yshift=-3.5]
(-\eps,{\R*sin(90-\a)}) -- ({-\R*cos(90-\a)+\eps},{\R*sin(90-\a)})  node [black,midway,yshift=-13] 
{$\rho$};
\draw [decorate,decoration={brace,amplitude=5,mirror},xshift=3.5,yshift=-0]
(0,{\r*sin(90-\a)+\eps/2}) -- (0,\R-\eps/2) node [black,midway,xshift=13] 
{$h$};

\draw[-{Latex[length=7]}] ({-\R*sin(\a)-1},-1) -- ({-\R*sin(\a)-1},{\R*cos(\a)+1.5}) node[right] {$\mathbb{R}$};
\draw[-{Latex[length=7]}] ({-\R*sin(\a)-1},-1) -- ({\R*sin(\a)+0.5},-1) node[above] {$\mathbb{R}^{d-1}$};
\end{tikzpicture}
\caption{Parametrization of a shell sector $S(A,O,\rho,s) = S[A, w, r, R, \phi]$.}\label{fig:2} 
\end{minipage}
\end{figure}
Simple geometric considerations yield the following relations:
\beq\label{eq:hphi} R=r+s,\quad \sin\phi=\frac{\rho}{R},\quad\tan \phi = \frac{\rho}{\rho_0},\quad h=s\cos \phi +\sqrt{\rho_0^2+\rho^2}-\rho_0, \eeq
where $\rho_0=\lvert O-A\rvert$. As a result, another way of parametrization is $S=S[A,w,r,R,\phi]$. 
Moreover, $S$ can be inscribed in a box with one side of length $h$ and all other sides of length $2\rho$. This box has diameter $\sqrt{h^2+4(d-1)\rho^2}$, which, in particular, implies 
\beq\label{eq:dist} \max_{P\in S}\dist(O,P) \leq \sqrt{h^2+4(d-1)\rho^2}. \eeq

Back to the cubes $Q(x^n_k,e^{\theta n})$, let $z^n_k$ be the center of these cubes and consider the shell sectors 
\beq\label{eq:Snk} S^n_k=S(0,z^n_k,  \tfrac1{4\sqrt{d}} e^{\theta n},\tfrac1{4\sqrt{d}} e^{\theta n}).\eeq
By \eqref{eq:hphi} and the elementary inequality $\sqrt{x+y}-\sqrt{x}\leq \sqrt{y}$, the height of $S^n_k$ is bounded by $\tfrac1{2\sqrt{d}} e^{\theta n}$. Together with \eqref{eq:dist}, it follows that
$\dist (z^n_k, P)\leq \frac12e^{\theta n}$ for any point $P$ in $S^n_k$. The important conclusion is that 
\beq\label{eq:subset} S^n_k\subseteq Q(x^n_k,e^{\theta n}).\eeq
 For later reference, let us also give an estimate on $\phi^n_k$, the angle of $S^n_k$: since $e^{n-1}\leq \lvert z^n_k \rvert\leq \sqrt{d}e^n$ and $\frac12 x\leq\arctan x \leq x$ for small $x>0$, \eqref{eq:hphi} implies
\beq\label{eq:phi} \frac1{8d}e^{(\theta -1)n}\leq\arctan \frac{e^{\theta n}}{4de^n} \leq \phi^n_k\leq \arctan \frac{e^{\theta n}}{4\sqrt{d}e^{n-1}}\leq \frac{e}{4\sqrt{d}}e^{(\theta-1)n} \eeq
for large $n$.

The reason why we have introduced the shell sectors $S^n_k$ at all is that $\ov S^n_k=\exp^{(N)}((S^n_k)^d)$ are again shell sectors. In fact,
\[ \ov S^n_k=S[0,w^n_k,\exp^{(N)}((r^n_k)^d),\exp^{(N)}((R^n_k)^d),\phi^n_k],  \]
where $w^n_k=z^n_k/|z^n_k|$ is the direction and $r^n_k$ and $R^n_k$ are the inner and outer radius of $S^n_k$, respectively. Given $n$ and $k$, we now define $u^{n,k}_0<\dots<u^{n,k}_{\ell^n_k}$ by setting $u^{n,k}_0=r^n_k$ and requiring $\ell^n_k$ be the maximal number such that
\beq\label{eq:unkl}  \exp^{(N)}((u^{n,k}_\ell)^d)-\exp^{(N)}((u^{n,k}_{\ell-1})^d)=1 \eeq
for all $\ell=1,\dots,\ell^n_k$ and $u^{n,k}_{\ell^n_k}\leq R^n_k$. By construction and 
the first identity in \eqref{eq:hphi}, 
\beq\label{eq:unkl-2}\begin{split} e^{n-1} \leq r^n_k= u^{n,k}_0<\dots<u^{n,k}_{\ell^n_k}&\leq R^n_k\leq \sqrt{d}e^n,\\
	\frac{1}{8\sqrt{d}}e^{\theta n}\leq  \frac12(R^n_k-r^n_k) \leq u^{n,k}_{\ell^n_k}-u^{n,k}_0 &\leq R^n_k-r^n_k = \frac{1}{4\sqrt{d}}e^{\theta n}.\end{split} \eeq

Next, given $n$, $k$ and $\ell$, consider 
\beq\label{eq:ovSnk} \ov S^{n,k}_\ell=S[0,w^n_k,\exp^{(N)}((u^{n,k}_{\ell-1})^d), \exp^{(N)}((u^{n,k}_{\ell})^d),\phi^n_k].\eeq 
By a simple calculation (cf.\ \cite[Sect.\ V]{Jacquelin03}), there is a constant $C_d>0$ that only depends on $d$ such that, with obvious notation,
\begin{align*} \Leb(\ov S^{n,k}_\ell) &= C_d\bigl((R^{n,k}_\ell)^d -(r^{n,k}_\ell)^d\bigr)\int_0^{\phi^n_k}(\sin t)^{d-2} \,\dd t\\ &\geq \frac{C_d}{2^{d-2}(d-1)}\bigl((R^{n,k}_\ell)^d -(r^{n,k}_\ell)^d\bigr)(\phi^n_k)^{d-1}\end{align*}
for large $n$. As a consequence of the Vitali covering theorem (see \cite[Thm.~5.5.2]{Bogachev07}), there are $\eps>0$ and pairwise disjoint cubes $Q^{n,k}_{\ell,1},\dots Q^{n,k}_{\ell,m^{n,k}_\ell}\subseteq \ov S^{n,k}_\ell$ with side length within $(\eps,1]$ such that 
\beq\label{eq:mnkl} m^{n,k}_\ell \geq \sum_{m=1}^{m^{n,k}_\ell} \Leb(Q^{n,k}_{\ell,m}) \geq \frac{C_d}{2^{d-1}(d-1)}\bigl((R^{n,k}_\ell)^d -(r^{n,k}_\ell)^d\bigr)(\phi^n_k)^{d-1}. \eeq

We are now ready for the final (probabilistic) part of the proof.  Whenever $n$ is sufficiently large, we deduce from \eqref{eq:subset} and \eqref{eq:ovSnk} that for all $k=1,\dots,K^n$, 
\begin{align*}
&\P\Biggl( \sup_{x\in Q(x^n_k,e^{\theta n})} \frac{Y(t,\exp^{(N)}(x^d))}{ \lvert\exp^{(N)}(x^d)\rvert^{\frac{d^2}2}(\prod_{p=1}^{N-1} \lvert\exp^{(p)}(x^d)\rvert^{\frac d2})|x|^{\ga}} \leq K\Biggr)\\
&\qquad\leq \P\Biggl( \sup_{x\in S^n_k} \frac{Y(t,\exp^{(N)}(x^d))}{ \lvert\exp^{(N)}(x^d)\rvert^{\frac{d^2}2}(\prod_{p=1}^{N-1} \lvert\exp^{(p)}(x^d)\rvert^{\frac d2})|x|^{\ga}} \leq K\Biggr)\\
&\qquad\leq\mathtoolsset{multlined-width=0.9\displaywidth} \begin{multlined}[t]\P\Biggl(\bigcap_{\ell=1}^{\ell^n_k} \bigcap_{m=1}^{m^{n,k}_\ell} \Biggl\{\sup_{x\in Q^{n,k}_{\ell,m}} Y(t,x)\leq K (\exp^{(N)}((u^{n,k}_\ell)^d))^{\frac{d^2}2}\\
	\times\Biggl(\prod_{p=1}^{N-1} (\exp^{(p)}((u^{n,k}_\ell)^d))^{\frac d2}\Biggr)(u^{n,k}_\ell)^{\ga}\Biggr\}\Biggr).\end{multlined}
\end{align*}
Let $Y^{n,k}_{\ell,m}$ be defined in the same way as $Y^{n}_{k,\ell}$ in \eqref{eq:Ynkl} but with $Q(z^n_{k,\ell},1)$ replaced by $Q^{n,k}_{\ell,m}$. By construction, the latter are mutually disjoint for different values of $\ell$ and $m$. Therefore, $\{Y^{n,k}_{\ell,m}:\ell=1,\dots, \ell^n_k, m=1,\dots, m^{n,k}_\ell\}$ is a family of independent random fields. In addition, they clearly satisfy $Y^{n,k}_{\ell,m}(t,x)\leq \wh Y(t,x)=e^{m_1(\la)t}Y(t,x)$, so
 \begin{align*}
 &\P\Biggl( \sup_{x\in Q(x^n_k,e^{\theta n})} \frac{Y(t,\exp^{(N)}(x^d))}{ \lvert\exp^{(N)}(x^d)\rvert^{\frac{d^2}2}(\prod_{p=1}^{N-1} \lvert\exp^{(p)}(x^d)\rvert^{\frac d2})|x|^{\ga}} \leq K\Biggr)\\
 &\qquad\leq\mathtoolsset{multlined-width=0.9\displaywidth} \begin{multlined}[t]\prod_{\ell=1}^{\ell^n_k} \prod_{m=1}^{m^{n,k}_\ell} \P\Biggl( \sup_{x\in Q^{n,k}_{\ell,m}} Y^{n,k}_{\ell,m}(t,x)\leq Ke^{-m_1(\la)t} (\exp^{(N)}((u^{n,k}_\ell)^d))^{\frac{d^2}{2}}\\
 	\times\Biggl(\prod_{p=1}^{N-1} \exp^{(p)}((u^{n,k}_\ell)^d)^{\frac d2}\Biggr)(u^{n,k}_\ell)^{\ga} \Biggr)\end{multlined}\\
 &\qquad\leq \exp\Biggl(-C\sum_{\ell=1}^{\ell^n_k} m^{n,k}_\ell(\exp^{(N)}((u^{n,k}_\ell)^d))^{-d} \Biggl(\prod_{p=1}^{N-1} \exp^{(p)}((u^{n,k}_\ell)^d)^{-1}\Biggr) (u^{n,k}_\ell)^{-\frac 2d \ga}\Biggr),
 \end{align*}
 where we used \eqref{eq:sup-lt}   and the estimate $1-x\leq e^{-x}$ for the last step. One detail is worth mentioning: The bounds in \eqref{eq:sup-lt} were proved for  cubes $Q$   of side length $1$. The reader may easily verify that the same bound holds uniformly for all cubes of side length larger than $\eps$, except that the values of the limit inferior and superior in \eqref{eq:sup-lt} now depend on $\eps$. This is why the constant $C$ in the previous display may depend on $\eps$ but not on $n$, $k$, $\ell$ or $m$.

By \eqref{eq:phi}, \eqref{eq:unkl} (which implies $\exp^{(p)}((u^{n,k}_\ell)^d)-\exp^{(p)}((u^{n,k}_{\ell-1})^d)\leq 1$ for all $p=1,\dots,N$), \eqref{eq:unkl-2} and \eqref{eq:mnkl} together with the mean-value theorem and the bound $x\geq \frac12(x+1)$ for $x>1$,  we deduce that 
\begin{align*}
&\sum_{\ell=1}^{\ell^n_k} m^{n,k}_\ell(\exp^{(N)}((u^{n,k}_\ell)^d))^{-d} \Biggl(\prod_{p=1}^{N-1} (\exp^{(p)}((u^{n,k}_\ell)^d))^{-1}\Biggr) (u^{n,k}_\ell)^{-\frac 2d \ga}\\
&\quad \geq \mathtoolsset{multlined-width=0.9\displaywidth} \begin{multlined}[t] C\sum_{\ell=1}^{\ell^n_k} \bigl((\exp^{(N)}((u^{n,k}_{\ell})^d))^d -(\exp^{(N)}((u^{n,k}_{\ell-1})^d))^d\bigr)(\phi^n_k)^{d-1}(\exp^{(N)}((u^{n,k}_\ell)^d))^{-d} \\
 \times\Biggl(\prod_{p=1}^{N-1} (\exp^{(p)}((u^{n,k}_\ell)^d))^{-1}\Biggr) (u^{n,k}_\ell)^{-\frac 2d \ga}\end{multlined}\\
&\quad\geq\mathtoolsset{multlined-width=0.9\displaywidth} \begin{multlined}[t] Ce^{(\theta-1)(d-1)n}\sum_{\ell=1}^{\ell^n_k}\biggl(\frac{\exp^{(N)}((u^{n,k}_{\ell-1})^d)}{1+\exp^{(N)}((u^{n,k}_{\ell-1})^d)}\biggr)^{d-1}\Biggl(\prod_{p=1}^N  \frac{\exp^{(p)}((u^{n,k}_{\ell-1})^d)}{1+\exp^{(p)}((u^{n,k}_{\ell-1})^d)}\Biggr)\\
 \times(u^{n,k}_{\ell-1})^{d-1}(u^{n,k}_\ell)^{-\frac 2d \ga}(u^{n,k}_\ell-u^{n,k}_{\ell-1})\end{multlined}\\
&\quad\geq Ce^{(\theta-1)(d-1)n} e^{n(d-1)-\frac 2d \ga n} (u^{n,k}_{\ell^n_k}-u^{n,k}_0)\geq e^{(\theta-1)(d-1)n} e^{n(d-1)-\frac 2d \ga n} e^{\theta n}=e^{(\theta d-\frac 2d\ga)n}.
\end{align*}
In summary,
\[ \P\Biggl( \sup_{x\in Q(x^n_k,e^{\theta n})} \frac{Y(t,\exp^{(N)}(x^d))}{ \lvert\exp^{(N)}(x^d)\rvert^{\frac{d^2}2}(\prod_{p=1}^{N-1} \lvert\exp^{(p)}(x^d)\rvert^{\frac d2})|x|^{\ga}} \leq K\Biggr)\leq \exp(-Ce^{(\theta d-\frac 2d\ga)n}).  \]
This bound is analogous to \eqref{eq:help5} in the proof of Theorem~\ref{thm:dim}. The subsequent arguments apply in our current situation as well and complete the proof of Step 2.  
\epr

By contrast, under a multiplicative noise, if we consider the peaks of the solution to \eqref{eq:SHE} on a lattice or if we consider the peaks on $\R^d$ and the noise is sufficiently heavy-tailed, they are \emph{not} self-similar in terms of their multifractal behavior. Given the tail estimates of Section~\ref{sec:tail}, the proof is very similar to that of the previous theorem (with $N=1$), which is why we omit it.

\bthm\label{thm:dim3} Let $M\in[0,\infty)$.
\benu 
\item[(i)] Assume Condition~\ref{cond:sup} and  Condition~\ref{cond:ht} with some $\al\in(0,\frac2d)$. 
Define the sets
\beq\label{eq:EMb}\calf^{(\times, \mathrm{c})}_M = \Bigl\{ x\in\R^d: Y(t,x)\geq \lvert x\rvert^{\frac d\al}e^{M(\log \lvert x \rvert)^{1/(1+\theta_\al)}} \Bigr\}.  \eeq
If $L_0$ is the number from \eqref{eq:L} and $M_1$ is the number from \eqref{eq:M1}, then almost surely,
\beq\label{eq:EM1b}\begin{split} \Dim_{\HH|\MM}\Bigl(\exp^{(1)}((\log^{(1)}(\calf^{(\times, \mathrm{c})}_M))^{\frac{1}{1+\theta_\al}})\Bigr)&\leq L_0(\tfrac{d}\al)^{\frac1{1+\theta_\al}}-\al M,\\
 \Dim_{\HH|\MM}\Bigl(\exp^{(1)}((\log^{(1)}(\calf^{(\times, \mathrm{c})}_M))^{\frac{1}{1+\theta_\al}})\Bigr)&\geq M_1(\tfrac{d}\al)^{\frac1{1+\theta_\al}}-\al M. \end{split}\eeq
\item[(ii)] Assume   Condition~\ref{cond:ht} with some $\al\in(0,1+\frac2d)$. 
Define the sets
\beq\label{eq:EM}\calf^{(\times, \mathrm{d})}_M = \Bigl\{ x\in\Z^d: Y(t,x)\geq \lvert x\rvert^{\frac d\al}e^{M(\log \lvert x \rvert)^{1/(1+\theta_\al)}} \Bigr\}.  \eeq
If $M_1$ and $M_2$ are the numbers from \eqref{eq:M1}, then almost surely,
\beq\label{eq:EM1}\begin{split} \Dim_{\HH|\MM}\Bigl(\exp^{(1)}((\log^{(1)}(\calf^{(\times, \mathrm{d})}_M))^{\frac{1}{1+\theta_\al}})\Bigr)&\leq M_2(\tfrac{d}\al)^{\frac1{1+\theta_\al}}-\al M,\\
 \Dim_{\HH|\MM}\Bigl(\exp^{(1)}((\log^{(1)}(\calf^{(\times, \mathrm{d})}_M))^{\frac{1}{1+\theta_\al}})\Bigr)&\geq M_1(\tfrac{d}\al)^{\frac1{1+\theta_\al}}-\al M. \end{split}\eeq
\item[(iii)] Assume   Condition~\ref{cond:lt} with   $\al=1+\frac2d$. 
Define the sets
\beq\label{eq:EMc}\calf^{(\times, \mathrm{d})}_M = \Bigl\{ x\in\Z^d: Y(t,x)\geq \lvert x\rvert^{\frac {d^2}{2+d}}e^{M(\log \lvert x \rvert)(\log\log\log \lvert x \rvert)/\log\log \lvert x \rvert} \Bigr\}  \eeq
and the function $H:\R^d\to\R^d$, $H(x)=\exp^{(1)}(\log^{(1)}(x)\log^{(3)}(x)/\log^{(2)}(x))$.
If $M_1$ and $M_2$ are the numbers from \eqref{eq:M2}, then almost surely,
\beq\label{eq:EM1c}\begin{split} \Dim_{\HH|\MM}\Bigl(H(\calf^{(\times, \mathrm{d})}_M) )\Bigr)&\leq M_2\tfrac{d^2}{2+d}-(1+\tfrac 2d) M,\\
 \Dim_{\HH|\MM}\Bigl(H(\calf^{(\times, \mathrm{d})}_M) )\Bigr)&\geq M_1\tfrac{d^2}{2+d}-(1+\tfrac 2d) M.  \end{split}\eeq
\eenu
\ethm

\begin{appendix}
	\section*{Technical results}\label{appn}
	\setcounter{theorem}{0}
		\setcounter{equation}{0}
In this appendix, we state and prove some technical results.
\blem\label{lem:Stirling}
For every $\al,\beta,\ga>0$, there is $C_{\al, \gamma}\in(0,\infty)$ such that for all 
$z\geq0$,
$$ \sum_{N=0}^\infty \frac{z^N}{\Ga(\al N+\beta)^{ 1/\ga}} \leq \frac \ga\al C_{\al,\ga} e^{C_{\al,\ga} z^{\ga/\al}}. $$
One can choose $C_{\al,\ga} $ such that it is locally bounded  in $\al$ and $1/\ga$ and independent of $\beta$.
\elem
\bpr In this proof, we use $C_{\al,\gamma}$ to denote a positive constant that is locally 
bounded in $\al$ and $\gamma$, and whose value may change from line to line. Let $z_0$ be 
the unique minimum of the gamma function on the positive real line. Then $\Ga(\al 
N+\beta)\geq \Ga((\al N+\beta)\vee z_0)\geq \Ga(\al N \vee z_0)$, so by Stirling's formula 
for gamma functions, there is $C\in(0,\infty)$ such that 
\begin{align*}\sum_{N=0}^\infty \frac{z^N}{\Ga(\al N+\beta)^{1/\ga}}&\leq C \sum_{N=0}^\infty(\al N\vee z_0)^{\frac1{2\ga}}  \frac {(e^{\al/\ga} z)^N}{(\al N)^{\al N/\ga}}\\
& \leq C \sum_{N=0}^\infty({\al N\vee z_0})^{\frac1{2\ga}} \frac {(e^{ \al/\ga} z)^Ne^{N/(\ga e)}}{N^{\al N/\ga}} \leq C_{\al,\ga} \sum_{N=0}^\infty \frac{(C_{\al,\ga}  z)^N}{N^{\al N/\ga}},\end{align*}
where we used the bound $\al^{\al N/\ga}\geq e^{-N/(\ga e)}$ for the second step. The function 
$x\mapsto (C_{\al,\ga}  z)^x/x^{\al x/\ga}$ has a unique maximum at $x=(C_{\al,\ga}   z)^{\ga/\al}e^{-1}$. Thus, by integral approximation, a change of variable ($y=\al x/\ga$) and a Riemann sum approximation,
\begin{align*}&\sum_{N=0}^\infty \frac{z^N}{\Ga(\al N+\beta)^{1/\ga}}\leq C_{\al,\ga} \Biggl( \int_0^\infty \frac{(C_{\al,\ga}   z)^x}{x^{\al x/\ga}} \,\dd x + e^{\al (C_{\al,\ga}  z)^{\ga/\al}/(\ga e)}\Biggr)\\
&\qquad\leq C_{\al,\ga} e^{C_{\al,\ga}  z^{\ga/\al}} + C_{\al,\ga} \int_0^\infty\frac{(\al/\ga)^{y-1}(C_{\al,\ga}   z)^{\ga y/\al }}{y^y}\,\dd y\\
&\qquad\leq C_{\al,\ga} e^{C_{\al,\ga}  z^{\ga/\al}} + C_{\al,\ga} \Biggl( \sum_{N=1}^\infty \frac{(\al/\ga)^{N-1}(C_{\al,\ga}  z)^{\ga N/\al}}{N^N} +\frac{\ga}{\al}e^{\al (C_{\al,\ga}  z)^{\ga/\al}/(\ga e)} \Biggr)\\& \qquad\leq \frac\ga\al C_{\al,\ga}  e^{C_{\al,\ga}   z^{\ga/\al}}.\qedhere \end{align*}
\epr

\blem\label{lem:PZ} For $\al,\delta\in(0,1)$, $p>1$ and a positive random variable $X$ with $0<\E[X^p]<\infty$, we have
$$ \P(X>\delta\,\E[X])\geq (1-\delta)^{\frac p{p-1}}\frac{\E[X]^{\frac p{p-1}}}{\E[X^p]^{\frac 1{p-1}}},\qquad \E[X^\al]\geq 2^{-\al-\frac p{p-1}}\frac{\E[X]^{\al+\frac p{p-1}}}{\E[X^p]^{\frac1{p-1}}}. $$
\elem
\bpr Both inequalities are variants of the classical Paley--Zygmund inequality. The first 
one was proved in \cite[Lemma~7.3]{Khos} (the assumption $p\geq2$ in the mentioned 
reference was not needed in the proof). The second follows from the first by Markov's
inequality.
\epr

\begin{lemma} \label{lemma:iter-int}
	For $R > 1$, $\alpha > -1, \beta > -1$,
	\[
	\begin{split}
		H_{N;\alpha, \beta}(R) & = \int_0^R \dotsm \int_0^R (y_1 \dotsm y_N)^{\alpha}
		\biggl( \log \frac{1}{y_1 \ldots y_N} \biggr)^{\beta} \bone_{\{ y_1 \dotsm y_N \leq 1 \}}\,
		\dd y_1 \cdots \dd y_N\\ &= \sum_{i=0}^{N-1} c_{N,i} (\log R)^i,
	\end{split}
	\]
	where
	\[
	c_{N,i} = \frac{N^i \Gamma(N - i + \beta)}{i! (N - i -1)! (\alpha+1)^{N-i+\beta}}.
	\]
\end{lemma}

\begin{proof}
	To ease notation, we suppress the subscripts $\alpha$ and $\beta$.
	Changing variables $u_i = y_N^{1/(N-1)} \, y_i$ for $i = 1,\ldots, N-1$, we obtain
	\begin{equation} \label{eq:H-recursion}
		\begin{split}
			H_N (R ) & = \mathtoolsset{multlined-width=0.75\displaywidth} \begin{multlined}[t] \int_0^R \biggl(  \int_{[0, Ry_N^{1/(N-1)}]^{N-1}}
			(u_1 \dotsm u_{N-1})^\alpha \\
			 \times
			\biggl( \log \frac{1}{u_1 \dotsm u_{N-1}} \biggr)^{\beta} \bone_{ \{ u_1 \dotsm u_{N-1} \leq 1 \}} y_{N}^{-1}
			\,\dd u_1 \cdots \,\dd u_{N-1} \biggr) \,\dd y_N \end{multlined}\\
			& = \int_0^R \frac{1}{y} H_{N-1}(Ry^{1/(N-1)}) \,\dd y
		\end{split}
	\end{equation}
	for all $R>0$.
	
	We prove the lemma by  induction. For $N = 1$, the statement is clear. Assume that the 
	statement holds for $N \geq 1$. Since
	\[
	\int_1^{R^{N+1}} \frac{1}{u}  ( \log u  )^i\,\dd u =
	\frac{(N+1)^{i+1}}{i+1}  ( \log R  )^{i+1}
	\]
	for $i \geq 0$, we can use \eqref{eq:H-recursion}, a change of variables $u = y R^N$ and the induction hypothesis to obtain
	\[
	\begin{split}
		H_{N+1}(R) &  = \int_0^R \frac{1}{y} H_{N}(Ry^{\frac1N}) \,\dd y = \int_0^1  \frac{1}{u} H_N(u^{\frac1N})\, \dd u + 
		\sum_{i=0}^{N-1} c_{N,i} \int_1^{R^{N+1}} \frac{1}{u}  ( \log u^{\frac1N}  )^{i}\,
		\dd u \\
		& = \int_0^1  \frac{1}{u} H_N(u^{\frac1N})\, \dd u + \sum_{i=0}^{N-1} c_{N,i} \frac{(N+1)^{i+1}}{N^i (i+1)}
		( \log R  )^{i+1}\\
		&=\int_0^1  \frac{1}{u} H_N(u^{\frac1N})\, \dd u + \sum_{i=0}^{N-1} c_{N+1,i+1}  
		( \log R  )^{i+1}.
	\end{split}
	\]
	Thus, it remains to show
	\beq\label{eq:toshow}
	c_{N+1, 0} = \int_0^1 \frac{1}{u} H_N(u^{\frac1N})\, \dd u.
	\eeq
	We claim that
	\begin{equation} \label{eq:c0-recursion}
		c_{N+1,0} = \int_0^1 \frac{ ( \log u^{-1}  )^j}{j!\, u} 
		H_{N-j} ( u^{1/(N-j)})\, \dd u,\qquad j=0,\ldots, N-1.
	\end{equation}
	When $j=0$, this becomes \eqref{eq:toshow}. Noting that $\frac{\dd}{\dd u} H_k(u) = ku^{-1}H_{k-1}(u^{1+1/(k-1)})$ for $u\in(0,1)$, one can show \eqref{eq:c0-recursion} using integration by parts and a backwards
	induction argument. Thus, it remains to verify \eqref{eq:c0-recursion} at the base case  $j = N-1$:
	\[
	\begin{split}
		\int_0^1 \frac{ ( \log u^{-1}  )^{N-1}}{(N-1)!\,u} 
		H_{1} ( u) \,\dd u  &= \int_0^1 \frac{ ( \log u^{-1}  )^{N-1}}{(N-1)!\, u} 
		\int_0^u y^\alpha  ( \log y^{-1}  )^\beta \,\dd y \,  \dd u \\
		&=  \int_0^1 \frac{ ( \log u^{-1}  )^{N}}{N!} 
		u^\alpha  ( \log u^{-1}  )^\beta  \,  \dd u \\
		&= \frac{\Gamma(N+1+\beta )}{N!\, (\alpha+1)^{N+1+\beta}}=c_{N+1,0}.\qedhere
	\end{split}
	\]
\end{proof}

\begin{lemma}\label{lem:decom}
For any $t>0$, 
the Poisson random measure $\mu$ can be decomposed into $\mu=\sum_{i=0}^\infty \mu_i$ such that 
\bit
\item the $\mu_i$'s are independent Poisson random measures,
\item $\mu_0$ is the restriction of $\mu$ to $[0,\infty)\times(\R^d\setminus (-2,2)^d)\times(0,\infty)$,
\item $\mu_i$ has intensity $\dd t \,\bone_{(-2,2)^d}(x)\,\dd x\,\la_i(\dd z)$ and  $m_0(\la_i)\leq 2$.
\eit
\end{lemma}
\bpr We first construct a decomposition $\la=\sum_{i=1}^\infty \la_i$ into pieces satisfying $m_0(\la_i)\leq2$ for all $i$, assuming that $\la((0,\infty))=\infty$ (if $\la((0,\infty))<\infty$, the construction is similar, with all but finitely many $\la_i$'s equal to $0$).
Define $z_0=\infty$ and $z_\nu=\sup\{z>0: \ov\la (z) > \nu\}$ for 
$\nu\in\N$, where $\overline \lambda(z)=\la((z,\infty))$. Clearly,
 $\overline \lambda(z_n) \leq n \leq \overline \lambda(z_n - )$.
Let $\lambda_1 = \lambda|_{(z_1, z_0)}$, and   assume that $z_{k+1} < z_k$ and that
$\lambda_{k+1}$ has already been  defined for some $k \geq 0$. If $z_{k+2} < z_{k+1}$, put
$\lambda_{k+2} = \lambda|_{(z_{k+2}, z_{k+1}]}$. Since $z_{k+1} < z_k$ implies 
$\overline \lambda (z_{k+1}) \geq k$, we have that 
\[
\lambda_{k+2}((0,\infty)) = \lambda( (z_{k+2}, z_{k+1}]) =
\overline \lambda(z_{k+2}) - \overline \lambda (z_{k+1}) \leq 2.
\]
If $z_{k+1} = z_{k+2}$, then let $\ell \geq 2$ be the number for which    
$z_{k+1} = z_{k+\ell} > z_{k+\ell +1}$. Then let 
\[
 \lambda_{k+2}  = \dots = \lambda_{k+\ell} = \delta_{z_{k+1}}, \qquad
 \lambda_{k+\ell + 1} = \lambda|_{(z_{k+\ell + 1}, z_{k+1})} + 
( \lambda( \{ z_{k+1} \}) - (\ell - 1) ) \delta_{z_{k+1}},
\]
where $\delta_x$ stands for the Dirac delta at $x$.
Note that $\ell-1\leq\lambda( \{ z_{k+1} \} ) \leq \ell $, so $\lambda_{k+\ell +1}$
is a positive measure. Furthermore, we have
\[
\lambda( (z_{k+\ell + 1}, z_{k+1} ) ) +\lambda( \{ z_{k+1} \} )  \leq \overline \lambda( z_{k+\ell + 1}) - 
\overline \lambda (z_{k+1} ) \leq k + \ell +1 - k = \ell + 1,
\]
which implies that $\lambda_{k+\ell +1}((0,\infty)) \leq 2$. This completes the construction of the decomposition $\la=\sum_{i=1}^\infty \la_i$. The $\mu_i$'s can now be obtained by restrictions and thinnings of $\mu$ (see \cite[Sect.~5]{PoiBook}).
\epr

\blem\label{lem:tech} For $m,N\in\N$ and $\beta>0$, let
\begin{align*}
Y^{(0)}(t,x)&=Y^{(0,m,\beta)}(t,x)=Y_<(t,x),\\
 Y^{(0,\beta)}_<(t,x)&=Y^{(0,\beta)}(t,x)=1,\quad u^{(0,\beta)}_<(s,y;t,x)=g(t-s,x-y)\end{align*}
and introduce  the following processes inductively:
\begin{align*}
	Y^{(N)}(t,x)&=Y_<(t,x)+\int_0^t\int_{\R^d} u_<(s,y;t,x)\bone_{\{\lvert x-y\rvert\leq \sqrt{t-s}\}} Y^{(N-1)}(s,y)\,\La_\geq(\dd s,\dd y),\\
	Y^{(m,\beta)}_<(t,x)&=1+\int_0^t\int_{\R^d} g(t-s,x-y)\bone_{\{\lvert x-y\rvert\leq \beta\sqrt{t-s}\}}Y^{(m-1,\beta)}_<(s,y)\,\La_<(\dd s,\dd y),\\		
	u^{(m,\beta)}_<(s,y;t,x)&=\mathtoolsset{multlined-width=0.75\displaywidth} \begin{multlined}[t] g(t-s,x-y)+\int_s^t\int_{\R^d} g(t-r,x-w)\\
		\times\bone_{\{\lvert x-w\rvert\leq \beta\sqrt{t-r}\}} u^{(m-1,\beta)}_<(s,y;r,w)\,\La_<(\dd r,\dd w),\end{multlined}\\
		Y^{(N,m,\beta)}(t,x)&=\mathtoolsset{multlined-width=0.75\displaywidth} \begin{multlined}[t] Y^{(m,\beta)}_<(t,x)+\int_0^t\int_{\R^d} u^{(m,\beta)}_<(s,y;t,x)\\
		\times\bone_{\{\lvert x-y\rvert\leq \sqrt{t-s}\}} Y^{(N-1,m,\beta)}(s,y)\,\La_\geq(\dd s,\dd y),\end{multlined}\\
Y^{(m,\beta)}(t,x)&=\mathtoolsset{multlined-width=0.75\displaywidth} \begin{multlined}[t] Y^{(m,\beta)}_<(t,x)+\int_0^t\int_{\R^d} u^{(m,\beta)}_<(s,y;t,x)\\
\times\bone_{\{\lvert x-y\rvert\leq \beta\sqrt{t-s}\}} Y^{(m-1,\beta)}(s,y)\,\La_\geq(\dd s,\dd y).\end{multlined}
\end{align*}
For every $t>0$, there is a constant $C\in(0,\infty)$ such that for all $p\in(1,1+\frac2d)$, $s\in(0,t)$, $x,y\in\R^d$, $\beta>0$, $N\in\N$ and $m\in\N$ with $\theta_pm>1$,
\beq\label{eq:tech}\begin{split}
&\E[\lvert 	Y_<(t,x)-Y_<^{(m,\beta)}(t,x)\rvert^p]^{\frac1p} \leq C \Bigl(e^{-C^{-1}\beta}+(C\theta_p^{-\frac 2p})^m m^{-\frac{\theta_p}{3p} m}\Bigr)e^{(C/ \theta_p)^{3/\theta_p}},\\
	&\E[\lvert u_<(s,y;t,x)-u_<^{(m,\beta)}(s,y;t,x)\rvert^p]^{\frac1p}\\&\qquad\leq C \Bigl(e^{-C^{-1}\beta}+(C\theta_p^{-\frac 2p})^m m^{-\frac{\theta_p}{3p} m}\Bigr)e^{(C/ \theta_p)^{3/\theta_p}}g(t-s,x-y). 
\end{split} \eeq
One can further choose $C$ in such a way that 
 \beq\label{eq:tech2}\begin{split}
\E[\lvert 	Y(t,x)-Y^{(m,\beta)}(t,x)\rvert^p]^{\frac1p} &\leq C \Bigl(e^{-C^{-1}\beta}+(C\theta_p^{-\frac 2p})^m m^{-\frac{\theta_p}{3p} m}\Bigr)e^{(CM_p(\la)/ \theta_p)^{3/\theta_p}},\\
	\E[\lvert Y^{(N)}(t,x)-Y^{(N,m,\beta)}(t,x)\rvert^p]^{\frac1p}&\leq C \Bigl(e^{-C^{-1}\beta}+(C\theta_p^{-\frac 2p})^m m^{-\frac{\theta_p}{3p} m}\Bigr)e^{(CM_p(\la)/ \theta_p)^{3/\theta_p}}. 
\end{split}\eeq
\elem
\bpr We start with the first inequality in \eqref{eq:tech}. Let $Y^{(m)}_<(t,x)=Y^{(m,\infty)}_<(t,x)$. Because 
\begin{multline*} \E[\lvert 	Y_<(t,x)-Y_<^{(m,\beta)}(t,x)\rvert^p]^{\frac1p}\leq \E[\lvert 	Y_<(t,x)-Y_<^{(m)}(t,x)\rvert^p]^{\frac1p}\\
	+\E[\lvert 	Y^{(m)}_<(t,x)-Y_<^{(m,\beta)}(t,x)\rvert^p]^{\frac1p}, \end{multline*}
we can bound the two terms on the right-hand side separately. Upon noticing that $Y^{(m)}_<(t,x)$ is, in fact, the sum of the first $m+1$ terms in the chaos expansion of $Y_<(t,x)$, we infer from \eqref{eq:ubar} that for $d\geq2$ and $p\in(1,1+\frac 2d)$, 
\beq\label{eq:est}\begin{split}
	&\E[\lvert 	Y_<(t,x)-Y_<^{(m)}(t,x)\rvert^p]^{\frac1p}\\
	&\qquad\leq C\Ga(\theta_p)^{\frac1p}\sum_{k=m+1}^\infty \frac{(C\Ga(\frac {\theta_p}3))^{k/p}}{\Ga(\frac{\theta_p}3 k +\theta_p)^{1/p}}
	\leq C \theta_p^{-\frac2p}\frac{(C\Ga(\frac{\theta_p}3))^{m/p}}{\Ga(\frac{\theta_p}3 m)^{1/p}}\sum_{k=1}^\infty \frac{(C\Ga(\frac {\theta_p}3))^{k/p}}{\Ga(\frac{\theta_p}3 k +\theta_p)^{1/p}}\\
	&\qquad\leq C \theta_p^{-1-\frac2p}\frac{(C\Ga(\frac{\theta_p}3))^{m/p}}{\Ga(\frac{\theta_p}3 m)^{1/p}} e^{C(\Ga(\frac {\theta_p}3))^{3/\theta_p}}\leq  (C\theta_p^{-\frac 2p})^mm^{-\frac{\theta_p}{3p}m}e^{(C/ \theta_p)^{3/\theta_p}},
\end{split}\eeq
where we used Lemma~\ref{lem:Stirling}, 
Stirling's formula for gamma functions and  the property $\Ga(x)\sim x^{-1}$ as $x\to0$. By \eqref{eq:ubar2} and \eqref{eq:mom-p}, the last bound remains true if $d=1$. 

Next, observe that 
\begin{align*}
	&Y^{(m)}_<(t,x)-Y^{(m,\beta)}_<(t,x)=\int_0^t\int_{\R^d} g(t-s,x-y)\bone_{\{\lvert x-y\rvert >\beta\sqrt{t-s}\}} Y^{(m-1)}_<(s,y)\,\La_<(\dd s,\dd y)\\
	&\qquad+ \int_0^t\int_{\R^d} g(t-s,x-y)\bone_{\{\lvert x-y\rvert\leq \beta\sqrt{t-s}\}}(Y^{(m-1)}_<(s,y)-Y^{(m-1,\beta)}_<(s,y))\,\La_<(\dd s,\dd y).
\end{align*}
Iterating this $m$ times, denoting $(t,x) = (t_{k+1}, x_{k+1})$, 
we derive the identity
\begin{align*}
	Y^{(m)}_<(t,x)-Y^{(m,\beta)}_<(t,x)&=\sum_{k=1}^m \int_{((0,t)\times\R^d)^k} \Biggl(\prod_{i= 3}^{k+1} g(\Delta t_i,\Delta x_i)\bone_{\{\lvert \Delta x_i\rvert\leq \beta\sqrt{\Delta t_i}\}}\Biggr)\\
&	\times g(\Delta t_2,\Delta x_2)\bone_{\{\lvert \Delta x_2\rvert >\beta\sqrt{\Delta t_2}\}}Y^{(m-k)}_<(t_1,x_1)\prod_{j= 1}^k \La_<(\dd t_j,\dd x_j).
\end{align*}
Representing $Y^{(m-k)}_<(t_1,x_1)$ itself in a series, we obtain
\beq\label{eq:Ym}
	Y^{(m)}_<(t,x)-Y^{(m,\beta)}_<(t,x)=\sum_{k=1}^{m}\sum_{\ell=1}^{m-k+1} \int_{((0,t)\times\R^d)^m} \prod_{i=2}^{k+\ell} g^{(\beta)}_{i,k,\ell}(\Delta t_i,\Delta x_i)
	\prod_{j=1}^m \La_<(\dd t_j,\dd x_j),
\eeq
where $g^{(\beta)}_{i,k,\ell}(t,x)=g(t,x)$ if $i=2,\ldots, \ell$, 
$g^{(\beta)}_{i,k,\ell}(t,x)=g(t,x)\bone_{\{\lvert x\rvert >\beta\sqrt{t}\}}$ if 
$i=\ell+1$ and $g^{(\beta)}_{i,k,\ell}(t,x)=g(t,x)\bone_{\{\lvert x\rvert 
\leq\beta\sqrt{t}\}}$ if $i= \ell+2,\ldots,k+\ell$. An important observation is now that 
the moment bounds on $Y_<(t,x)$ obtained in \cite[Proposition 6.1]{Berger21b} or through 
the series of arguments leading to \eqref{eq:mom-p} (if $d=1$ and $p\in(2,3)$) are, first 
of all, obtained by estimating each term in a series expansion of $Y_<(t,x)$ separately 
and, second of all, can only increase if the kernels $g^{(\beta)}_{i,k,\ell}$ are replaced 
by something larger. Therefore, bounding $g^{(\beta)}_{i,k,\ell}(t,x)\leq g(t,x)\leq 
Cg(2t,x)$ if $i\neq \ell+1$ and 
$$ g^{(\beta)}_{i,k,\ell}(t,x)\leq(2\pi t)^{-\frac d2} e^{-\frac{\lvert x\rvert^2}{4 t}}e^{-\frac{\beta}{4 t}} \leq Ce^{-C^{-1}\beta}g(2t,x)$$
if $i=\ell+1$, we conclude that
\begin{align*}
	\E[\lvert Y^{(m)}_<(t,x)-Y^{(m,\beta)}_<(t,x)\rvert^p]\leq Ce^{-C^{-1}\beta}e^{(C/\theta_p)^{3/\theta_p}}.
\end{align*}
Together with \eqref{eq:est}, this 
shows the first inequality in \eqref{eq:tech}; the proof of the second inequality in \eqref{eq:tech} and the proof of \eqref{eq:tech2} are similar and therefore skipped. 
%
\epr

\blem\label{lem:Y0Z}
Let $d=1$, $\beta>0$, $m\in\N$ and 
consider the processes $Y_0$ and $Z^{(m,\beta)}$ defined in \eqref{eq:Y0} and \eqref{eq:Zmb}, respectively. Assume Condition~\ref{cond:sup} and Condition~\ref{cond:lt} with $\al=2$. Then there exists a constant $C>0$ such that for any interval $I$ of length $1$ and $R>1$, we have that 
$$ \P\Biggl(\sup_{x\in I} \lvert Y_0(t,x)-Z^{(m,\beta)}(t,x)\rvert> R\Biggr)\leq C_{m,\beta}R^{-2}\log R,$$
where $C_{m,\beta}=C(e^{-C^{-1}\beta}+C^m m^{-m/12})$. If $M_{\al}(\la)<\infty$ for some $\al>2$, the factor $\log R$  can be omitted.
\elem
\bpr
The proof is very similar to how we dealt with $Y_1$ in the proof of Theorem~\ref{thm:sup-lt}. In fact, we only need to estimate the last integral in \eqref{eq:aux} with   $\wt Y'(s,y)$ replaced by (a copy of) $\lvert Y(s,y)-Y^{(m,\beta)}(s,y)\rvert$. By 
Markov's inequality and Lemma~\ref{lem:tech}, we have that
\begin{align*}
\P((t-s)^{-\frac 12}\lvert Y(s,y)-Y^{(m,\beta)}(s,y)\rvert z>\tfrac RN) &\leq C_{m,\beta}R^{-2}N^2z^2(t-s)^{-1}\wedge1\\
&\leq C_{m,\beta}N^2(R^{-2}z^2(t-s)^{-1}\wedge 1).
\end{align*}
Therefore,
\begin{align*}
&\int_0^t\int_{(-2,2)}\int_{(0,\infty)} \P((t-s)^{-\frac 12}\lvert Y(s,y)-Y^{(m,\beta)}(s,y)\rvert z>\tfrac RN)\,\dd s\,\dd y\,\la(\dd z)\\
&\qquad\leq 4 N^2 C_{m,\beta}\int_{(0,\infty)} \biggl(R^{-2}z^2\int_{R^{-2}z^2\wedge t}^t 
s^{-1}\,\dd s+  R^{-2}z^2\biggr)\,\la(\dd z)\\
&\qquad \leq 4 N^2 C_{m,\beta}(C\mu_2(\la)R^{-2}\log R + 2m_2^{\log}(\la) R^{-2}),
\end{align*}
which yields the desired bound since $m_2^{\log}(\la)<\infty$ by Condition~\ref{cond:sup}. If $M_\al(\la)<\infty$ for some $\al>2$, we can get rid of the logarithmic factor by using power $\al$ in Markov's inequality above.
\epr


\blem\label{lem:Lipschitz} Let $E\subseteq\R^d$ and $f\colon E\to \R^p$ be a Lipschitz continuous function such that 
\beq\label{eq:growth}\eps=\liminf_{x\in E,\lvert x\rvert_\infty \to\infty} \frac{\lvert f(x)\rvert_\infty}{\lvert x\rvert_\infty} >0.\eeq
Then $\Dim_\HH(f(E))\leq \Dim_\HH(E)$ and $\Dim_\MM(f(E))\leq \Dim_\MM(E)$.
\elem
\bpr The statement for the Hausdorff dimension is exactly \cite[Lemma~2.4]{Khoshnevisan17}. In order to obtain the statement concerning the Minkowski dimension, we notice that 
\[ \Dim_\MM (E) = \limsup_{n\to\infty} \frac1n \log_+\bigl\lvert \{q\in\Z^d: E\cap Q(0,e^n)\cap Q(q,1)\neq \emptyset\} \bigr\rvert, \]
which can be easily deduced from \cite[Prop.~2.5]{Khoshnevisan17b}. 
Let $A_E(n)$ be the set whose cardinality is counted in the previous line. Then $E\cap Q(0,e^m)\subseteq \bigcup_{q\in A_E(m)} Q(q,1)$ by definition and hence, $f(E\cap Q(0,e^m))\subseteq \bigcup_{q\in A_E(m)} f(Q(q,1))$ for every $m\in\N$.

If $C\in\N$ is larger than $\log \eps^{-1}$ and $n$ is large enough, then $f(x) \in Q(0,e^n)$  for some $x\in E$ implies $x\in Q(0,e^{n+C})$ by the growth assumption on $f$. Therefore, $f(E)\cap Q(0,e^n)\subseteq f(E\cap Q(0,e^{n+C}))\subseteq \bigcup_{q\in A_E(n+C)} f(Q(q,1))$. If $L$ is the Lipschitz constant of $f$ with respect to the supremum norm, then $f(Q(q,1))$ has at most diameter $L$ (in the same norm) and can therefore be covered by $L^p$ unit cubes, or $(L+1)^p$ unit cubes with integer corners. In total, we need at most $\lvert A_E(n+C)\rvert (L+1)^p$ such cubes to cover $f(E)\cap Q(0,e^n)$. Thus,
\begin{align*} \Dim_\MM(f(E))&=\limsup_{n\to\infty} \frac1n \log_+\lvert A_{f(E)}(n)\rvert \leq \limsup_{n\to\infty} \frac1n \log_+(\lvert A_E(n+C)\rvert (L+1)^p)\\ 
&= \limsup_{n\to\infty} \frac1n \log_+\lvert A_E(n+C)\rvert  = \Dim_\MM(E).\qedhere \end{align*}
\epr

\blem\label{lem:Lip} Let $\lvert \cdot\rvert$ and $\lVert \cdot\rVert$ be norms on $\R^d$ and $\exp^{(n)}(r)=\exp(\exp^{(n-1)}(r))$  for $n\in\N$ and $r\in\R$ (with $\exp^{(0)}(r)=r$). Then, for every $N\in\N$, the function 
\beq\label{eq:f} f\colon \R^d\setminus\{0\} \to \R^d,\qquad x\mapsto \frac{x}{\lvert x\rvert} \biggl(\log^{(N)}\biggl( \frac{\lvert x\rvert}{\lVert x\rVert} \exp^{(N)}(\lVert x\rVert^d)\biggr)\biggr)^{1/d},\eeq
satisfies \eqref{eq:growth} and
is  Lipschitz continuous on $\{x\in\R^d: \lvert x\rvert \geq s\}$ for some $s>0$. 
\elem
\bpr Because all norms are equivalent on $\R^d$, we have $C^{-1}\leq\lvert x\rvert/\lVert x\rVert\leq C$ for some $C>1$. 
Consider the mapping $h(r,s)=(\log^{(N)}(r\exp^{(N)}(s^d)))^{1/d}$ for $r\in(C^{-1},C)$ and $s>0$. For sufficiently large $s$ (so that $C^{-1}\exp^{(N)}(s^d)>e$), its partial derivatives are given by
\begin{align*}
\frac{\partial}{\partial r}h(r,s)&=\frac{d^{-1}(\log^{(N)}(r\exp^{(N)}(s^d)))^{1/d-1}}{r\prod_{p=1}^{N-1} \log^{(p)}(r\exp^{(N)}(s^d))},\\ \frac{\partial}{\partial s}h(r,s)&=\frac{s^{d-1}(\log^{(N)}(r\exp^{(N)}(s^d)))^{1/d-1}\prod_{p=1}^{N-1} \exp^{(p)}(s^d)}{\prod_{p=1}^{N-1} \log^{(p)}(r\exp^{(N)}(s^d))}.
\end{align*}
By induction on $p$, one can easily verify that $\log^{(p)}(r\exp^{(N)}(s^d))\geq \frac12 \exp^{(N-p)}(s^d)$ as soon as $s$ is large enough so that $\frac12 s^d>\log 2\vee \log C$. This shows \eqref{eq:growth} on the one hand and that the partial derivatives of $h$ are uniformly bounded for $r\in(C^{-1},C)$ and large $s$ on the other hand.

Moreover, by elementary estimates,
\begin{align*} \biggl\lvert \frac{x}{\lvert x\rvert}-\frac{y}{\lvert y\rvert}\biggr\rvert &= \frac{\bigl\lvert(x-y)\lvert y\rvert +y(\lvert y\rvert -\lvert x\rvert)\bigr\rvert}{\lvert x\rvert \lvert y\rvert} \leq \frac{2\lvert x-y\rvert}{\lvert x\rvert}. \\
\biggl\lvert   \frac{\lvert x\rvert}{\lVert x\rVert} -\frac{\lvert y\rvert}{\lVert y\rVert}    \biggr\rvert &= \frac{\bigl\lvert (\lvert x\rvert-\lvert y \rvert) \lVert y\rVert + \lvert y\rvert (\lVert y\rVert - \lVert x \rVert)\bigr\rvert}{\lVert x \rVert \lVert y\rVert} \leq \frac{\lvert x -y\rvert}{\lVert x\rVert} + \frac{\lvert y\rvert}{\lVert x\rVert \lVert y\rVert}\lVert x-y\rVert.
\end{align*}
By writing
\begin{align*}  f(x)-f(y)&= \biggl( \frac{x}{\lvert x\rvert}-\frac{y}{\lvert y\rvert}\biggr) h\biggl(\frac{\lvert x\rvert}{\lVert x\rVert},\lVert x\rVert\biggr)  + \frac{y}{\lvert y\rvert}\Biggl(h\biggl(\frac{\lvert x\rvert}{\lVert x\rVert},\lVert x\rVert\biggr)-h\biggl(\frac{\lvert y\rvert}{\lVert y\rVert},\lVert y\rVert\biggr)\Biggr), \end{align*}
the Lipschitz property of $f$ now follows from the previous estimates and a straightforward application of the mean-value theorem to the second difference above.
\epr

\end{appendix}

\begin{acks}[Acknowledgments]
We are thankful to Davar Khoshnevisan for discussions on 
Hausdorff and Minkowski dimension.
PK's research was supported by the János Bolyai Research Scholarship of the 
Hungarian Academy of Sciences.
\end{acks}

\bibliographystyle{abbrv}
\bibliography{as-heat}
 
\end{document}